\pdfoutput=1

\documentclass[reqno]{amsart}
\usepackage{etex}
\usepackage[a4paper,hmarginratio=1:1]{geometry}
\usepackage{amssymb,amsfonts,amsmath}
\usepackage{comment} 
\usepackage[all,arc]{xy}
\usepackage{enumerate}
\usepackage{mathrsfs,mathtools}
\usepackage{todonotes,booktabs}
\usepackage{stmaryrd}
\usepackage{marvosym}
\usepackage{graphicx}
\usepackage{pdflscape} 
\usepackage{scalerel}
\usepackage{float}

\usepackage{bbm}
\usepackage{microtype}
\usepackage{dsfont}
\usepackage{enumitem}
\usepackage{tikz}
\usepackage{tikz-cd}
\usetikzlibrary{trees}
\usetikzlibrary[shapes]
\usetikzlibrary[arrows]
\usetikzlibrary{patterns}
\usetikzlibrary{fadings}
\usetikzlibrary{backgrounds}
\usetikzlibrary{decorations.pathreplacing}
\usetikzlibrary{decorations.pathmorphing}
\usetikzlibrary{positioning}
\usetikzlibrary{decorations.markings}

\def\biblio{\bibliography{duality}\bibliographystyle{alpha}}

\usepackage{xcolor} 
\usepackage{graphicx}
\SelectTips{cm}{10}

\usepackage[pagebackref]{hyperref}

\definecolor{dark-red}{rgb}{0.5,0.15,0.15}
\definecolor{dark-blue}{rgb}{0.15,0.15,0.6}
\definecolor{dark-green}{rgb}{0.15,0.6,0.15}
\hypersetup{
    colorlinks, linkcolor=dark-red,
    citecolor=dark-blue, urlcolor=dark-green
}

\renewcommand*{\backref}[1]{}
\renewcommand*{\backrefalt}[4]{%
  \ifcase #1 %
No citations.
  \or
(cit. on p. #2).%
  \else
(cit on pp. #2).%
  \fi%
}

\usepackage[nameinlink,capitalise,noabbrev]{cleveref}
 
\newtheorem{thm}{Theorem}[section]
\newtheorem{theorem}[thm]{Theorem}
\newtheorem{corollary}[thm]{Corollary}
\newtheorem{proposition}[thm]{Proposition}
\newtheorem{lemma}[thm]{Lemma}

\newtheorem{warning}[thm]{Warning}
\newtheorem{dictionary}[thm]{Dictionary}

\newtheorem*{theorem*}{Theorem}
\newtheorem*{conjecture*}{Conjecture}
\newtheorem*{convention*}{Convention}

\newtheorem{thmx}{Theorem}

\theoremstyle{definition}
\newtheorem{definition}[thm]{Definition}
\newtheorem{example}[thm]{Example}

\theoremstyle{remark}
\newtheorem{remark}[thm]{Remark}
\newtheorem{convention}[thm]{Convention}

\usepackage{stackengine}
\usepackage{scalerel}

\makeatletter
\let\c@equation\c@thm
\makeatother
\numberwithin{equation}{section}

\makeatletter
\@namedef{subjclassname@2020}{%
  \textup{2020} Mathematics Subject Classification}
\makeatother

\newcommand{\cpctRecollement}[3]{
\xymatrix@C=2em{{#1} \ar[r]|-{#2} & {#3}}
}

\newcommand{\recollement}[5]{
\xymatrix@C=4em{{#1} \ar@{<-}[r]|-{#2} & #3 \ar@{<-}[r]|-{#4} \ar@{<-}@<1.5ex>[l]^-{{#2}_!} \ar@{<-}@<-1.5ex>[l]_-{{#2}^*} & #5, \ar@{<-}@<1.5ex>[l]^-{{#4}!} \ar@{<-}@<-1.5ex>[l]_-{{#4}^*}
}}
\let\lim\relax

\DeclareMathOperator{\Cat}{Cat}
\DeclareMathOperator{\stableCat}{Cat_{\infty}^{ex}}
\DeclareMathOperator{\stable2Cat}{Cat_{(\infty,2)}^{ex}}
\DeclareMathOperator{\lim}{lim}
\newcommand{\rlaxlim}{\mathrm{rlim}^\mathrm{lax}}

\newcommand{\F}{\mathbb{F}}
\newcommand{\N}{\mathbb{N}}
\newcommand{\Q}{\mathbb{Q}}
\newcommand{\R}{\mathbb{R}}
\newcommand{\Z}{\mathbb{Z}}

\newcommand{\bT}{\mathbb{T}}

\DeclareMathOperator{\cA}{\mathcal{A}}

\DeclareMathOperator{\cD}{\mathcal{D}}
\DeclareMathOperator{\cE}{\mathcal{E}}
\DeclareMathOperator{\cF}{\mathcal{F}}

\newcommand{\cL}{\mathcal{L}}

\DeclareMathOperator{\cP}{\mathcal{P}}

\DeclareMathOperator{\fX}{\mathfrak{X}}

\newcommand{\SSi}{\Sigma} 
 \newcommand{\e}{{(e)}}

\newcommand{\Gt}{\widetilde{G}}

\newcommand{\mcD}{\mathcal{D}}

\DeclareMathOperator{\Top}{Top}
\DeclareMathOperator{\Spc}{Spc}
\DeclareMathOperator{\Hom}{Hom}

\DeclareMathOperator{\colim}{colim}

\DeclareMathOperator{\Spec}{Spec}
\DeclareMathOperator{\Mod}{Mod}

\DeclareMathOperator{\Loc}{Loc}

\DeclareMathOperator{\Fun}{Fun}

\DeclareMathOperator{\ad}{ad}

\DeclareMathOperator{\id}{id}

\DeclareMathOperator{\unit}{\mathbbm{1}}

\DeclareMathOperator{\free}{free}

\DeclareMathOperator{\cofree}{cofree}

\DeclareMathOperator{\supp}{supp}

\DeclareMathOperator{\tors}{tors}
\DeclareMathOperator{\comp}{comp}

\DeclareMathOperator{\loc}{loc}

\DeclareMathOperator{\inv}{inv}

\DeclareMathOperator{\Sub}{Sub}
\DeclareMathOperator{\cons}{con}

\DeclareMathOperator{\rank}{rank}

\DeclareMathOperator{\gen}{gen}

\DeclareMathOperator{\type}{type}

\newcommand{\Stone}{\Top_{\mathrm{Stone}}}

\newcommand{\Spectral}{\Top_{\mathrm{Spec}}}

\newcommand{\Priestley}{\Top_{\mathrm{Pries}}}
\newcommand{\priestley}{\mathrm{Pries}}

\DeclareMathOperator{\Pries}{Pries}
\newcommand{\Prism}{\mathrm{Prism}}

\newcommand{\cotoral}{\preccurlyeq_{\mathrm{ct}}}

\newcommand{\Ord}{\mathrm{Ord}}

\newcommand{\sfK}{\mathsf{K}}

\newcommand{\sfT}{\mathsf{T}}
\newcommand{\sfS}{\mathsf{S}}
\newcommand{\sfD}{\mathsf{D}}

\newcommand{\leftsquigarrow}{%
\mathrel{\reflectbox{{$\rightsquigarrow$}}}}

\newcommand{\Sp}{\mathsf{Sp}}

\newcommand{\Lct}{{\bigwedge}_{ct}}
\newcommand{\mct}{\mathrm{max}_{ct}}
\newcommand{\sub}{\mathrm{Sub}}

\newcommand{\aut}{\mathrm{Aut}}

\newcommand{\sm}{\wedge}
\newcommand{\lra}{\longrightarrow}
\newcommand{\tensor}{\otimes}
\newcommand{\st}{\mid}
\newcommand{\normal}{\trianglelefteq}
\newcommand{\T}{\mathbb{T}}

\newsavebox\prismsym
\savebox\prismsym{\begin{tikzpicture}\draw[thick] (0,-0.1) -- (0.2,1);
\draw[thick, dotted] (-0.3,0.3) -- (0.7,0.3);
\draw[ultra thick] (0,-0.1) -- (-0.3,0.3);
\draw[ultra thick] (0,-0.1) -- (0.7,0.3);
\draw[ultra thick] (0.2,1) -- (0.7,0.3);
\draw[ultra thick] (0.2,1) -- (-0.3,0.3);\end{tikzpicture}}

\newcommand{\SpG}{\Sp_{G}}
\newcommand{\SpGQ}{\Sp_{G,\Q}}
\newcommand{\SpGQo}{\Sp_{G,\Q}^{\omega}}

\newcommand{\height}{\mathrm{ht}}

\newcommand{\cbht}{\height_\mathrm{CB}}
\newcommand{\thht}{\height_{\mathrm{Th}}}
\newcommand{\repht}{\height_\mathrm{Rep}}

\newcommand{\Tbar}{\overline{T}}

\newcommand{\TTW}{\widetilde{TW}}

\definecolor{darkBlue}{rgb}{0.0, 0.18, 0.65}

\newcommand{\DdgMod}[2][]{\sfD(\mathrm{dg}\text{-} {\Mod}_{{#2}}^{#1})}

\title{Prismatic decompositions and rational $G$-spectra}

\author{Scott Balchin}
 \address{Mathematical Sciences Research Centre, Queen's University Belfast, UK}
 \email{s.balchin@qub.ac.uk }

\author{Tobias Barthel}
\address{Max Planck Institute for Mathematics, Vivatsgasse 7, 53111 Bonn, Germany}
\email{tbarthel@mpim-bonn.mpg.de}

\author{J.~P.~C.~Greenlees}
\address{Warwick Mathematics Institute, Zeeman Building, Coventry CV4 7AL, UK}
\email{john.greenlees@warwick.ac.uk}

\date{\today}

\begin{document}

\maketitle
\begin{abstract}
    We study the tensor-triangular geometry of the category of rational $G$-spectra for a compact Lie group $G$. In particular, we prove that this category can be naturally decomposed into local factors supported on individual subgroups, each of which admits an algebraic model. This is an important step and strong evidence towards the third author's conjecture that the category of rational $G$-spectra admits an algebraic model for all compact Lie groups. 
    
    To facilitate these results, we relate topological properties of the associated Balmer spectrum to structural features of the group $G$ and the category of rational $G$-spectra. A key ingredient is our presentation of the spectrum as a Priestley space, separating the Hausdorff topology on conjugacy classes of closed subgroups of $G$ from the cotoral ordering. We use this to prove that the telescope conjecture holds in general for rational $G$-spectra, and we determine exactly when the Balmer spectrum is Noetherian.  In order to construct the desired decomposition of the category, we develop a general theory of `prismatic decompositions' of rigidly-compactly generated tensor-triangulated categories, which in favourable cases gives a series of recollements for reconstructing the category from local factors over individual points of the spectrum. 
\end{abstract}



\setcounter{tocdepth}{1}
\makeatletter
\def\l@subsection{\@tocline{2}{0pt}{2.5pc}{5pc}{}}
\makeatother

\tableofcontents
\def\biblio{}

\newpage
\section*{Introduction}

The principal aim of this paper is to study the structural
properties of the category of  $G$-equivariant cohomology
theories taking rational values, where $G$ is a compact Lie
group. This is a fairly complicated category with a rich structure enabling us to describe it in general terms, and we explain which properties of the group lead to particularly good behaviour. 
The method applies to any tensor-triangulated category equipped with a sufficiently nice $\infty$-categorical enhancement, showing that if the Balmer
spectrum is well behaved, then the category can be reconstructed from 
local factors concentrated at single primes.

\subsection*{Context}
It is well known that equivariant cohomology theories $E_G^*(- )$ are represented
by $G$-spectra, so that there is a $G$-spectrum $E$ with
$[X,E]_G^*=E_G^*(X)$ for any based $G$-space $X$. As such, 
the category of $G$-equivariant cohomology theories is the homotopy
category of the category $\SpG$ of $G$-spectra. To reach a more algebraic
realm we will restrict attention to the category of 
cohomology theories with rational values, which is the homotopy
category of rational $G$-spectra which we denote $\SpGQ$. Based on results for particular groups, the third author made the following conjecture in the 1990s, 
eventually published in 2006:

\begin{conjecture*}[\cite{greenleesreport}]
For any compact Lie group $G$ there is  a graded abelian category $\cA(G)$ of injective dimension $\rank(G)$ and an equivalence 
\[
\SpGQ\simeq \sfD(\cA(G)) 
\]
between the category of rational $G$-equivariant cohomology theories and the derived category of differential-graded objects in $\cA(G)$. 
\end{conjecture*}

Numerous instances of this conjecture have been established including
for finite groups~$G$ \cite{greenlees_may,barnesfinite},
$SO(2)$~\cite{greenleesrationals1,shipleys1,BGKSs1},
$O(2)$~\cite{greenleeso2,barneso2},
$SO(3)$~\cite{greenleesso3,kedziorekso3}, and
tori~\cite{greenleestorus,GStorus}. The general case, however, remains open. Our main theorem establishes this conjecture up to extensions; for a more precise statement, see \cref{thmx:prismaticdecomposition} below:

\begin{theorem*}
For any compact Lie group $G$, the category $\SpGQ$ can be reconstructed from the categories of $G$-spectra with isotropy in a single conjugacy class of subgroups of $G$, and each of these has an algebraic model.
\end{theorem*}

Roughly speaking, the idea is to exploit the tensor-triangular geometry of $\SpGQ$ to construct a filtration on $\SpGQ$ whose strata agree with the ones of the conjectural algebraic model. The remaining step to a full solution of the third author's conjecture thus becomes to prove that the assembly process is formal; we intend to return to this in future work. 

The techniques developed in this paper have applications beyond the construction of algebraic models. We give a convenient presentation of the Balmer spectrum of $\SpGQ$ as a Priestley space and establish several structural properties of it, which is then in turn used to elucidate the geometry of $\SpGQ$. Detailed statements of our results are given in the following subsections.

\subsection*{Algebraic models and reconstruction}

Experience and ancient wisdom dictates the expected form of the 
category $\cA (G)$. The simplest $G$-spaces $X$ to understand are
the free $G$-spaces, so it is wise to begin by understanding those. Secondly, for a general space $X$ it is wise
to  consider the $H$-fixed point spaces $X^H$ for
all closed subgroups $H$, remembering the action of the Weyl group
$W_G(H)=N_G(H)/H$. In summary, we attempt to study $G$-spaces $X$ by
considering the free $W_G(H)$-spaces $X^H\sm EW_G(H)_+$ for all subgroups $H$.

We apply the same philosophy to $G$-spectra. The first critical ingredient is an understanding of free $G$-spectra. To this end, the third author and Shipley \cite{GSfree} have provided an algebraic model for free $G$-spectra: if $G$ is
connected, this is an equivalence between the category of rational free $G$-spectra and the derived category of differential graded torsion $H^\ast(BG)$-modules:
\begin{equation}
    \SpGQ^{\mathrm{free}} \simeq \DdgMod[\tors]{H^\ast(BG)},
\end{equation}
essentially using the torsion module $H_*(EG_+\sm_G X)$ as the invariant of $X$. In
general, if $G$ has identity component $G_e$, we remember the action of
the component group $G_d=\pi_0(G)=G/G_e$ on $G_e$ to construct the twisted group ring $H^*(BG_e)[G_d]$. With this, on the level of $\infty$-categories we obtain a natural symmetric monoidal equivalence 
\begin{equation}\label{eq:algebraic-free}
    \SpGQ^{\mathrm{free}} \simeq  \DdgMod[\tors]{H^*(BG_e)[G_d]}
\end{equation}
inspired by passage to the torsion module, as above.

The idea is that the  category $\cA (G)$ is constructed by assembling the data from each
closed subgroup $H$, and the local data at $H$ is the free
$W_G(H)$-spectrum formed from the geometric fixed point spectrum
$\Phi^HX$. Of course the information at different subgroups is
related, and a construction of a model must describe the relationship
in detail. We are not concerned here with the construction of an
abelian category,
but we will give general structural results at the homotopical level and give strong evidence for the
conjecture in general by showing the category of $G$-spectra admits a
filtration whose subquotients are exactly of the expected form. 

Our main result in this direction is as follows: 

\begin{thmx}\label{thmx:prismaticdecomposition}
 Let $G$ be a compact Lie group. Then the category $\SpGQ$ can be reconstructed from the categories $\DdgMod[\tors]{H^*(BH_e)[H_d]}$ as $H$ runs through conjugacy classes of subgroups of $G$. 
\end{thmx}

In fact, we formulate the reconstruction process explicitly as a homotopy limit whose combinatorics are governed by the geometry of the space of subgroups of $G$, as explained in detail below. Our general methods are directed towards this type of result, and we will rephrase the project in a way that helps make this plain, and then restate the method at the end of the introduction.

\subsection*{A revisionist account}
Returning to structural statements, the first main point is that the
category of $G$-spectra admits a tensor product in such a way that its 
homotopy category becomes a tensor-triangulated (tt)
category. In this context the basic tool for understanding the
structure of $\Spc(\SpGQ^\omega)$ is the Balmer spectrum \cite{balmer_spectrum}; we will recall the definitions
below, but we note for now that the spectrum is a spectral space, i.e., a topological space with the same character as the
prime spectrum of a commutative ring. 

The third author identified the Balmer spectrum $\Spc(\SpGQo)$ of
finite rational $G$-spectra in
\cite{greenlees_bs}. It is rather easy to see that for any closed
subgroup $H$, the set 
\[
\cP_H=\{ X \in \SpGQo\st \Phi^HX\simeq_1*\}
\]
is a tt-prime, evidently depending only on the conjugacy class of $H$. This gives a map 
\[
\cP_{\bullet} \colon \sub(G)/G\stackrel{\cong}\lra \Spc(\SpGQo).
\]
The first fact is that $\cP_{\bullet}$ is indeed a bijection, which shows that the support of a finite rational spectrum $X$ is precisely the geometric
isotropy
\[
\supp (X)=\{ (K) \in \Sub(G)/G \st \Phi^KX\not \simeq_1 *\}.
\]
The closed sets of the Zariski topology giving the Balmer spectrum is by definition generated by the supports
of finite objects $X$. 

The identification of supports already lets us
reinterpret the ancient wisdom of transformation groups (that the
important thing is to look at the geometric isotropy, which is a
collection of conjugacy classes of subgroups) in intrinsic
structural terms. It is not hard to prove this bijection using 
isotropy separation together with 
tom Dieck's identification of the endomorphisms of the unit
object and the equivalence of \eqref{eq:algebraic-free}. The second piece of structure is the containment
order in the Balmer spectrum. One can see that 
the Borel--Hsiang--Quillen Localization Theorem shows that if $K$ is
normal in $H$ with $H/K'$ a torus (with $K'$ some conjugate of $K$), then $\cP_K \subseteq \cP_H$. One may
easily construct enough finite spectra to show that the reverse implication
holds so that
\[
\xymatrix{
\cP_K \subseteq \cP_H \ar@{<=>}[r]& K\cotoral H,
}
\]
where the {\em cotoral} ordering $K\cotoral H$ means that $K$ is conjugate to $K'$ normal in $H$
with $H/K'$ a torus.  The results of \cite{greenlees_bs} also describe the
Zariski topology on the Balmer spectrum in terms of the subgroup
structure of $G$, but the description is
difficult to work with in practice. One may first think that $\sub (G)$ consists
of compact subspaces of the compact space $G$, and it is therefore a
metric space using the Hausdorff metric; the most obvious topology on
$\sub(G)/G$ is the {\em $h$-topology}, which is the topology induced by the
Hausdorff metric. The $h$-topology is Hausdorff and, when $\dim G >0$, 
completely different from the Zariski topology,
but many of our new results here are made possible by expressing the Zariski topology in terms of more familiar structures. The following is the subject of \cref{thm:prism}.

\begin{thmx}\label{thm:prismsratgspectra}
    Let $G$ be a compact Lie group. Then the closed sets of the Zariski topology are precisely the collections of subgroups closed in the $h$-topology and closed under cotoral specialization.
\end{thmx}

This resulting presentation of the Balmer spectrum $\Spc(\SpGQo)$ has been vital for our understanding of the category of rational $G$-spectra: it  allows for a significantly better control over the key point-set topological properties of the space and also makes transparent how to filter it in a way that is compatible with the structure of $\SpGQ$. 

In fact, the separation of the topology into a Hausdorff space and an ordering is a general phenomenon. It is captured by the concept of a \emph{Priestley space} which provides a presentation of a spectral space in terms of a Stone space equipped with a certain poset structure. Let us take a moment to discuss the psychological benefits of working with Priestley spaces. It is often the case that the topology of the Balmer spectrum is determined purely by the inclusions of the prime ideals, and this observation has led to fruitful developments. However, this is not always the case, and indeed will not be the case for us; this is a source of substantial complication in the theory. Nonetheless, an amplified version of this approach can still be used if one can determine some extra data, namely the constructible topology of the Balmer spectrum which is always a Stone space. This Stone space along with the inclusions between primes then uniquely determines the Balmer spectrum. We call the corresponding pair the \emph{prism}. The result of \cref{thm:prismsratgspectra} can thus be restated as
\[
         \Prism(\SpGQo) = ((\Sub(G)/G)_h, \cotoral)
\]
and places the focus on 
\begin{itemize}
    \item[(i)] the Hausdorff metric on the closed 
subgroups, and the resulting $h$-topology on the space of conjugacy 
classes and
\item[(ii)] the cotoral inclusion relation,
\end{itemize}
an observation made (even integrally) in \cite{BarthelGreenleesHausmann2020}. As a result we are involved in detailed work with spaces of subgroups.  The key
topological fact is the Montgomery--Zippin theorem stating that any
subgroup $K$ close enough to $H$ is subconjugate to $H$. We also need
to use the structure of compact Lie groups and the pioneering work of T.~tom Dieck from the 1970s.

\subsection*{Structural consequences}
Using the description of the Balmer spectrum in terms of its Priestley space, we will show that the  spectral spaces occurring as $\Spc(\SpGQo)$
are  rather special, but the range of behaviours as $G$ varies is quite
rich, so that it is a valuable source of examples.  In \cref{sec:finitenessproperties} we explore topological properties of the Balmer spectrum. Our first structural result is a full classification of when the Balmer spectrum is Noetherian.

\begin{thmx}
The Balmer spectrum of $\SpGQo$ is Noetherian if and only if $G$ is a 
finite central extension of a torus.  
\end{thmx}

In fact the finite central extensions of a torus are precisely the
compact Lie groups with finitely many conjugacy classes of subgroups with finite
Weyl groups, or equivalently they are the groups not containing a
subgroup $H$ which has identity component $H_e$ a torus and a component group $H_d$ acting non-trivially on it. These are precisely the groups whose rational Burnside ring is Noetherian, so in this case the spectrum is Noetherian if and only if the endomorphism ring of the unit is.

While not all the spectra that we encounter are Noetherian, the
general behaviour is excellent. We recall that a space is called \emph{generically Noetherian} if the generalization closure of each point is a Noetherian subspace. We have the following:

\begin{thmx}
For any compact Lie group the Balmer spectrum of $\SpGQo$ is 
generically Noetherian.
\end{thmx}

Being generically Noetherian is sufficient to show that any single prime
is the difference of two Thomason subsets, and hence that both completion
and cellularization at a single prime are very well behaved. As a more
global statement, knowing that $\SpGQ$ is stratified by \cite{greenlees_bs}, we may apply a result of Barthel--Heard--Sanders \cite[Theorem 9.11]{BHS2023} to deduce that the kernels of all smashing localizations are generated by compact
objects, that is:

\begin{thmx}
The telescope conjecture holds in $\SpGQ$ for every compact Lie group $G$. 
\end{thmx}

\subsection*{The height filtration}
The last of the  properties of a  Balmer spectrum that is critical to
our method is that of admitting a good filtration: we give a precise
definition in \cref{sec:dispersions}, but the
essential features are that it is finite, and the  strata are
discrete. Because this has important structural implications, we call the
filtration a {\em dispersion} and that a
spectral space with this property is {\em dispersible}. 

We prove  that the Balmer spectrum $\fX_G=\Spc(\SpGQo)$
 is dispersible.  We start with $\fX_G$ and declare 
the Thomason points (those which are $h$-isolated and cotorally minimal) 
to be of height 0. Removing points of height 0, we obtain a new space $\delta_{T} \fX_G$ and the 
Thomason points of this are of height 1, and so forth. We show
that this filtration is finite, and in fact all  subgroups are of height $\leqslant \rank (G)$; the precise height of each 
subgroup $H$ depends on the action of the component group $H_d$ on the 
identity component of the centre of $H_e$. In particular we have in \cref{thm:heightformula}:

\begin{thmx}
Let $G$ be a compact Lie group with subgroup $H$. We consider the identity component of each subgroup $H$, and let $T$ be the identity component of its centre. The component group $H_d$ acts on $H_1(T; \Q)$: the height of $H$ is the number of simple summands of this rational representation.
\end{thmx}

Furthermore we will show that the 
 strata $\delta^s_{T}\fX_G\setminus \delta^{s-1}_{T}\fX_G$ are discrete, the importance of which we now discuss.

\subsection*{The categorical local-to-global principle}

If we suppose our tensor-triangulated category admits a good enhancement as the homotopy category of a symmetric monoidal stable $\infty$-category, we may draw further 
structural conclusions. The most obvious is that if we partition the
Balmer spectrum into finitely many clopen sets, the $\infty$-category
splits; in fact one can show that $\SpGQ$ admits such a partition into rather
standard pieces, but we will not prove this here. Instead, we give a
general procedure showing that if the Balmer  spectrum is  dispersible then the underlying 
$\infty$-category can be reconstructed by a finite sequence of recollements (we say that the
category satisfies the {\em categorical local-to-global principle}),
from local categorical factors supported at individual primes. The following is \cref{thm:catl2g}:

\begin{thmx}
Suppose $\sfT$ is a rigidly-compactly generated tensor-triangulated
category with a good enhancement. If the Balmer spectrum of $\sfT^\omega$ is dispersible, then $\sfT$ satisfies  the categorical local-to-global principle. 
\end{thmx}
 
Applying this to rational $G$-spectra along with \eqref{eq:algebraic-free} we obtain convincing evidence for the third author's conjecture:

\begin{thmx}
    For any compact Lie group $G$, the Balmer spectrum $\Spc(\SpGQo)$
    is dispersible. Hence $\SpGQ$ satisfies the categorical
    local-to-global principle and can be reconstructed from the local
    factors $\Gamma_H \SpGQ  \simeq  \Sp_{W_G(H),\Q}^{\free}$ as
    $H$ varies through closed subgroups.
\end{thmx}

We will describe this reconstruction explicitly in a  number of small
examples in \cref{sec:catstratforGspectra}.

\subsection*{Related work}
Some of the abstract machinery developed in this paper is related to work of Stevenson~\cite{stevenson_localglobal} and Ayala--Mazel-Gee--Rozenblyum~\cite{AMGR}.  We refer to the beginning of \cref{part:prisms} for a detailed description of our approach and its relation to the aforementioned papers.

\subsection*{Outline of the document}

 In \cref{part:prisms} we will describe the relevant general machinery and
 strategy for tensor-triangulated categories with an $\infty$-categorical enhancement. 
In Sections \ref{sec:priestley} and  \ref{sec:dispersions} we recall the 
apparatus of spectral spaces, with a particular focus on the language of Priestley spaces and good filtrations, and in 
\cref{sec:priestleyintt} we make explicit the reformulation of concepts
from tensor-triangular geometry in the Priestley setting. Finally, in 
\cref{sec:obstrat} we describe how dispersions give us decompositions
of objects, and in \cref{sec:catstrat_lax,sec:catstrat_strict} we describe the corresponding
decompositions of categories in terms of lax and strict limits, respectively.

In \cref{part:gspectra} we apply the language and results of
\cref{part:prisms} to the category of rational $G$-spectra. \cref{sec:sub} recalls
standard facts about spaces of subgroups of compact Lie
groups. \cref{sec:prism} gives the  description of the Balmer spectrum
$\Spc(\SpGQo)$ in familiar terms using Priestley
spaces. \cref{sec:subgroups} contains the key results about
spaces of subgroups, their topologies and the cotoral ordering: the
fact that semisimple groups cannot be approximated by subgroups
reduces many arguments to finite groups acting on
tori. \cref{sec:finitenessproperties} applies the results of the
previous section to prove general finiteness
results. \cref{sec:dispersionsforGspec} describes the Thomason height
filtration and derives a formula for it in terms of the representation theory of the given subgroup. In the final  \cref{sec:catstratforGspectra}, we make the categorical
local-to-global principle explicit for some small groups, in particular covering all examples studied in the literature previously.

\subsection*{Conventions}

We collect some pieces of standard notation and terminology that we will use throughout the paper.

\newcommand{\Nvee}{{\N^+}}
\begin{itemize}
	\item $\N = \{0, 1, 2, \dots\}$ denotes the set of natural
          numbers. 
    \item $\Nvee =\{ 1, 2, \dots\}$ denotes the set of positive natural
          numbers.
	\item $\mathbb{P}$ denotes the set of prime numbers. 
	\item $[n] = \{0,1,2, \dots, n\}$.
 	\item For a subset $S \subseteq X$ in a set $X$, we write $S^c = X \setminus S$ for its complement in $X$.
        \item Priestley spaces will be denoted by boldface letters such as $\pmb{P}$, while the underlying ordered Stone space is then written $(P, \leqslant)$.
	\item Minimal points are closed in the Balmer spectrum. Accordingly, when drawing an ordered topological space, the minimum elements will be at the bottom of the graphic.
\end{itemize}
\subsection*{Acknowledgements }

We are grateful to Markus Hausmann, Greg Stevenson for conversations related to the results of this project and thank Luca Pol and Jordan Williamson for useful comments on a preliminary draft of this paper. SB and TB would like to thank the Max Planck Institute for Mathematics for its hospitality, and were supported by the European Research Council (ERC) under Horizon Europe (grant No.~101042990). JPCG is grateful to the EPSRC for support from EP/P031080/1. The authors would also like to thank the Hausdorff Research Institute for Mathematics for its hospitality and support during the trimester program `Spectral Methods in Algebra, Geometry, and Topology', funded by the Deutsche Forschungsgemeinschaft under Germany's Excellence Strategy – EXC-2047/1 – 390685813.

\newpage

\part{Prismatic Decompositions}\label{part:prisms}

The goal of this part is to set up a general framework for constructing models of a tensor-triangulated (tt) category $\sfT$ from its local factors. The general idea is easy to describe: $\sfT$ should be viewed as a sheaf of tt-categories over its Balmer spectrum $\Spc(\sfT^{\omega})$. Locality in $\sfT$ then refers to objects supported in a singleton, and a suitable filtration of the spectrum should lift to a categorical filtration of $\sfT$. 

In order to make this precise, it is crucial to be able to construct filtrations that interact correctly with the geometry of $\sfT$. Many prominent examples, such as the derived category of a Noetherian commutative ring or the equivariant stable homotopy category for a finite group, have the property that the topology on $\Spc(\sfT^{\omega})$ is captured completely by the specialization poset of $\Spc(\sfT^{\omega})$, which then significantly simplifies the problem. In the context of equivariant homotopy theory for a general compact Lie group, however, this is no longer the case and leads to rich topological phenomena on the spectrum. This necessitates the development of a theory that works in the generality of arbitrary spectral spaces. The key steps of our approach are:
    \begin{enumerate}
        \item Use Priestley spaces to present the spectral topology on $\Spc(\sfT^{\omega})$ as an ordered Stone space. This captures the extra data required to determine the spectral topology via its specialization order, namely the underlying constructible topology. We call the Priestley space of $\Spc(\sfT^{\omega})$ the prism, denoted $\Prism(\sfT^{\omega})$. For convenience, we have collected the required background material in \cref{sec:priestley} and \cref{sec:priestleyintt}.
        \item Transposed to the realm of Priestley spaces, a modification of Stevenson's dimension functions \cite{stevenson_localglobal}, dubbed dispersions, then provide a good notion of filtration on $\Prism(\sfT^{\omega})$. We introduce dispersions and establish their key properties in \cref{sec:dispersions} and \cref{sec:priestleyintt}. In particular, a dispersion then leads to a single filtration on the prism, in a way which takes into account both the specialization order and the constructible topology. Stevenson's work shows that any object in $\sfT$ can be reconstructed from local pieces via such a filtration; this is the objectwise local-to-global principle discussed in \cref{sec:obstrat}.
        \item In order to lift the filtration to a categorical decomposition of $\sfT$, we then employ the theory of Ayala--Mazel-Gee--Rozenblyum \cite{AMGR}. Their input is a suitable filtration on $\sfT$ such as the one constructed in Step (2), and the output is the reconstruction of $\sfT$ as a right-lax limit of local categorical factors. In \cref{sec:catstrat_lax,sec:catstrat_strict} and following \cite{AMGR}, we make this reconstruction process explicit in two ways: as an iteration of recollements and as a (strict) homotopy limit over a suitable cubical diagram of categories. 
    \end{enumerate}
The abstract theory will be put in action in \cref{part:gspectra}, and we invite the reader mostly interested in the applications to rational $G$-spectra to skip ahead to that part.

\section{Priestley spaces}\label{sec:priestley}

In this first section, we introduce the concept of Priestley spaces as a convenient presentation of spectral spaces. Our main reference for these results comes from \cite[\S 1.5]{book_spectralspaces}, while the original reference is \cite{priestley}. We begin by recalling the definition of a spectral space.

\subsection{Spectral spaces}

Spectral spaces can be characterised as those spaces arising as the Zariski spectrum $\Spec(R)$ for $R$ a commutative ring~\cite{Hochster}. The following provides an intrinsic characterization of spectral spaces.

\begin{definition}
A topological space $X$ is \emph{spectral} if satisfies the following properties:
\begin{enumerate}
    \item finite intersections of quasi-compact opens in $X$ (and hence $X$ itself) are quasi-compact, 
    \item the collection of quasi-compact opens in $X$ forms a basis for the topology of $X$,
    \item and $X$ is sober, i.e., for every non-empty closed and irreducible subset $C \subseteq X$ there exists a unique $x \in X$ whose closure is $C$.
\end{enumerate}
A continuous map $f\colon X \to Y$ between spectral spaces is called \emph{spectral} if it is quasi-compact, i.e., the preimage under $f$ of any quasi-compact open subset of $Y$ is quasi-compact in $X$. We write $\Spectral$ for the category of spectral spaces and spectral maps between them.\footnote{A further characterization of spectral spaces is as (cofiltered) inverse limits of finite posets.} 
\end{definition}

Of particular importance for a spectral space $X$ is the \emph{specialization order}, which provides a partial order on the points of $X$.

\begin{definition}
For $X$ a spectral space, and $x,y \in X$, we say that $y$ is a \emph{specialization} of $x$ if $y \in \overline{\{x\}}$ and write $y \leftsquigarrow x$. The relation $\leftsquigarrow$ will be referred to as the \emph{specialization order}.
\end{definition}

\begin{convention}\label{conv:ordering}
    Let $X$ be a spectral space. We will follow the convention that $y\leqslant x$ if and only if $y \leftsquigarrow x$. In particular, closed points of the space are minimal with respect to the specialization order while generic points are maximal. (We alert the reader to the fact that this convention is opposite of what is used in \cite{book_spectralspaces}. Our choice is particularly convenient when applied to the study of tensor-triangulated categories (cf.~\cref{ssec:ttgeometry}).)
\end{convention}

Priestley spaces give a presentation of a spectral space $X$ which records the topology on $X$ via an ordered topological space where the ordering is given by the specialization order. In preparation for the theory, we now introduce the definition of the constructible topology for a spectral space along with its properties. We recall that a topological space $X$ is a \emph{Stone space}\footnote{In \cite{book_spectralspaces}, Stone spaces are referred to as Boolean spaces. The latter term is sometimes used for the more general class of zero-dimensional, locally compact Hausdorff spaces, which is why we will use the more traditional terminology.} if it is totally disconnected, compact, and Hausdorff. The collection of Stone spaces and continuous maps between them forms the category $\Stone$. 

\begin{definition}\label{defn:constructibletopology}
Let $X$ be a spectral space. The \emph{constructible topology} on $X$ has a subbasis of open subsets consisting of the quasi-compact opens of $X$ and their complements. The resulting topological space is denoted $X_{\cons}$ and called the \emph{patch space} of $X$. A subset of $X$ is called \emph{constructible} if it is open and quasi-compact in $X_{\cons}$.
\end{definition}

\begin{proposition}[{\cite[\S 1.3]{book_spectralspaces}}]\label{prop:constructibletopology}
Let $X$ be a spectral space.
\begin{enumerate}[label=(\roman*)]
\item The patch space $X_{\cons}$ is a Stone space, and a subset $S \subseteq X$ is constructible if and only if it is clopen in $X_{\cons}$.
\item The patch space of a Stone space $X$ coincides with $X$, hence the patch space construction is idempotent, i.e., $(X_{\cons})_{\cons} = X_{\cons}$.
\item The assignment of $X \mapsto X_{\cons}$ is functorial and exhibits $\Stone$ as a coreflective subcategory of $\Spectral$.
\end{enumerate}
\end{proposition}

\subsection{Priestley spaces}
We will now introduce the theory of Priestley spaces. The punchline is as follows: any spectral space $X$ is uniquely presented by the ordered topological space $(X_{\cons}, \leftsquigarrow)$, and moreover this assignment facilitates an isomorphism of categories. Let us begin by recalling the definition of an ordered topological (sub)space.

An \emph{ordered topological space} is a pair $\pmb{X} = (X, \leqslant)$ where $X$ is a topological space and $\leqslant$ is a partial ordering on the underlying set of $X$. For $S \subseteq X$, we can talk about ordered subspaces $\pmb{S} = (S, \leqslant)$ of $\pmb{X}$ by restricting the order. The same logic applies to constructions such as the complement of a subspace. 

Let us fix some further terminology before we continue. Suppose that $\pmb{X} = (X, \leqslant)$ is an ordered topological space. An ordered subspace $\pmb{S} = (S, \leqslant) \subseteq \pmb{X}$ is \emph{open} (\emph{closed}) if $S \subseteq X$ is open (closed). As usual, a subset is called \emph{clopen} if it is both open and closed. An ordered subspace $\pmb{S}$ is an \emph{up-set} (\emph{down-set}) if for all $s \in S$ we have $x \in S$ whenever $x \geqslant s$ (whenever $x \leqslant s$).

We do not want to consider all ordered spaces, only those where the ordering can be realised as the specialization ordering for a spectral space. If one unravels the properties of specialization, one is led to the following separation axiom.

\begin{definition}\label{def:spectralorder}
Let  $\pmb{X} = (X, \leqslant)$ be an ordered topological space. Then $\leqslant$ is a \emph{spectral order} if it satisfies the \emph{Priestley separation axiom}:
\[
\text{for all } x \not\leqslant y \text{ there exists a clopen up-set } \pmb{U} \text{ such that } x \in \pmb{U} \text{ and } y \not\in \pmb{U}.
\]
\end{definition}

We now wish to apply this perspective to spectral spaces. Indeed, the specialization relation on a spectral space is an example of a spectral order for the constructible topology:

\begin{lemma}\label{lem:isspectralorder}
If $X$ is a spectral space, then the specialization order $\leftsquigarrow$ is a spectral order on $X_{\cons}$.
\end{lemma}

This motivates the main definition of this section, due to \cite{priestley}.

\begin{definition}
A \emph{Priestley space} $\pmb{P} = (P,\leqslant)$ is a Stone space $P$ equipped with a spectral order $\leqslant$. If $\pmb{P} = (P,\leqslant)$ and $\pmb{Q} = (Q,\leqslant)$ are Priestley spaces, then a map $f \colon \pmb{P} \to \pmb{Q}$ is a \emph{Priestley map} if it is continuous as a map of Stone spaces $P \to Q$ and monotone for the spectral orders. The collection of Priestley spaces and Priestley maps between them assembles into a category $\Priestley$.
\end{definition}

The following is immediate from \cref{prop:constructibletopology} combined with \cref{lem:isspectralorder}.

\begin{corollary}\label{lem:spectopriestley}
If $X$ is a spectral space then $(X_{\cons}, \leftsquigarrow)$ is a Priestley space.
\end{corollary}

We now discuss how to go the other way, that is, constructing a spectral space from a Priestley space.  Suppose that we have a Priestley space $\pmb{P}= (P,\leqslant)$. Then we can equip $P$ with two additional topologies coming from the spectral order: 
\begin{itemize}
    \item $\tau_u$ which has open sets given by the open subsets of $P$ which are also up-sets; 
    \item $\tau_d$ which has open sets given by the open subsets of $P$ which are also down-sets. 
\end{itemize}

\begin{lemma}
If $\pmb{P}= (P,\leqslant)$ is a Priestley space, then the spaces $(P, \tau_u)$ and $(P, \tau_d)$ are spectral spaces.
\end{lemma}

The assignment $\pmb{P} = (P, \leqslant) \mapsto (P, \tau_u)$ provides a functor $\tau_u \colon \Priestley \to \Spectral$.  We also have the functor  $\priestley \colon \Spectral \to \Priestley$ coming from \cref{lem:spectopriestley}. The following theorem tells us that these functors are inverses of one another in the strongest sense.

\begin{theorem}[{\cite[1.5.15]{book_spectralspaces}}]\label{thm:priestleyspaces}
The functor $\priestley \colon \Spectral \to \Priestley$ is an isomorphism of categories with inverse $\tau_u \colon \Priestley \to \Spectral$.
\end{theorem}

\begin{remark}
It is worth emphasizing that when referring to open or closed subsets of a Priestley space we are referring to {\em its} topology, which is the constructible topology on the associated spectral space. 
\end{remark}
Unravelling definitions we arrive at the following observations.

\begin{corollary}[{\cite[Theorem 1.5.11]{book_spectralspaces}}]\label{cor:topologydesc}
    Let $\pmb{P} = (P, \leqslant)$ be a Priestley space with associated spectral space $(P,\tau_u)$. Then:
    \begin{enumerate}[label=(\arabic*)]
        \item The open sets of $(P,\tau_u)$ are in bijection with the open up-sets of $\pmb{P}$.
        \item The closed sets of $(P,\tau_u)$ are in bijection with the closed down-sets of $\pmb{P}$.
        \item A basis of open sets of $(P,\tau_u)$ is given by the clopen up-sets of $\pmb{P}$.
        \item\label{item:basicclosed} A basis of closed sets of $(P,\tau_u)$ is given by the clopen down-sets of $\pmb{P}$. 
    \end{enumerate}
\end{corollary}

Given \cref{thm:priestleyspaces}, we see that we lose no information in representing a spectral space by its associated Priestley space. It allows us to separate out the specialization order and constructible topology in a useful manner which will prove advantageous in our main application of interest in \cref{part:gspectra}.

\begin{remark}\label{rem:noethpriestley}
We will use topological terms to describe the corresponding Priestley space. For example, if $\pmb{P}=(P, \leqslant)$ is a Priestley space whose corresponding spectral space is Noetherian, we will say that $\pmb{P}$ is a \emph{Noetherian Priestley space}. Explicitly, this means that $\pmb{P}$ satisfies the descending chain condition on closed down-sets. For a host of other equivalent characterizations of Noetherianness in the context of spectral spaces, we refer to \cite[Section 8.1]{book_spectralspaces}.
\end{remark}

We mentioned previously that there are two natural topologies associated to any Priestley space, namely $\tau_u$ and $\tau_d$. So far, we have only seen the role of $\tau_u$. The topology $\tau_d$ has a very natural interpretation in the world of spectral spaces, namely it exhibits the \emph{inverse topology} (sometimes referred to as the \emph{Hochster dual}) of the topology $\tau_u$ which is an involution on the category $\Spectral$. Let us define the ingredients of the inverse topology as it will be convenient for us to have later.

\begin{definition}\label{defn:thomason}
Let $X$ be a spectral space. A subset $V \subseteq X$ is \emph{Thomason} if it can be written as a union
\[
V = \bigcup_{\lambda} V_{\lambda}
\]
where each $V_\lambda$ is closed with quasi-compact complement.
\end{definition}

\begin{definition}\label{defn:inversetopology}
Let $X$ be a spectral space. The \emph{inverse topology} on $X$ is the space which has the same points as $X$, and whose open subsets are given by the Thomason subsets of $X$. We will denote the inverse topology as $X_{\inv}$.
\end{definition}

\begin{proposition}[{\cite[Theorem 1.4.3]{book_spectralspaces}}]\label{prop:inversetopology}
If $X$ is a spectral space, then so is $X_{\inv}$. A subset $U \subseteq X_{\inv}$ is open and quasi-compact in $X_{\inv}$ if and only if $U \subseteq X$ is the complement of a quasi-compact open subset in $X$. Moreover, the specialization order $\leftsquigarrow_{\inv}$ on $X_{\inv}$ is the inverse of the specialization order $\leftsquigarrow$ on $X$, i.e., for two elements $x,y \in X$ we have $x \leftsquigarrow y$ if and only if $y \leftsquigarrow_{\inv} x$. 
\end{proposition}

One can check that the inverse of a spectral order is again a spectral order.

\begin{lemma}\label{lem:inverseisspec}
Let $\pmb{P}=(P,\leqslant)$ be a Priestley space. Then the \emph{inverse} $\pmb{P}_{\inv} = (P, \leqslant_{\inv})$ is a Priestley space.
\end{lemma}

\begin{lemma}\label{lem:invonpriestley}
The involution $(-)_{\inv}$ on the category $\Spectral$ corresponds to taking the inverse order on the category $\Priestley$.
\end{lemma}

\begin{remark}\label{rem:thomasoninpriest}
In light of \cref{lem:invonpriestley} we see that if $X$ is a spectral space, then the Thomason subsets of $X$ are the open down-sets of the associated Priestley space. In particular, the Thomason points are exactly the isolated minimal points.
\end{remark}

It should be evident that any ordered subspace of a Priestley space need not be a Priestley space (indeed, a subspace of a Stone space need not be Stone). We obtain the following result via the corresponding result in spectral spaces from~\cite[Theorem 2.1.3]{book_spectralspaces}.

\begin{lemma}\label{lem:subpriestley}
Let $\pmb{P}= (P,\leqslant)$ be a Priestley space. Then any closed up-set of $\pmb{P}$ is again a Priestley space.
\end{lemma}

\subsection{The examples}

It is high time for some examples. We will explore Priestley spaces arising from a fixed Stone space of interest.

\begin{definition}\label{def:onepointcpt}
Let $X$ be a topological space, and write $X^* = X \cup \{\infty\}$ topologized by taking the open sets to be all of the open subsets $U$ of $X$ together with the sets $V = (X \backslash C) \cup \{\infty\}$ where $C$ is closed and quasi-compact in $X$. Then we say that the open inclusion $c\colon X \to X^*$ is the \emph{Alexandroff extension} of $X$. 
\end{definition}

If $X$ is locally compact, Hausdorff but not quasi-compact, then $X^*$ is compact Hausdorff and the image of $c$ is dense. In this case we refer to $X^{*}$ as the \emph{one-point compactification} of $X$. If $X$ is discrete, then $X^{*}$ is a Stone space.

\begin{example}\label{example:onepointcpt}
To illustrate this construction, we may take $X = \N$ to be the natural numbers with the discrete topology, whose one-point compactification will be the Stone space appearing in our guiding examples of Priestley spaces. Note that $\N^*$ is homeomorphic to the space $\{0\} \cup \{\frac{1}{n+1}\mid n \in \N\} \subset \R$ and can thus be visualized as in \cref{fig:Nast2}.
\end{example}

Equipped with our Stone space $\N^*$, we are in a position to explore some spectral orders on it along with the resulting spectral spaces under the isomorphism of categories from \cref{thm:priestleyspaces}.

\begin{example}\label{ex:priestleyonepointcpt}
Out first order, $\leqslant_1$ will be the  {\em trivial order} (equality) and as such we will denote it by $=$. This is a spectral order so that $(\N^{*}, =)$ is a Priestley space. As there is no data contained within the ordering, the corresponding spectral space is $\N^{*}$ itself. More generally, the same observation applies to any Stone space $X$ in place of $\N^{*}$. This space is represented in \cref{fig:Nast2}.
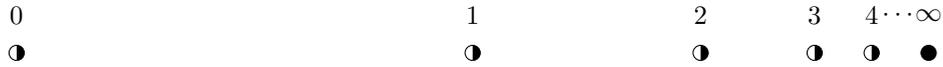
\begin{figure}[ht!]
\centering
\begin{tikzpicture}[xscale=-1]
\node at (12,1.5) {0};
\node at (6,1.5) {1};
\node at (3,1.5) {2};
\node at (1.5,1.5) {3};
\node at (0.75,1.5) {4};
\node at (0.375,1.5) {$\cdots$};
\node at (0,1.5) {$\infty$};
\draw[fill=white] (12,1) circle (.1cm);
 \fill[black] (12,0.9) arc (270:90:.1cm);
\draw[fill=white] (6,1) circle (.1cm);
 \fill[black] (6,0.9) arc (270:90:.1cm);
\draw[fill=white] (3,1) circle (.1cm);
 \fill[black] (3,0.9) arc (270:90:.1cm);
\draw[fill=white] (1.5,1) circle (.1cm);
 \fill[black] (1.5,0.9) arc (270:90:.1cm);
\draw[fill=white] (.75,1) circle (.1cm);
 \fill[black] (.75,0.9) arc (270:90:.1cm);
\draw[fill=black] (0,1) circle (.1cm);
\end{tikzpicture}\caption{The spectral space associated to the Priestley space $\pmb{N}^*_{=}$. The half coloured circles indicate that the subsets $\{n\}$ corresponding to $n \in \N$ are clopen, while the accumulation point $\infty$ is closed but not open.}\label{fig:Nast2}
\end{figure}
\end{example}

\begin{example}\label{ex:specz}
There is an order $\leqslant_2$ on $\N^{*}$ defined by $x \leqslant_2 y$ if and only if $x = y$ or $y = \infty$ for any two points $x,y \in \N^{*}$ (i.e., the compactifying point $\infty$ is the maximal element). One can check that the ordering $\leqslant_2$ is a spectral ordering. 

Let us determine the resulting spectral space coming from the topology $\tau_u$. Any up-set of $\leqslant_2$ contains the point $\infty$, so we are reduced to checking which open sets of $\N^*$ contain $\infty$. By construction of the one point compactification these look like $V = (X \backslash C) \cup \{\infty\}$ where $C$ is closed and compact in $\N$ (i.e., $C$ is finite). As such the resulting space can be pictorially described as in \cref{fig:specZ}. It is homeomorphic to the familiar spectral space $\Spec(\Z)$ by choosing an identification $\N \cong \mathbb{P}$ where $\mathbb{P}$ is the set of prime numbers. In particular it is a Noetherian topological space (i.e., the closed subsets satisfy the descending chain condition).
\begin{figure}[ht!]
\begin{tikzpicture}[yscale=1]
\draw [->,line join=round,decorate, decoration={zigzag, segment length=6, amplitude=.9,post=lineto, post length=2pt}]  (0,0) -- (-3.75,-1.2);
\draw[white, fill=white] (0,0) circle (.1cm);
\draw [->,line join=round,decorate, decoration={zigzag, segment length=6, amplitude=.9,post=lineto, post length=2pt}]  (0,0) -- (3.75,-1.2);
\draw[white, fill=white] (0,0) circle (.3cm);
\draw [->,line join=round,decorate, decoration={zigzag, segment length=6, amplitude=.9,post=lineto, post length=2pt}]  (0,0) -- (-1.8,-1.25);
\draw [->,line join=round,decorate, decoration={zigzag, segment length=6, amplitude=.9,post=lineto, post length=2pt}]  (0,0) -- (1.8,-1.25);
\draw [->,line join=round,decorate, decoration={zigzag, segment length=6, amplitude=.9,post=lineto, post length=2pt}]  (0,0) -- (0.0,-1.4);
\draw[white, fill=white] (0,0) circle (.3cm);
\draw[dotted, fill=black!10,thick] (0,0) circle (.15cm);
\draw[fill = black] (-4,-1.25) circle (0.1cm);
\draw[fill = black] (-2,-1.5) circle (0.1cm);
\draw[fill = black] (0,-1.75) circle (0.1cm);
\draw[fill = black] (2,-1.5) circle (0.1cm);
\node at (0,.5) {$\infty$};
\node at (-4,-1.75) {$0$};
\node at (-2,-2.0) {$1$};
\node at (0,-2.25) {$2$};
\node at (2,-2.0) {$3$};
\node at (4,-1.75) {$\cdots$};
\end{tikzpicture}\caption{The spectral space associated to the Priestley space $\pmb{N}^*_2 = (\N^*, \leqslant_2)$. The solid bullets represent closed points, while the dotted gray circle indicates the unique generic point.}\label{fig:specZ}
\end{figure}
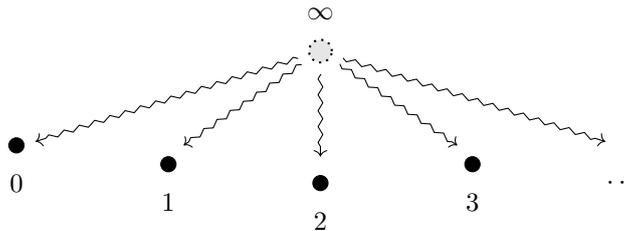
\end{example}

\begin{example}\label{ex:nbar}
We now consider a third ordering on $\N^*$ which we denote $\leqslant_3$ given by the opposite of the usual size order on $\N$ and declaring $\infty \leqslant_3 x$ for all $x \in \N^*$. Once again we can check that $\leqslant_3$ is a spectral order. The spectral space associated to $(\N^{*},\leqslant_3)$ can be visualized as in \cref{fig:nbar}. Unlike in \cref{ex:specz} the resulting topological space is not Noetherian as we have an infinite descending chain of closed subsets (by contrast, the Hochster dual topology on this space is Noetherian).

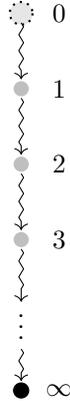
\begin{figure}[ht!]
\centering
\begin{tikzpicture}[yscale=-1]
\draw[dotted, fill=black!10,thick] (0,0) circle (.15cm);
\draw[fill = black!25, draw = none] (0,1) circle (0.1cm);
\draw[fill = black!25, draw = none] (0,2) circle (0.1cm);
\draw[fill = black!25, draw = none] (0,3) circle (0.1cm);
\node at (0,4.05) {$\vdots$};
\node at (0.5,0.0) {0};
\node at (0.5,1) {1};
\node at (0.5,2) {2};
\node at (0.5,3) {3};
\node at (0.5,5) {$\infty$};
\draw[fill = black] (0,5) circle (0.1cm);
\draw [->,line join=round,decorate, decoration={zigzag, segment length=6, amplitude=.9,post=lineto, post length=2pt}]  (0,0.125) -- (0,1.75/2);
\draw [->,line join=round,decorate, decoration={zigzag, segment length=6, amplitude=.9,post=lineto, post length=2pt}]  (0,2.25/2) -- (0,3.75/2);
\draw [->,line join=round,decorate, decoration={zigzag, segment length=6, amplitude=.9,post=lineto, post length=2pt}]  (0,4.25/2) -- (0,5.75/2);
\draw [->,line join=round,decorate, decoration={zigzag, segment length=6, amplitude=.9,post=lineto, post length=2pt}]  (0,6.25/2) -- (0,7.65/2);
\draw [->,line join=round,decorate, decoration={zigzag, segment length=6, amplitude=.9,post=lineto, post length=2pt}]  (0,8.9/2) -- (0,9.75/2);
\end{tikzpicture}
\caption{The spectral space associated to the Priestley space $\pmb{N}^*_3 = (\N^*, \leqslant_3)$. The solid bullet represents the closed point, the dotted gray circle indicates the (open) generic point, and the grey bullets are points which are neither open or closed.}\label{fig:nbar}
\end{figure}
\end{example}

\begin{example}\label{ex:speczinv}
We now consider a fourth ordering on $\N^*$ which we denote $\leqslant_4$. Define $x \leqslant_4 y$ if and only if $x=y$ or $x=\infty$. That is, $\leqslant_4$ is the inverse of the partial order $\leqslant_1$ as introduced in \cref{ex:specz}. As it is the inverse of a spectral order, $\leqslant_4$ is again a spectral order by \cref{lem:inverseisspec}. We see that $\pmb{N}_4^\ast = (\N^\ast, \leqslant_4)$ is therefore the inverse of the Priestley space $\pmb{N}^*_2$ from \cref{ex:specz} and as such can be presented as in \cref{fig:specZinv}. We highlight that the minimal point $\infty$ is not isolated in $\N^\ast$, this is in stark contrast to the minimal points of the space of \cref{ex:specz} where the minimal points are indeed isolated in $\N^\ast$. In fact, one can see that $\pmb{N}_4^\ast$ has no isolated minimal points whatsoever. 
\begin{figure}[ht!]
\begin{tikzpicture}[yscale=-1]
\draw[white, fill=white] (0,0) circle (.3cm);
\node[circle] at (0,0) (e) {};
\draw[fill = black] (0,0) circle (0.1cm);
\draw[draw = black] (-4,-1.25) circle (0.1cm);
\draw[draw = black] (-2,-1.5) circle (0.1cm);
\draw[draw = black] (0,-1.75) circle (0.1cm);
\draw[draw = black] (2,-1.5) circle (0.1cm);
\node at (0,.5) {$\infty$};
\node at (-4,-1.75) {$0$};
\node at (-2,-2.0) {$1$};
\node at (0,-2.25) {$2$};
\node at (2,-2.0) {$3$};
\node at (4,-1.75) {$\cdots$};
\draw [->,line join=round,decorate, decoration={zigzag, segment length=6, amplitude=.9,post=lineto, post length=2pt}]  (-3.75,-1.2)  -- (e);
\draw [->,line join=round,decorate, decoration={zigzag, segment length=6, amplitude=.9,post=lineto, post length=2pt}]  (3.75,-1.2) -- (e);
\draw [->,line join=round,decorate, decoration={zigzag, segment length=6, amplitude=.9,post=lineto, post length=2pt}]  (-1.8,-1.25) -- (e);
\draw [->,line join=round,decorate, decoration={zigzag, segment length=6, amplitude=.9,post=lineto, post length=2pt}]  (1.8,-1.25)  -- (e);
\draw [->,line join=round,decorate, decoration={zigzag, segment length=6, amplitude=.9,post=lineto, post length=2pt}]  (0.0,-1.4)  -- (e);
\end{tikzpicture}\caption{The spectral space associated to the Priestley space $\pmb{N}_4^\ast = (\N^*, \leqslant_4)$. The solid bullet represents a closed point, while the white circles indicate the open points.}\label{fig:specZinv}
\end{figure}
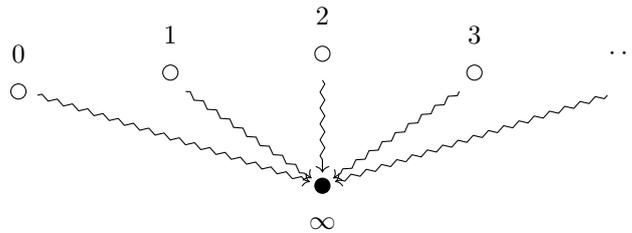
\end{example}

\section{Dispersions of Priestley spaces}\label{sec:dispersions}

In this section we will introduce the theory of \emph{dispersions} on Priestley spaces. These are certain exhaustive filtrations of Priestley spaces such that the corresponding strata have a particularly nice form. The definition of a dispersion is inspired by Stevenson's spectral dimension functions on the corresponding spectral spaces~\cite{stevenson_localglobal}. We will return to this observation in more detail in \cref{rem:stevenson}. 

\subsection{Dispersions}\label{ssec:dispersions}

Let us begin with defining a certain motivating filtration (which we will eventually see it is the universal such). In what follows we write $\Ord$ for the class of ordinals, and we remind the reader that for a Priestley space $\pmb{P} = (P, \leqslant)$ the isolated minimal points are in correspondence with the Thomason points of the corresponding spectral space (see \cref{rem:thomasoninpriest}).

\begin{definition}\label{defn:cbrank}
Let $\pmb{P} = (P, \leqslant)$ be a Priestley space, and write $\delta_T\pmb{P}$ for the set of points which are not isolated and minimal in $\pmb{P}$. The set $\delta_T \pmb{P}$ is called the \emph{Thomason derivative} of $\pmb{P}$. For $\lambda \in \Ord$ we define $\delta^{\lambda}_T \pmb{P}$ as a subset of $\pmb{P}$ via transfinite recursion on $\lambda$:
    \begin{itemize}
        \item $\delta^0_T \pmb{P} = \pmb{P}$;
        \item $\delta^{\lambda + 1}_T \pmb{P} = \delta_T \delta^{\lambda}_T \pmb{P}$;
        \item If $\lambda$ is a limit ordinal then $\delta^{\lambda}_T \pmb{P} = \cap_{\nu < \lambda} \delta^{\nu}_T \pmb{P}$.
    \end{itemize}
Moreover, we set $\delta_T^{\infty} \pmb{P} = \bigcap_{\lambda \in \Ord}\delta_T^{\lambda} \pmb{P}$. For a point $p \in \pmb{P}$ we then define the \emph{Thomason height} of $p$ as follows:
\[
{\height}_{\pmb{P}}(p) = 
\begin{cases}
    \infty & \text{if } p \in \delta^{\infty}_T\pmb{P} \\
    \sup(\lambda \in \Ord \mid p \in \delta_T^{\lambda}\pmb{P}) & \text{otherwise}.
\end{cases}
\]
If the ambient space is clear from context, we also write $\height(p)$ for $\height_{\pmb{P}}(p)$. Finally, the \emph{Thomason height} of a non-empty Priestley space $\pmb{P}$, denoted $\height(\pmb{P})$, is the supremum of all $\lambda \in \Ord$ with $\delta_T^{\lambda}\pmb{P} \neq \varnothing$ or $\infty$ if the supremum does not exist. We will drop the subscript $T$ when it is clear from context.
\end{definition}

\begin{example}\label{ex:cantorbendixsonisthomason}
    In the case that $\pmb{P} = (P, =)$ is a Priestley space with a trivial ordering then the Thomason filtration as defined above coincides with the \emph{Cantor--Bendixson filtration} of $\pmb{P}$. In short, this is the filtration on $P$ which removes isolated points at each stage: for the trivial ordering, the Thomason filtration coincides with the Cantor--Bendixson filtration. For general spectral orders removing Thomason points of a spectral space $X$ corresponds to removing isolated points of $X_{\inv}$.
\end{example}

\cref{defn:cbrank} provides us with filtrations for all Priestley spaces. However, we are only going to be interested in those filtrations which are \emph{exhaustive}, that is, which do not attain height $\infty$. We abstract the features of exhaustive Thomason filtrations in the following definition of a \emph{dispersion}.

\begin{definition}\label{defn:dispersion}
Let $\pmb{P} = (P, \leqslant)$ be a Priestley space. A \emph{dispersion} on $\pmb{P}$ (if it exists) is a function $\chi \colon \pmb{P} \to \Ord$ such that for all $p,q \in P$:
\begin{enumerate}[label=(\arabic*)]
    \item if $p < q$, then $\chi(p) < \chi(q)$;\label{dispersion1}
    \item for any closed subset $S \subseteq P$, if $a \in S$ is a non-isolated point then there exists $s \in S$ such that $\chi(s) < \chi(a)$.\label{dispersion2}
\end{enumerate}
If such a dispersion exists we shall say that $\pmb{P} = (P, \leqslant)$ is \emph{dispersible}. If $n = \max_{p \in P}(\chi(p)) \in \N$, then we say that $n$ is the \emph{height} of the dispersion and that $\pmb{P}$ is \emph{finitely dispersible}.
\end{definition}

\begin{definition}\label{defn:prismaticspacefiltration}
Let $\pmb{P} = (P, \leqslant)$ be a dispersible Priestley space with associated dispersion $\chi \colon \pmb{P} \to \Ord$. Given $\lambda \in \Ord$, we write $\pmb{P}_{\lambda} = \chi^{-1}(\{ \lambda \})$ for the corresponding \emph{stratum}. The corresponding filtration is then defined as
\[
\pmb{P}_{< \lambda} = \bigcup_{\beta < \lambda}\pmb{P}_{\beta} = \chi^{-1}(\{\beta \in\Ord\mid \beta < \lambda\}) \quad \text{and} \quad \pmb{P}_{\geqslant \lambda} = (\pmb{P}_{< \lambda})^c.
\]
\end{definition}

A dispersion gives rise to an ascending filtration by open subsets of the spectral space associated with $\pmb{P}$, a particularly simple form of a topological stratification. 

\begin{proposition}\label{prop:dispersionsbasicproperties}
With notation as above, for every $\lambda \in \Ord$ we have:
    \begin{enumerate}[label=(\arabic*)]
        \item $\pmb{P}_{< \lambda}$ is an open down-set in $\pmb{P}$;
        \item $\pmb{P}_{\geqslant \lambda}$ is closed up-set in $\pmb{P}$;
        \item $\pmb{P}_{\lambda}$ consists of isolated points in $\pmb{P}_{\geqslant \lambda}$.\label{dispersionsbasicproperties:3}
    \end{enumerate}
\end{proposition}
\begin{proof}
    First of all, by construction, $\pmb{P}_{<\lambda}$ is the complement of $\pmb{P}_{\geqslant \lambda}$ in $\pmb{P}$, so Statements (1) and (2) are equivalent. Next, we will prove that $(2)$ implies $(3)$ for any given $\lambda \in \Ord$. Indeed, consider some $p \in \pmb{P}_{\lambda}$ and assume $\pmb{P}_{\geqslant \lambda}$ is a closed up-set in $\pmb{P}$. We wish to show that $p \in \pmb{P}_{\geqslant \lambda}$ is isolated. If $p$ is not isolated then there are two potential situations:
        \begin{enumerate}[label=(\roman*)]
            \item $p$ is not minimal with respect to the ordering $\leqslant $ in $\pmb{P}_{\geqslant \lambda}$.\label{item:op1}
            \item $p$ is not topological isolated in the underlying space of $\pmb{P}_{\geqslant \lambda}$.\label{item:op2}
        \end{enumerate}
    If \ref{item:op1} fails then it gives the existence of some $q \in \pmb{P}_{\geqslant \lambda}$ with $q < p$. By \cref{defn:dispersion}\ref{dispersion1} this implies that $\chi(q) < \chi(p)$. But $\chi(p) = {\lambda}$, and as such $\chi(q) < \lambda$. In particular, $q \not\in \pmb{P}_{\geqslant \lambda}$, a contradiction. Similarly, if \ref{item:op2} fails, then we appeal to  \cref{defn:dispersion}\ref{dispersion2} which implies the existence of some $s \in \pmb{P}_{\geqslant \lambda}$ with $\chi(s) < \chi(p)$, and again, this forces $\chi(s) < \lambda$ and thus $s \notin \pmb{P}_{\geqslant \lambda}$, a contradiction. Therefore, $p$ is isolated in $\pmb{P}_{\geqslant \lambda}$ and (3) holds for the given $\lambda$.

    In order to finish the proof, it now suffices to verify by transfinite induction on $\lambda \in \Ord$ that $\pmb{P}_{\geqslant \lambda}$ is a closed up-set in $\pmb{P}$. The base induction holds because $\pmb{P}_{\geqslant 0} = \pmb{P}$. Suppose first that $\lambda$ is a successor ordinal. By induction hypothesis, we know that $\pmb{P}_{< \lambda-1}$ is an open down-set in $\pmb{P}$ and that $\pmb{P}_{\lambda-1}$ is an open down-set in $\pmb{P}_{\geqslant \lambda-1}= (\pmb{P}_{< \lambda-1})^c$. This implies that $\pmb{P}_{< \lambda} = \pmb{P}_{< \lambda-1} \cup \pmb{P}_{\lambda-1}$ is an open down-set of $\pmb{P}$, hence the complement $\pmb{P}_{\geqslant \lambda}$ is a closed up-set in $\pmb{P}$. Finally, if $\lambda$ is a limiting ordinal, then $\pmb{P}_{\geqslant \lambda} = \cap_{\mu < \lambda} \pmb{P}_{\geqslant \mu}$. As an intersection of closed up-sets in $\pmb{P}$, it is then itself a closed up-set, as desired. 
\end{proof}

\begin{corollary}\label{cor:strataissub}
Let $\lambda \in \Ord$. Then $\pmb{P}_{\geqslant \lambda}$  is a Priestley subspace of $\pmb{P}$.
\end{corollary}

\begin{proof}
As $\pmb{P}_{\geqslant \lambda}$ is a closed up-set in $\pmb{P}$, this follows from \cref{lem:subpriestley}.
\end{proof}

We stated that the definition of a dispersion was designed to abstract the features of the Thomason filtration on a Priestley space. The next result shows that this is indeed the case.

\begin{lemma}\label{lem:thomasinfiltisdisp}
    Let $\pmb{P} = (P, \leqslant)$ be a Priestley space. Then the Thomason filtration of \cref{defn:cbrank} satisfies \ref{dispersion1} and \ref{dispersion2} of \cref{defn:dispersion}. As such, if the Thomason filtration on $\pmb{P}$ is exhaustive then it is a dispersion on $\pmb{P}$.
\end{lemma}

\begin{proof}
    We suppose that there are two points $p < q$ in $P$ such that $\height(p) \geqslant \height(q)$. Then by definition $q$ is minimal and isolated in $\pmb{P}_{\geqslant \height(q)}$, but $p$ is simultaneously less than $q$ and an element of $\pmb{P}_{\geqslant \height(q)}$, a contradiction. As such, the Thomason filtration satisfies \cref{defn:dispersion}\ref{dispersion1}.

    Now, assume that $S \subseteq P$ is a closed subset with $a \in S$ non-isolated but with no $s \in S$ such that $\height(s) < \height(a)$. By assumption there is a convergent sequence $s_i \to a$ in $S$. Consider the subspace $S \cap \pmb{P}_{\geqslant \height(a)}$. This subset is closed in $P$. As $a \in S \cap \pmb{P}_{\geqslant \height(a)}$ is isolated, it follows that infinitely many of the $s_i$ are not in this intersection. In particular, there exists some $s_i$ such that $\height(s_i) < \height(a)$ as required. 

    If the Thomason filtration is exhaustive then it defines a dispersion.
\end{proof}

\begin{remark}
    The proof of \cref{lem:thomasinfiltisdisp} shows us that we can strengthen \cref{defn:dispersion}\ref{dispersion2} to ask for the existence of infinitely many $s \in S$ with $\chi(s) < \chi(a)$.
\end{remark}

\begin{remark}\label{rem:thomasonisuniversal}
While any given Priestley space may admit many dispersions, the Thomason filtration provides the coarsest such dispersion, since it simply proceeds by the greedy algorithm. Thus, for any dispersion $\chi$ we will have $\chi(p) \geqslant \height(p)$ for any $p \in \pmb{P}$.    
\end{remark}

The following result will be an essential ingredient in our prismatic machinery. The corresponding statement for topological spaces can be found, for example in \cite[Proposition 8.1.1]{PPbook}. We also remark that the following is not an if and only if statement; for a counter-example consider the Priestley space $\pmb{C} = (C, =)$ where $C$ is the Cantor space.

\begin{lemma}\label{prop:dispersibleimpliestd}
 Let $\pmb{P} = (P, \leqslant)$ be a dispersible Priestley space, then every point $p \in \pmb{P}$ can be written as the intersection of an open down-set and a closed up-set. We can moreover take the closed up-set to be the upwards closure of $p$, denoted $\bigvee(p)$.
\end{lemma}
\begin{proof}
Consider $p \in \pmb{P}$ and the corresponding closed up-set $\bigvee(p)$. By assumption $\bigvee(p)$ contains a Thomason point, say $q$. We prove that $p=q$. Assume not, then $p \not\in \{q\} = \bigvee(p) \cap \pmb{U}$ for $\pmb{U}$ some open down-set. Then $p \in \pmb{P} \setminus \pmb{U}$ and thus $\bigvee(p) \subseteq  \pmb{P} \setminus \pmb{U}$ and as such $\bigvee(p) \cap \pmb{U} = \varnothing$, a contradiction.  Hence $p$ is Thomason in $\bigvee(p)$ and the result follows.
\end{proof}

This property on points is one that we will constantly need to refer to. As such we shall provide it with a definition.

\begin{definition}\label{defn:locallyclosed}
Let  Let $\pmb{P} = (P, \leqslant)$ be a Priestley space. A point $p \in \pmb{P}$ is \emph{weakly visible} if $p = \bigvee(p)\cap \pmb{U}$ for some  open down-set $\pmb{U}$ in $\pmb{P}$.
\end{definition}

We note that a point in a spectral space $X$ is weakly visible in the sense of \cite{BHS2023} if and only if the point is weakly visible in the associated Priestley space.

\begin{remark}\label{rem:stevenson}
    As remarked in the introduction of this section, the definition of dispersion is inspired by the spectral dimension functions of Stevenson \cite{stevenson_localglobal}. A \emph{spectral dimension function} on a spectral space $X$ is a function $\dim \colon X \to \Ord$ such that:
    \begin{enumerate}
        \item $\dim$ does not take limit ordinal values;
        \item for $x,y \in X$ if $x \in \{y\}$ then $\dim(x) \leqslant \dim(y)$ with equality if and only if $x=y$.
        \item the subset $X_{\leqslant \alpha}$ of those points $x$ with $\dim(x) \leqslant \alpha$ is a Thomason subset of $x$.
    \end{enumerate}
    If we compare this to the definition of a dispersion (after suitably translating the above description of spectral dimension functions to the realm of Priestley spaces) we see that the difference is that for a dispersion the  strata $\pmb{P}_\lambda$ are discrete (\cref{prop:dispersionsbasicproperties}\ref{dispersionsbasicproperties:3}), whereas in a spectral dimension function the strata need  not be discrete. We remark, however, that all examples appearing in \cite{stevenson_localglobal} do have this additional discreteness property.
\end{remark}

\subsection{The examples}

Let us return to the examples that we introduced in \cref{sec:priestley}, and discuss their dispersions. In light of \cref{rem:thomasonisuniversal} we will just discuss the Thomason dispersion of the spaces (provided that it exists).

\begin{example}[\cref{ex:priestleyonepointcpt} continued]\label{ex:nstartisdisp}
We consider the space $(\mathbb{N}^*, =)$. In this case the Thomason dispersion splits the space into the isolated points of $\mathbb{N}^*$ and the accumulation point $\{ \infty \}$:
\[
 (\mathbb{N}^*, =)_0 = \{n\mid n\in \N\} \qquad \text{and} \qquad   (\N^*, =)_1  = \{\infty\}.
\]
In particular, $(\mathbb{N}^*, =)$ is dispersible of height 1. 
\end{example}

\begin{example}[\cref{ex:specz} continued]
In the case of $(\mathbb{N}^*, \leqslant_{2})$, the points $n \in \N$ are all isolated and minimal with respect to $\leqslant$. As such, similar to the case of $(\mathbb{N}^*, =)$ we will have a dispersion of height one given as
\[
 (\mathbb{N}^*, \leqslant_{2})_0 = \{n\mid n\in \N\} \qquad \text{and} \qquad   (\N^*, \leqslant_{2})_1  = \{\infty\}.
\]
The similarity between this dispersion and the one in \cref{ex:nstartisdisp} is exactly what motivated this study. The only difference between the two is the way in which the strata are glued in the reassembly process that we will eventually see.
\end{example}

\begin{example}[\cref{ex:nbar} continued]
    We consider the Priestley space $(\mathbb{N}^*, \leqslant_{3})$. We note that the only minimal element with respect to $\leqslant_{3}$ is the point $\{\infty\}$. However, $\{\infty\}$ is not isolated in $\N^*$. As such, there are no isolated minimal points, and it follows that this Priestley space is not dispersible.
\end{example}

\begin{example}[\cref{ex:speczinv} continued]
We consider the Priestley space $(\mathbb{N}^*, \leqslant_{4})$. We once again note that we have no isolated minimal points at all. Indeed, $\infty$ is the only minimal point with respect to $\leqslant_{4}$, but it is not isolated in $\mathbb{N}^*$. As such $(\mathbb{N}^*, \leqslant_{4})$ cannot be dispersible.
\end{example}

\subsection{Amenable Priestley spaces}\label{subsec:amenablepriestley}
For an arbitrary Priestley space $\pmb{P} = (P, \leqslant)$, there are two natural candidate dispersions:
\begin{enumerate}
	\item One can take the Thomason height of $\pmb{P}$ as defined above.
	\item One can take the Thomason height of the underlying Stone space $P$ or, in other words, the Cantor--Bendixson filtration on $P$ (see \cref{ex:cantorbendixsonisthomason}). This is the Thomason height on the \emph{trivialization of $\pmb{P}$}, denoted $\pmb{P}_{=}$, which we define as $(P, =)$.
\end{enumerate}

In general, these two heights will be wildly different. Indeed, consider the Priestley space of \cref{ex:speczinv}. In this case there is no Thomason dispersion on $\pmb{P}$, but the underlying Stone space of $\N^*$ admits a height one dispersion. However, there are  exceptional cases where the strata of these two dispersions coincide. This happens, for example, in \cref{ex:specz}, and  in our main example of interest in \cref{part:gspectra}. We now provide a formal definition of amenable Priestley spaces for future reference.

\begin{definition}\label{defn:anemabile}
Let $\pmb{P} = (P, \leqslant)$ be a Priestley space. Then we say that $\pmb{P}$ is \emph{amenable} if the strata of the Thomason dispersions (when they exist) on $\pmb{P}$ and $\pmb{P}_{=}$ coincide.
\end{definition}

\begin{remark}
    Given \cref{ex:cantorbendixsonisthomason}, we can reword the property of being amenable by saying that the Thomason dispersion of $\pmb{P} = (P, \leqslant)$ coincides with the Cantor--Bendixson filtration of the underlying Stone space $P$.
\end{remark}

Consequently, for an amenable Priestley space $\pmb{P} = (P,\ \leqslant)$ we see that the topology of $P$ alone determines the properties of the Thomason dispersion.

\section{Priestley spaces in tensor-triangular geometry}\label{sec:priestleyintt}

In the previous sections we have explored the abstract theory of Priestley spaces and dispersions on them. We will now bring the theory of tensor-triangular geometry in the sense of \cite{balmer_spectrum} into the picture. Throughout we will assume the existence of a good model for our tensor-triangulated category so that we can, among other things, discuss homotopy (co)limits. Our enhancement of choice will be via the theory of symmetric monoidal stable $\infty$-categories as developed by Lurie \cite{htt,ha}. In \cref{sec:obstrat} we will want to focus our attention on those tensor-triangulated categories which are rigidly-compactly generated and for these we pick a suitable model in commutative algebras in $\Cat_\infty^\omega$, the $\infty$-category of compactly generated $\infty$-categories. By \cite[Proposition 5.5.7.8]{ha} $\Cat_\infty^\omega$ is equivalent to the $\infty$-category of idempotent complete essentially small $\infty$-categories via passage to the full subcategory of compact objects. The inverse equivalence is provided by passage to ind-objects. 

\subsection{Tensor-triangular geometry}\label{ssec:ttgeometry}

We begin with a rapid review of some basic notions from tensor-triangular geometry, as introduced by Balmer \cite{balmer_spectrum}, and translated into the language of Priestley spaces.

\begin{definition}
Let $\sfK$ be an essentially small tensor-triangulated category. As a set, the \emph{Balmer spectrum} $\Spc(\sfK)$ consists of the prime, thick tensor-ideals of $\sfK$. That is, those proper thick subcategories closed under tensoring with arbitrary objects of $\sfK$ such that if $a \otimes b \in \mathsf{P}$ then $a \in \mathsf{P}$ or $b \in \mathsf{P}$.

For $x \in \sfK$ the support of $x$ is defined as
\[
\supp(x) = \{\mathsf{P} \in \Spc(\sfK) \mid x \not\in \mathsf{P} \}.
\]
As $x$ ranges through the objects of $\sfK$, hese give the basis of closed subsets for the Zariski topology on $\Spc(\sfK)$. 
\end{definition}

The main observation to make is that for any such $\sfK$, the space $\Spc(\sfK)$ is in fact a spectral space. In line with the perspective taken here, we now wish to use the presentation of the Balmer spectrum as a Priestley space. 

\begin{definition}
Let $\sfK$ be an essentially small tensor-triangulated category. The \emph{prism}\footnote{Besides capturing the intuitive geometric interpretation of the concept, the term is an abbreviation for Priestley spectrum.} of $\sfK$ is the Priestley space associated to the Balmer spectrum:
    \[
        \Prism(\sfK) \coloneqq \Pries(\Spc(\sfK)).
    \]
\end{definition}

One may check that for $\mathsf{P}$ and $\mathsf{Q}$ that $\mathsf{P} \leftsquigarrow \mathsf{Q}$ if and only if $\mathsf{P} \subseteq \mathsf{Q}$. Consequently, when looking at the corresponding Priestley space, according to \cref{conv:ordering}, we have $\mathsf{P} \leqslant \mathsf{Q}$ if and only if $\mathsf{P} \subseteq \mathsf{Q}$. This allows us to recast the main classification result of Balmer in terms of prisms:

\begin{proposition}
Let $\sfK$ be an essentially small tensor-triangulated category. Then there is a bijection facilitated by the theory of supports:
\[
\{\text{Thick tensor-ideals of } \sfK \} \overset{\sim}{\longleftrightarrow} \{\text{Open down-sets of } \Prism(\sfK) \}. 
\]
\end{proposition}

We now introduce the main conceptual definition of this section.

\begin{definition}\label{def:ttdispersible}
Let $\sfK$ be an essentially small tensor-triangulated category. 
    \begin{enumerate}
        \item We say that $\sfK$ is \emph{dispersible} if $\Prism(\sfK)$ is a dispersible Priestley space.
        \item If the dispersion is of finite height $n$, we say that $\sfK$ is \emph{finitely dispersible}.
        \item If $\Prism(\sfK)$ is amenable in the sense of \cref{defn:anemabile}, we will say that $\sfK$ is \emph{amenable}.
    \end{enumerate}
If $\sfK$ arises as the full subcategory $\sfT^{\omega}$ of compact objects in a rigidly-compactly generated $tt$-category $\sfT$, we will apply the same terminology as introduced above also to $\sfT$. In particular, we will refer to a pair $(\sfT,\chi)$ consisting of a rigidly-compactly generated tt-category $\sfT$ and a dispersion $\chi$ on $\Prism(\sfT^{\omega})$ as a \emph{dispersible tt-category}.
\end{definition}

To conclude this section, we will outline a dictionary between language regarding pointset topological definitions for Balmer spectra and the corresponding statement in the prism. We recall that a point of a spectral space is \emph{weakly visible} if it can be written as $x = V \cap W^c$ where $V,W$ are Thomason subsets.

\begin{dictionary}\label{tab:dictionary}
Let $\sfK$ be an essentially small tensor-triangulated category with Balmer spectrum $\Spc(\sfK)$ and prism $\Prism(\sfK)$.
	\begin{itemize}
		\item A subset $V \subseteq \Spc(\sfK)$ is Thomason if and only if it is an open down-set of  $\Prism(\sfK)$.
		\item A point $x \in \Spc(\sfK)$ is Thomason if and only if it is isolated and minimal in $\Prism(\sfK)$.
		\item A point $x \in \Spc(\sfK)$ is weakly visible if and only if it can be written as the intersection of an open down-set and a closed up-set in $\Prism(\sfK)$. That is, it is weakly visible in $\Prism(\sfK)$ (cf. \cref{defn:locallyclosed}).
	\end{itemize}
\end{dictionary}

\subsection{The examples}

Let us introduce some prominent examples of interest. Our first two examples will be of the form $\sfK = \sfD(R)^\omega$ where $R$ is a commutative ring. In this case, the Hopkins--Neemann--Thomason theorem \cite{neemanchromtower, thomasonclassification} says that the Balmer spectrum of $\sfK$ is homeomorphic to the Zariski spectrum $\Spec(R)$, and we will use this without further remark.

\begin{example}\label{ex:running1}
Let $\sfK = \sfD(\mathbb{Z})^\omega$. Then $\Prism(\sfK)$ is isomorphic to the Priestley space given in \cref{ex:specz}, where the accumulation point corresponds to the tt-prime ideal generated by the finitely generated torsion abelian groups. In particular, $ \sfD(\mathbb{Z})^\omega$ is dispersible of height 1.
\end{example}

\begin{example}\label{ex:running2}
Let $k$ be a field equipped with the discrete topology and $R= C(\N^*, k)$ be the ring of continuous functions from $\N^*$ to $k$. Then $\Prism(\sfD(R)^\omega) \cong (\N^*, =)$ as in \cref{ex:priestleyonepointcpt}. As such, $ \sfD(R)^\omega$ is dispersible of height 1.
\end{example}

\begin{example}
Let $\sfK = \mathsf{SH}^\omega$, the category of finite $p$-local spectra for some prime $p$. Then $\Prism(\sfK)$ is the Priestley space  $(\mathbb{N}^*, \leqslant_3) $ introduced in \cref{ex:nbar}. In particular $\mathsf{SH}^\omega$ is not dispersible.
\end{example}

\section{Objectwise local-to-global via dispersions}
\label{sec:obstrat}

We now recall the necessary theory of localizations with respect to certain subsets of the prism. We refer the reader to \cite{balmerfavi_idempotents, bhv} for more details. Recall for $\sfS$ a full subcategory of $\sfT$, the \emph{right-orthogonal} $\sfS^\bot$ is the full subcategory of $\sfT$ on those objects $Y \in \sfT$ such that $\Hom(X,Y) \simeq 0$ for all $X \in \sfS$.

\subsection{Localizations and colocalizations}\label{ssec:localization}

Let $\pmb{U}$ be a open down-set of $\pmb{P}$, i.e., equivalently a Thomason subset of the Balmer spectrum, with  complement $\pmb{U}^c = \pmb{Z}$. We define $\sfT(\pmb{U})_{\tors}$ to be the localizing ideal in $\sfT$ generated by the set of compact objects 
\[\{a \in \sfT^\omega \mid \supp(a) \subseteq \pmb{U}\}\]
and denote the right-orthogonal of $\sfT(\pmb{U})_{\tors}$ by $\sfT(\pmb{Z})$. This admits a further right-orthogonal denoted $\sfT(\pmb{U})_{\comp}$. There are corresponding inclusion functors 
\begin{align*}
\iota_{\tors} &\colon \sfT(\pmb{U})_{\tors} \hookrightarrow \sfT,\\
\iota_{\loc} &\colon \sfT(\pmb{Z}) \hookrightarrow \sfT,\\
\iota_{\comp} &\colon \sfT(\pmb{U})_{\comp} \hookrightarrow \sfT.
\end{align*}
We now recall the relevant results regarding the above \emph{local duality context} formalism from~\cite{bhv}. 

\begin{proposition}\label{prop:localduality}
Let $\sfT$ and $\pmb{U}$ be as above.
\begin{enumerate}
    \item The functor $\iota_{\tors}$ has a right adjoint $\Gamma_{\pmb{U}}$, and the functors $\iota_{\loc}$ and $\iota_{\comp}$ have left adjoints denoted $L_{\pmb{U}}$ and $\Lambda_{\pmb{U}}$ respectively. These induce natural cofibre sequences
    \[\Gamma_{\pmb{U}} X \to X \to L_{\pmb{U}}X \qquad \text{and} \qquad
    V_{\pmb{U}}X \to X \to \Lambda_{\pmb{U}} X\]
    for all $X \in \sfT$. In particular $\Gamma_{\pmb{U}}$ is a colocalization and $L_{\pmb{U}}$ and $\Lambda_{\pmb{U}}$ are localizations.
    \item The functor $L_{\pmb{U}} \colon \sfT \to \sfT(\pmb{Z})$ is a finite and in particular smashing localization, while $\Gamma_{\pmb{U}} \colon \sfT \to \sfT(\pmb{U})_{\tors}$ is a smashing colocalization. In particular, there are natural equivalences 
    \[\Gamma_{\pmb{U}}(X) \simeq X \otimes \Gamma_{\pmb{U}} {\unit} \qquad L_{\pmb{U}}(X) \simeq X \otimes L_{\pmb{U}} {\unit}\]
    for all $X \in \sfT$.
    \item The functors $\Lambda_{\pmb{U}} \iota_{\tors}\colon \sfT(\pmb{U})_{\tors} \to \sfT({\pmb{U}})_{\comp}$ and $\Gamma_{\pmb{U}} \iota_{\comp} \colon \sfT({\pmb{U}})_{\comp} \to \sfT({\pmb{U}})_{\tors}$ are mutually inverse equivalences of stable $\infty$-categories. Moreover, there are natural equivalences of functors
    \[\Lambda_{\pmb{U}} \Gamma_{\pmb{U}} \xrightarrow{\sim} \Lambda_{\pmb{U}} \qquad \Gamma_{\pmb{U}} \xrightarrow{\sim} \Gamma_{\pmb{U}} \Lambda_{\pmb{U}}.\]
    \item When viewed as endofunctors on $\sfT$ via the inclusions, the functors $(\Gamma_{\pmb{U}}, \Lambda_{\pmb{U}})$ form an adjoint pair in that there is a natural equivalence
    \[\hom(\Gamma_{\pmb{U}} X, Y) \simeq \hom(X, \Lambda_{\pmb{U}} X)\]
    for all $X,Y \in \sfT$. In particular we have $\hom(\Gamma_{\pmb{U}} \unit, Y) \simeq \Lambda_{\pmb{U}} Y$.
    \item For every $X \in \sfT$  there is a homotopy bicartesian square
    \[
    \xymatrix{X \ar[r] \ar[d] & \Lambda_{\pmb{U}} X \ar[d] \\ L_{\pmb{U}}X \ar[r] & L_{\pmb{U}} \Lambda_{\pmb{U}} X \,.}
    \]
    whose horizontal and vertical fibres are $V X$ and $\Gamma X$ respectively.
\end{enumerate}
\end{proposition}

This can be reorganized into recollement (\cite{BBD1981}), the definition of which we now recall for the convenience of the reader following \cite[Definition A.8.1]{ha}.

\begin{definition}
    Let $\mathsf{C}$ be a presentable stable $\infty$-category. A \emph{stable recollement} of $\mathsf{C}$ is a diagram
    \[
\xymatrix@C=20mm{
\mathsf{C}_0 \ar@{<-}[r]|{j^\ast} \ar@<-2ex>@{^(->}[r]_{j_\ast}  \ar@<2ex>@{^(->}[r]^{j_!}  & \mathsf{C} \ar@<2ex>@{->}[r]^{i^\ast} \ar@<0ex>@{<-^)}[r]|{i_\ast} \ar@<-2ex>@{->}[r]_{i^!}  & \mathsf{C}_1 
} 
\]
of adjunctions between presentable stable $\infty$-categories such that
\[
\mathrm{im}(j_!) = \mathrm{ker}(i^\ast), \quad \mathrm{im}(i_\ast) = \mathrm{ker}(j^\ast), \quad \mathrm{im}(j_\ast) = \mathrm{ker}(i^!). 
\]
Since all recollements appearing in this paper are stable, we will drop this adjective from now on.
\end{definition}

The data of a local duality context is equivalent to a recollement; for an explicit reference recording this observation, see \cite[Proposition 2.12]{BHV2}. The recollement involving the torsion subcategory then takes the following form:
    \begin{equation}\label{eq:recollementform}
    \xymatrix@C=20mm{
    \sfT(\pmb{U})_{\tors} \ar@{<-}[r]|-{\Gamma_{\pmb{U}}} \ar@<-2ex>@{^(->}[r]_-{\iota_{\comp}\circ \Lambda_{\pmb{U}}}  \ar@<2ex>@{^(->}[r]^-{\iota_{\tors}}  & \sfT \ar@<2ex>@{->}[r]^-{L_{\pmb{U}}} \ar@<0ex>@{<-^)}[r]|-{\iota_{\loc}} \ar@<-2ex>@{->}[r]_-{V_{\pmb{U}}}  & \sfT(\pmb{Z}) \rlap{.}} 
    \end{equation}
There is a similar and equivalent recollement using the complete subcategory $\sfT(\pmb{U})_{\comp}$ in place of $\sfT(\pmb{U})_{\tors}$.

\begin{definition}\cite{greenlees_axiomatic}\label{defn:tateconstruction}
The \emph{Tate construction} is defined as the lax symmetric monoidal functor $t_{\pmb{U}} = L_{\pmb{U}} \Lambda_{\pmb{U}} \colon \sfT(\pmb{U})_{\tors} \to \sfT(\pmb{Z})$. When convenient we may view $t_{\pmb{U}}$ as an endofunctor on $\sfT$. 
\end{definition}

\begin{remark}\label{rem:warwickduality}
    By \emph{Warwick duality} we have an alternative description of the Tate construction as $t_{\pmb{U}} = L_{\pmb{U}} \Lambda_{\pmb{U}} \simeq \Sigma V_{\pmb{U}} \Gamma_{\pmb{U}}$ (\cite[Corollary 2.5]{greenlees_axiomatic}), reflecting the `complete' counterpart of the recollement \eqref{eq:recollementform}.
\end{remark}

A fact that will play an important role in the following sections is that the Tate construction encodes the entire recollement. 

\begin{proposition}[{\cite[Remark A.8.12]{ha}}]\label{cor:homotopy-cartesian-height1}
Suppose that $\sfT$ is equipped with an enhancement as a stable symmetric monoidal $\infty$-category and consider an open down-set $\pmb{U}$ of $\Prism(\sfT^\omega)$ with complement $\pmb{U}^c = \pmb{Z}$. Then there is a homotopy pullback
\begin{equation}\label{eq:cospan}
\begin{gathered}
\xymatrix{
\sfT \ar[r] \ar[d] & \sfT(\pmb{U}) \ar[d]^{t_{\pmb{U}} } \\ \Fun([1] , \sfT(\pmb{Z})) \ar[r]_-{\pi_1} & \sfT(\pmb{Z}) \,.
}
\end{gathered}
\end{equation}
The equivalence between $\sfT$ and the homotopy pullback of the cospan sends an object $X \in \sfT$ to the diagram
\[
\xymatrix{
& \Lambda_{\pmb{U}}X \ar[d] \\
(L_{\pmb{U}}X \to L_{\pmb{U}} \Lambda_{\pmb{U}}X) \ar[r] & L_{\pmb{U}} \Lambda_{\pmb{U}}X \rlap{.}
}
\]
\end{proposition}

\subsection{Support, strata, and local factors}\label{ssec:support+strata}

One can use the (idempotent functors) $\Gamma_{\pmb{U}}$ and $L_{\pmb{U}}$ to extend the theory of supports from just the compact objects to all objects. The idea is that the image of $\Gamma_{\pmb{U}}$ should ``detect objects supported at ${\pmb{U}}$'' while 
$L_{\pmb{U}}$ should ``detect objects supported away from ${\pmb{U}}$''. As such, if $x$ is weakly visible in $\Prism(\sfT^\omega)$, i.e., we can isolate it as the intersection  ${\pmb{U}} \cap {\pmb{V}}^c$ where $\pmb{U}$ and $\pmb{V}$ are open down-sets, then we can we can detect support ``exactly at $x$'' by considering the image of $\Gamma_{\pmb{U}}L_{\pmb{V}}$; we denote this functor by $\Gamma_x$. We refer to $\Gamma_x \sfT$ and its objects as the \emph{local factors (at $x$)}.

\begin{warning}
It is essential to alert the reader that while this notation is standard in tensor-triangular geometry, it contradicts ordinary usage in commutative algebra where $\Gamma_\wp$ refers to support in primes contained in $\wp$, and where $\Lambda_\wp$ is the derived completion at $\wp$. Similarly, for an abelian group $M$,  $L_{(p)}M=M[1/p]$ (localizing away from $p$) which is of course very different from $M_{(p)}$ (localizing at $p$).
\end{warning}

More generally, for any two open down-sets $\pmb{U},\pmb{V} \subseteq \pmb{P}$  with $ \pmb{X} \coloneqq {\pmb{U}} \cap {\pmb{V}}^c$ we can consider the idempotent functor $\Gamma_{\pmb{X}}= \Gamma_{\pmb{U}}L_{\pmb{V}}$. As we will explain next, the construction of $\Gamma_{\pmb{X}}$ is independent of the ambient prism in an appropriate sense. Indeed, write $\pmb{P} = (P, \leqslant)\coloneqq \Prism(\sfT^\omega)$ for the prism of $\sfT$ and pick $\pmb{V} \subseteq \pmb{Z}$ an open down-set of $\pmb{Z}$. With this set-up we can consider the local category $\sfT(\pmb{Z})$, and inside here consider the idempotent functor $\Gamma_{\pmb{V}}^{\pmb{Z}}$. Alternatively, we can form the down-closure $\pmb{V}'$ of $\pmb{V}$ in the ambient Priestley space $\pmb{P}$. Let $\pmb{V} = \pmb{U}^c \cap \pmb{V}'$ be  the intersection in $\pmb{P}$ and form the idempotent functor $\Gamma_{\pmb{V}}^{\pmb{P}} = \Gamma_{\pmb{V}'}^{\pmb{P}}L_{\pmb{U}}^{\pmb{P}}$ on $\sfT$. We will prove that these two functors $\Gamma_{\pmb{V}}^{\pmb{Z}}$ and $\Gamma_{\pmb{V}}^{\pmb{P}}$, and consequently their essential images, are naturally equivalent. A version of this result was first proven by Stevenson under the additional assumption that $\pmb{P}$ is Noetherian and $\pmb{V}$ is a point:

\begin{proposition}[Stevenson]\label{prop:base_change}
    Let $\sfT$, $\pmb{P}$, $\pmb{U}$, $\pmb{Z} = \pmb{U}^c$, and $\pmb{V}$ be as above. The finite localization functor $L_{\pmb{U}}\colon \sfT \to \sfT(Z)$ induces an equivalence
        \[
            \xymatrix{\Gamma_{\pmb{V}}^{\pmb{P}} \sfT \ar[r]^-{\sim} & \Gamma_{\pmb{V}}^{\pmb{Z}} \sfT(\pmb{Z}).}
        \]
\end{proposition}
\begin{proof}
    This is essentially \cite[Proposition 8.4]{stevenson_13}, and the argument there extends to the case considered here. For convenience, we will sketch the proof. Since all categories involved are full subcategories of $\sfT$, it suffices to verify the claim objectwise on the level of homotopy categories. The projection formula for $L_{\pmb{U}}$ then reduces further to showing that there is an equivalence of idempotents 
        \[  
            L_{\pmb{U}}\Gamma_{\pmb{V}}^{\pmb{P}}\unit \simeq \Gamma_{\pmb{V}}^{\pmb{Z}}\unit.
        \]
    Unwinding the definitions, this follows from the base-change formulas for idempotents established in the generality needed here in \cite[Lemma 1.27 and Remark 1.28]{BHS2023}.
\end{proof}

In order to have a good support theory for big objects, we would like every point of the prism to be weakly visible. We recall from \cref{prop:dispersibleimpliestd} that if the prism is a dispersible Priestley space, then all points are weakly visible, and as such we can discuss support at all points. From now on, we will assume that all points of our prism are weakly visible. The next definition recovers the Balmer--Favi notion of support \cite{balmerfavi_idempotents}, as extended to the weakly visible setting in \cite{BHS2023}:

\begin{definition}\label{def:bigsupport}
Let $\sfT$ be a rigidly-compactly generated tensor-triangulated category and $X \in \sfT$ an arbitrary object. Then we define the \emph{support} of $X$ as the set
\[
\mathrm{Supp}(X) = \{ x \in \Prism(\sfT^\omega) \mid \Gamma_x X \neq 0 \}.
\] 
\end{definition}

Let $\pmb{P} = \Prism(\sfT^{\omega})$ and consider a dispersion $\chi\colon \pmb{P} \to \Ord$ on $\sfT$. The resulting filtration on $\pmb{P}$ of \cref{ssec:dispersions} can then be lifted to the categorical level via support, as follows. For any $\lambda \in \Ord$, define $\sfT_{<\lambda} \coloneqq \sfT(\pmb{P}_{<\lambda})$ as the full subcategory of $\sfT$ of all objects supported in the open down-set $\pmb{P}_{<\lambda}$ of $\pmb{P}$. It then follows from \cref{prop:localduality} that there is a finite localization 
    \begin{equation}\label{eq:finitelocalization}
        \sfT \to \sfT_{\geqslant \lambda} \coloneqq \sfT(\pmb{P}_{\geqslant\lambda})
    \end{equation}
with kernel $\pmb{P}_{<\lambda}$. By construction, this localization functor induces an identification of the prism of $\sfT_{\geqslant \lambda}$ with $\pmb{P}_{\geqslant\lambda}$, i.e., 
    \begin{equation}\label{eq:prism_identification}
       \xymatrix{\Prism(\sfT_{\geqslant \lambda}^{\omega}) \ar[r]^-{\cong} & \pmb{P}_{\geqslant\lambda} \subseteq \pmb{P}.} 
    \end{equation}
Finally, we call the full subcategory $\sfT(\pmb{P}_{\lambda}) = \sfT_{\lambda} \subseteq \sfT_{\geqslant \lambda} \subseteq \sfT$ of all objects supported in $\pmb{P}_{\lambda}$ the \emph{stratum} of $\sfT$ at $\lambda$. Note that, in light of \cref{prop:base_change}, the stratum is independent of the ambient category.

\subsection{The local-to-global principle}\label{ssec:localtoglobal}

Now that we have recalled the theory of support for a rigidly-compactly generated tensor-triangulated category, we can begin to assemble the pieces of our theory in full. If $X \in \sfT$, one may ask when $X$ can be reconstructed from its ``pieces supported at individual primes'', that is, when can $X$ be built from the $\Gamma_x X$. This is exactly the statement of the local-to-global principle, as introduced by Benson--Iyengar--Krause \cite{bik11} and Stevenson \cite{stevenson_13}. 

\begin{definition}
Let $\sfT$ be a rigidly-compactly generated tensor-triangulated category. We say that $\sfT$ satisfies the \emph{local-to-global principle} if for any object $X \in \sfT$, $X \in \Loc^\otimes (\Gamma_x X \mid x \in \mathrm{Supp}(X))$.
\end{definition}

If $\sfT$ satisfies the local-to-global principle, then any object in $\sfT$ can be reconstructed from its local factors. To highlight the importance of this property, we note that the local-to-global principle in conjunction with mininality of the local factors $\Gamma_x \sfT$ is equivalent to the classification of localizing ideals of $\sfT$ via arbitrary subsets of the Balmer spectrum \cite{BHS2023}.

Recall from \cref{rem:stevenson} that the notion of a dispersion gives rise to a \emph{spectral dimension function} on the associated spectral space. The relevance of this comes from the following theorem of Stevenson, which is the first place we make use of the existence of a monoidal model\footnote{Under the additional hypothesis that the spectrum is Noetherian, the local-to-global principle is proven without the assumption of the existence of a model in \cite[Theorem 3.22]{BHS2023}. We suspect that the methods there can be extended to cover the generality of \cref{prop:localtoglobal} as well. However, we remark that Stevenson in fact proves a stronger version of the local-to-global principle for localizing subcategories as opposed to localizing ideals.} for our tensor-triangulated category:

\begin{proposition}[Stevenson]\label{prop:localtoglobal}
Let $\sfT$ be a rigidly-compactly generated tensor-triangulated category admitting a model. If $\sfT$ is dispersible, then $\sfT$ satisfies the local-to-global principle.
\end{proposition}
\begin{proof}
    In \cite{stevenson_localglobal}, Stevenson proves that any rigidly-compactly generated tensor-triangulated category $\sfT$ admitting a model and equipped with a spectral dimension function satisfies the local-to-global principle. The statement of the proposition is then a translation to the context of dispersions, keeping in mind \cref{rem:stevenson}. 
\end{proof}

We conclude this section with a classical example of the local-to-global principle in action.

\begin{example}\label{ex:derivedabelianobj}
Let us consider the example of $\sfD(\Z)$. Recall from \cref{ex:running1} that this admits a height 1 dispersion where the open down-set $\pmb{U}$ consists of the collection of non-zero prime ideals, and $\pmb{Z} = \pmb{U}^c = \{(0)\}$. Applying~\cref{prop:localduality} we obtain the following homotopy cartesian square, allowing us to reconstruct arbitrary objects:
\[
\xymatrix{
M \ar[r] \ar[d] & \prod_p M_p^\wedge \ar[d] \\
\mathbb{Q} \otimes M \ar[r] & \mathbb{Q} \otimes \prod_p M_p^\wedge \rlap{.}
}
\]
We note that this is exactly the Hasse square for abelian groups. In the next section, we will explain how to generalize and categorify this observation. 
\end{example}

\section{The categorical local-to-global principle I: lax limits}\label{sec:catstrat_lax}

In the previous section we saw that if we have a rigidly-compactly generated tensor-triangulated category $\sfT$ which is dispersible, then $\sfT$ satisfies the local-to-global principle. This is a  statement regarding objects. In the final two sections of this part we wish to discuss how to lift this to a \emph{categorical local-to-global principle}, that is, a statement regarding reconstructing the tensor-triangulated category $\sfT$ from its local factors $\Gamma_{x} \sfT$. In fact, we will deconstruct our dispersible tt-category $\sfT$ in three increasingly explicit ways, namely as 
    \begin{itemize}
        \item an iterated recollement (\cref{ssec:iteratedrecoll}),
        \item a lax limit of a lax categorical diagram (\cref{ssec:laxdiagram_reconstruction}), and
        \item a strict limit of a categorical cubical diagram (\cref{sec:adelicification}),
    \end{itemize}
working naturally from the abstract to more concrete descriptions. An outline of our approach is given in \cref{ssec:overview}. We reiterate the following convention.

\begin{convention}\label{con:rcgentt}
We will assume throughout this section and the next that our rigidly-compactly generated tensor-triangulated category is equipped with an enhancement as a stable symmetric monoidal $\infty$-category. 
\end{convention}

In this section and the next, we work within the $(\infty,2)$-category of $\infty$-categories. More precisely, we will consider lax directed diagrams in the $(\infty,2)$-category $\stable2Cat$ of (not necessarily small) stable $\infty$-categories, exact functors, and natural transformations between them. We refer to \cite[Appendix B]{AMGR} for the relevant material from the theory of $(\infty,2)$-categories. Moreover, the $(\infty,1)$-category of (not necessarily small) stable $\infty$-categories and exact functors will be denoted by $\stableCat$; it may be regarded as the $(\infty,2)$-subcategory of $\stable2Cat$ in which all natural transformations are invertible.

\subsection{Overview}\label{ssec:overview}

Let $\sfT$ be a dispersible tt-category with dispersion $\chi\colon \pmb{P} \coloneqq \Prism(\sfT^{\omega}) \to [n]$ of height $n$. We have seen in \cref{prop:dispersionsbasicproperties} that there is an associated stratification of the prism $(\pmb{P}_{<k})$ by open down-sets in $\pmb{P}$ and with strata denoted by $\pmb{P}_{k}$, for $k \in [0,n]$. Our goal in this section and the next is to construct a corresponding categorical decomposition of $\sfT$.

The idea is simple enough: first, set $\sfT_{\geqslant 0} = \sfT$. We then proceed by induction, where at each stage $k \in [0,n]$ of the filtration, we form quotient of $\sfT_{\geqslant k}$ by the full subcategory spanned by those objects supported in $\pmb{P}_{k}$. By assumption, this process terminates at a finite stage and thus provides a decomposition of $\sfT$ in terms of a sequence of recollements. In a second step, we wish to unfold this iterative description into a closed description of $\sfT$. Our approach relies on the concept of left-lax diagrams reviewed in more detail in \cref{ssec:laxdiagrams+limits} below. The main purpose of this section is to explain how to extract from the given filtration $(\pmb{P}_{<k})$ a left-lax diagram $\cL_{(\sfT,\chi)}\colon [n] \to \Cat_{\infty}$ whose right-lax limit recovers $\sfT$. However, due to the $(\infty,2)$-categorical nature of this approach, it can be difficult to work with in practice. Therefore, in a final step carried out in \cref{sec:catstrat_strict}, we show that any such left-lax diagram may be rewritten as a punctured $(n+1)$-dimensional cube $\cD_{(\sfT,\chi)}$ whose (strict) homotopy limit is equivalent to $\sfT$. 

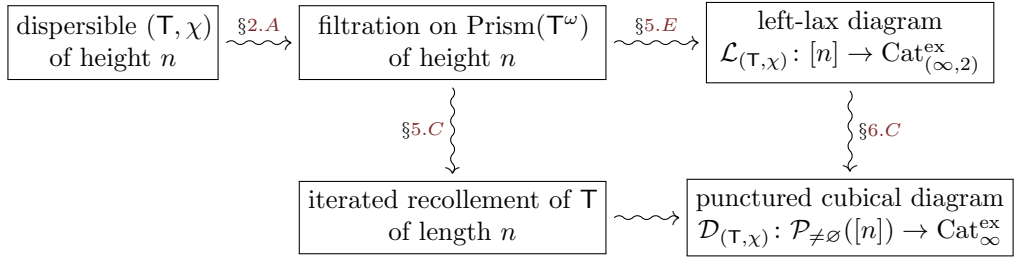
\begin{figure}[ht!]
\centering
\xymatrix@R=3em{
\fbox{\text{\parbox{1in}{\centering dispersible $(\sfT,\chi)$ \\ of height $n$}}}
\ar@{~>}[r]^-{\S\ref{ssec:dispersions}} &
\fbox{\text{\parbox{1.5in}{\centering filtration on $\Prism(\sfT^\omega)$ \\ of height $n$}}}
\ar@{~>}[r]^-{\S\ref{ssec:laxdiagram_reconstruction}} \ar@{~>}[d]_-{\S\ref{ssec:iteratedrecoll}} &
\fbox{\text{\parbox{1.4in}{\centering left-lax diagram \\ $\cL_{(\sfT,\chi)}\colon [n] \to \stable2Cat$}}}
\ar@{~>}[d]^-{\S\ref{sec:adelicification}} \\
& 
\fbox{\text{\parbox{1.5in}{\centering iterated recollement of $\sfT$ \\ of length $n$}}}
\ar@{~>}[r] &
\fbox{\text{\parbox{1.6in}{\centering punctured cubical diagram \\ $\cD_{(\sfT,\chi)}\colon\mathcal{P}_{\neq \varnothing}([n]) \to \stableCat$}}}
}
\caption{Schematic summary of our approach.}\label{fig:schematic}
\end{figure}

\subsection{The categorical local-to-global principle}\label{ssec:catlocaltoglobal}

The main goal of this section is to express a dispersible tt-category $\sfT$ as a suitable lax limit over a lax diagram $\cL_{\sfT,\chi}$ constructed from the dispersion $\chi$ of $\sfT$. The key features of such a decomposition of $\sfT$ are captured abstractly in the categorical local-to-global principle, formulated in \cref{def:catl2g} below. This definition relies on the notion of lax diagrams and lax limits of $\infty$-categories, which we will briefly review in \cref{ssec:laxdiagrams+limits}. In the remainder of this section we will then prove that if a category is dispersible, then it satisfies the categorical local-to-global principle, while moreover describing the reconstruction process iteratively and in a number of examples. 

\begin{definition}\label{def:catl2g}
Let $\sfT$ be a rigidly-compactly generated tensor-triangulated category. Then we say that $\sfT$ satisfies the \emph{categorical local-to-global principle} if $\sfT$ it is equivalent to a right-lax limit over a poset of products of local factors $\{ \Gamma_x \sfT \mid x \in \Prism(\sfT^\omega) \}$.
\end{definition}

We can now formulate the main theorem of this section, with the subsequent subsections providing the ingredients of the proof.

\begin{theorem}\label{thm:catl2g}
    Let $\sfT$ be a dispersible rigidly-compactly generated tensor-triangulated category. Then $\sfT$ satisfies the categorical local-to-global principle.
\end{theorem}

There will be two main ingredients for this proof. In \cref{ssec:laxdiagram_reconstruction} we will first see how one can use iterated right-lax limits to build a dispersible category $\sfT$ out of its strata $\sfT_i$ where $\pmb{P}$ is the prism of $\sfT$ and $\sfT_i=\sfT(\pmb{P}_i)$ consists of objects supported exactly on the stratum $\pmb{P}_i$. Secondly, in \cref{sec:localfactors}, we will prove that these strata can be split into individual local factors, and then assemble the result. 

In \cref{sec:catstrat_strict}, we will describe how to rewrite the aforementioned right-lax limits as strict homotopy limits over a punctured cube. This rewrites the categorical local-to-global principle for $\sfT$ in a more \emph{adelic} flavour. We will also provide two illuminating examples in this section which will be representative of the situations that we will encounter in \cref{part:gspectra}.

\begin{remark}
The reconstruction methods that we discuss here via right-lax limits can also be obtained as a special case of the \emph{macrocosm principle} of Ayala--Mazel--Gee-Rozenblyum \cite{AMGR}.  Explicitly, the results of \cref{sec:obstrat} imply that our filtration $(\sfT_{< i})$ on $\sfT$ is a stratification $[n] \to \sfT$ in the sense of Ayala--Mazel-Gee-Rozenblyum. The results of \cref{ssec:laxdiagram_reconstruction} then follow from an application of the `macrocosm reconstruction theorem' \cite[Theorem 2.5.14]{AMGR} using the fact that any the poset $[n]$ for $n \in \N$ is down-finite. We include a sketch of the proof of the result using a strategy via direct iteration because it clarifies the construction in this case. For the special case of $[n]$, this construction was also considered in the realm of chromatic homotopy theory by the second author in joint work with Antol\'{i}n-Camarena~\cite{fracture_cube}. 
\end{remark}

\begin{remark}\label{rem:moduleapproach}
In this paper we are only considering reconstruction theorems where the building blocks are the local factors $\Gamma_x \sfT$, but this is not the only option available to us. The first and third authors prove the existence of ``adelic models'' of rigidly-compactly generated tensor-triangulated categories where the building blocks appearing in the decomposition are module categories \cite{adelicm}.  Comparisons between these different models in the case that $\sfT$ has Noetherian Balmer spectrum of dimension one can be found in \cite{adelic1}. We will revisit this observation in \cref{subsec:circle}.

We anticipate that it is possible to look at the image of the tensor unit $\unit$ in the categorical local-to-global decomposition and then consider modules over this cubical diagram to obtain an adelic model. We intend to pursue this line of thinking elsewhere. 
\end{remark}

\subsection{Iterated recollements}\label{ssec:iteratedrecoll}

Suppose $(\sfT,\chi)$ is a \emph{dispersible tt-category of height $n$}, i.e., a tt-category $\sfT$ satisfying \cref{con:rcgentt} and equipped with finite dispersion $\chi\colon \pmb{P}\coloneqq \Prism(\sfT^{\omega}) \to [n]$ of height $n$. In order to motivate the constructions that will follow, and to orient our intuition, we begin with a direct approach to building a decomposition of $\sfT$ via iterated recollements. 

Consider the first filtration step $\pmb{P}_{<1} \subseteq \pmb{P}$ with complement $\pmb{P}_{\geqslant 1}$. As explained in \cref{ssec:localization} and \cref{ssec:support+strata}, the finite localization $\sfT \to \sfT_{\geqslant 1} = \sfT(\pmb{P}_{\geqslant 1})$ from \cref{eq:finitelocalization} is part of a recollement which is controlled by the associated Tate construction
    \[
        t_0\colon \sfT_0 \to \sfT_{\geqslant 1}.
    \]
In other words, $\sfT$ can be reconstructed from $\sfT_{\geqslant 1}$ and its zeroth stratum $\sfT_0$ via the recipe of \cref{cor:homotopy-cartesian-height1}. Under the localization, the prism of $\sfT_{\geqslant 1}$ identifies with $\pmb{P}_{\geqslant 1}$ (see \eqref{eq:prism_identification}).  Therefore we would like to continue this process with $\sfT_{\geqslant 1}$ and the dispersion restricted to $\pmb{P}_{\geqslant 1}$. The next lemma says that this is indeed possible, in the expected way:

\begin{lemma}\label{lem:iterative_step}
    Let $(\sfT, \chi)$ be a dispersible tt-category of height $n$ and let $\chi_{\geqslant 1}\colon \pmb{P}_{\geqslant 1} \to [n]$ be the restriction of $\chi$ to $\pmb{P}_{\geqslant 1}$. Then $(\sfT_{\geqslant 1}, \chi_{\geqslant 1})$ is a dispersible tt-category of height $n-1$.
\end{lemma}
\begin{proof}
    As finite localizations may be performed at the level of symmetric monoidal $\infty$-categories, our \cref{con:rcgentt} again holds for $\sfT_{\geqslant 1}$. By \cref{prop:dispersionsbasicproperties}, $\pmb{P}_{\geqslant 1} = \chi^{-1}(\{k \in [n]\mid k \geqslant 1\})$ is a closed up-set of $\pmb{P}$, and it is then straightforward to verify that $\chi_{\geqslant 1}\colon \pmb{P}_{\geqslant 1} = \Prism(\sfT_{\geqslant 1}^{\omega}) \to [n]$ satisfies the conditions of a dispersion (\cref{defn:dispersion}). Appealing to \cref{prop:base_change}, we see that
        \[
            (\Prism(\sfT_{\geqslant 1}^{\omega}))_{\geqslant k} = (\pmb{P}_{\geqslant k+1})
        \]
    for all $k\in[n]$. It follows that $(\sfT_{\geqslant 1}, \chi_{\geqslant 1})$ is dispersible of height $n-1$.
\end{proof}

Since $(\sfT, \chi)$ is of height $n$, this process has to stop after finitely many steps, yielding a decomposition of $\sfT$ into its strata $\sfT_{k}$. Here, we make use of the observation of \cref{prop:base_change} which tells us the $k$-th stratum of $\sfT_{\geqslant1}$ coincides with that of $\sfT$ for all $k\geqslant 1$, and similarly for later filtration steps. 

While fairly transparent, the key drawback of this iterative algorithm is that it is difficult to extract explicit formulas for splicing data required to reconstruct $\sfT$ from its strata. We will therefore formulate a different approach based on lax limits, that encodes the entire decomposition in a single lax diagram of $\infty$-categories.

\subsection{Lax diagrams and limits}\label{ssec:laxdiagrams+limits}

We begin by reviewing some material about lax diagrams of $\infty$-categories and their (lax) limits. 
Full arguments for the theory that we use in this section, in particular pertaining to the $(\infty,2)$-categorical notions, can be found in Appendix A and Appendix B of \cite{AMGR} which guides our exposition here.
The first ingredient will be that of \emph{left-lax functors}  $\mathsf{D} \xrightarrow{F} \stable2Cat$ (see \cite[B.1.7]{AMGR}). As the name suggests, a lax functor is one that only respects composition of 1-morphisms in a lax fashion. That is, whenever there is a non-degenerate composite morphism $d_0 \to d_1 \to d_2$ in $\mathsf{D}$, we have a lax-commutative triangle
\begin{equation}\label{fig:laxtriangle}
\begin{gathered}
\xymatrix@C=0em{
& F(d_1) \ar[dr] \\F(d_0) \ar[ur] \ar[rr] & \ar@{}[u]|<<<<<<{\big\Uparrow} & F(d_2) 
}
\end{gathered}
\end{equation}
satisfying the relevant commutativity and higher coherence data. The \emph{left} in \emph{left-lax} refers to the direction of the natural transformation in \eqref{fig:laxtriangle}.

The second ingredient is that of \emph{right-lax limits}, which \cite[Proposition A.5.2]{AMGR} describes via a universal property. Suppose that $\mathsf{D} \xrightarrow{F} \stable2Cat$ is a left-lax functor of presentable stable $\infty$-categories as defined above. Then the \emph{right-lax limit} $\rlaxlim_{\sfD} F$ of this diagram is the $\infty$-category described by the data:
\begin{itemize}
    \item an object $e_d \in F(d)$ for each $d \in \mathsf{D}$,
    \item For each morphism $d_0 \xrightarrow{\phi} d_1$ a morphism $F(\phi)(e_{d_0}) \leftarrow e_{d_1}$,
\end{itemize}
 along with higher commutativity and coherence data (\cite[Example A.5.3]{AMGR}). If the diagram $\sfD$ is clear from context, we will omit it from the notation. The \emph{right} in \emph{right-lax} refers to the direction of the morphisms $F(\phi)(e_{d_0}) \leftarrow e_{d_1}$.

\begin{remark}\label{defn:rll}
Although we have defined dispersions more generally, allowing ordinal values, all categories that we encounter in this paper will be finitely dispersible, so working with $[n]$ is sufficient for our purposes. As such we will only be concerned with the case that $\mathsf{D} = [n]$. In this case, the above description of a right-lax limit of a left-lax functor $[n] \xrightarrow{F} \stable2Cat$ simplifies to:
\begin{itemize}
    \item an object $e_i \in F(i)$ for each $0 \leqslant i \leqslant n$,
    \item a morphism $F(\phi)(e_{i}) \leftarrow e_{j}$ whenever $0 \leqslant i < j \leqslant n$.
\end{itemize}
 along with higher commutativity and coherence data.
\end{remark}

\begin{example}
    To reconnect to \cref{ssec:iteratedrecoll}, we first need a concrete understanding of right-lax limits over the diagram $[1]$. Let $\sfT$ be as in \cref{con:rcgentt}, and $\pmb{U}$ an open down-set of $\Prism(\sfT^\omega)$ with complement $\pmb{U}^c = \pmb{Z}$. We shall write $\sfT(\pmb{U})$ for $\sfT(\pmb{U})_\mathrm{tors}$ as in \cref{prop:localduality} (which is equivalent to $\sfT(\pmb{U})_\mathrm{comp}$). Then we have a diagram
        \begin{equation}\label{eqn:functorforheight1}
            [1]  \xrightarrow[\qquad]{\cL} \stable2Cat \qquad \qquad (0 \to 1)  \xmapsto[\qquad]{} \left( t_{\pmb{U}} \colon \sfT(\pmb{U}) \to \sfT(\pmb{Z})\right),
        \end{equation}
    which is (vacuously) left-lax. Following \cref{defn:rll}, an object of the right-lax limit of this functor is the data of
        \begin{itemize}
            \item an object $X_0 \in \sfT(\pmb{U})$;
            \item an object $X_1 \in \sfT(\pmb{Z})$;
            \item a morphism $t_{\pmb{U}} (X_0) \leftarrow X_1$.
        \end{itemize}
    along with compatibility and higher homotopies. Through an application of \cref{cor:homotopy-cartesian-height1}, we are able to conclude that the right-lax limit of \eqref{eqn:functorforheight1} is equivalent to $\sfT$. For the details of this argument, see \cite[Section 1.1]{AMGR}.
\end{example}

Using the notation established in \cref{ssec:iteratedrecoll}, we summarize this discussion in the following:

\begin{lemma}\label{lem:homotopy-cartesian-height1-complete}
Let $(\sfT,\chi)$ be dispersible tt-category of height $n$ and write $\pmb{P}$ for the prism of $\sfT$. Then $\sfT$ is equivalent to the right-lax limit of
    \[
    [1]  \xrightarrow[\qquad]{\cL} \stable2Cat \qquad \qquad (0 \to 1)  \xmapsto[\qquad]{} t_{0} \colon \sfT_0 \to \sfT_{\geqslant 1},
    \]
where $t_0$ is the associated Tate construction.
\end{lemma}

The next result describes how to unpack an iteration of left-lax functors, indexed over linear posets $[n]$ where $n \in \N$, into a single linear left-lax diagram. This result will form the backbone for extracting a categorical decomposition from a dispersion. 

\begin{proposition}\label{lemma:iteraterll}
    Let $[1] \xrightarrow{F} \stable2Cat$ and $[n] \xrightarrow{G} \stable2Cat$ be left-lax diagrams of presentable stable $\infty$-categories such that $F(1)$ is equivalent to the right-lax limit of the diagram $G$. Then there exists a left-lax diagram $H\colon [1+n] \xrightarrow{H} \stable2Cat$ obtained by pasting $G$ into $F(1)$, such that
        \[
            \rlaxlim H \simeq \rlaxlim F,
        \]
    i.e., the right-lax limit of $H$ is equivalent to the right-lax limit of $F$.
\end{proposition}
\begin{proof}
    We begin by constructing the left-lax diagram $H$. For the purposes of this proof, let us reindex the left-lax diagram $G$ on $[1,n+1] = \{1,\ldots,n+1\}$. Taking the right-lax limit of $G$ as the cone point, we obtain a canonical extension
        \[
            \xymatrix{[1,n+1] \ar[r]^-G \ar[d] & \stable2Cat \\
            [n+1]\simeq {}^{\triangleleft}[1,n]. \ar@{-->}[ru]_{G'}}
        \]
    By assumption, $G'(0) \simeq F(1)$, so precomposing with $F(0\to 1)$ and contracting the resulting left-lax diagram at $[0]$ yields a left-lax diagram $H\colon [n+1] \to [1] \vee [n+1] \to \stable2Cat$ with
        \[
            H(i) = 
            \begin{cases}
                F(0) & i= 0 \\
                G(i-1) & i>0.
            \end{cases}
        \]
    It thus remains to determine the right-lax limit of $H$. To the end, as in the proof of \cite[Lemma 6.1.12]{AMGR}, we have a pullback diagram in $\stable2Cat$:
        \[
            \xymatrix{\rlaxlim_{[n+1]} H \ar[r] \ar[d] & \rlaxlim_{[0]} H \simeq  F(0) \ar[d] \\
            \Fun([1],\rlaxlim_{[1,n+1]} G) \ar[r] & \rlaxlim_{[1,n+1]} G.}
        \]
    Combining \cref{cor:homotopy-cartesian-height1} and \cref{lem:homotopy-cartesian-height1-complete}, we deduce that 
        \[
            \rlaxlim_{[n+1]} H \simeq \rlaxlim_{[1]}(F(0) \to \rlaxlim_{[1,n+1]} G) \simeq \rlaxlim_{[1]}(F),
        \]
    as desired. 
\end{proof}

\begin{example}
    Let us make the construction of $H$ explicit in the case that $n=1$. We have diagrams $F(0) \xrightarrow{f} F(1)$ and $G(0) \xrightarrow{g} G(1)$. By assumption, $F(1)$ is equivalent to the right-lax limit of $G$, and as such we have the following diagram by universal property:
    \[
    \xymatrix@C=1em{&F(1) \ar[dr]|{\mathsf{pr}_2}\ar[dl]|{\mathsf{pr}_1} & \\ G(0) \ar[rr]_{g} &\ar@{}[u]|<<<<<<{\Leftarrow} \ar@{}[u]|<<<<<<<<{\alpha} & G(1).}
    \]
    Whiskering this with the data of $F$ we obtain the extended diagram
    \begin{equation}\label{eq:whiskered}
    \begin{gathered}
    \xymatrix{&F(0) \ar[d]_{f} \ar@/^/[ddr] \ar@/_/[ddl]\\&F(1) \ar[dr]|{\mathsf{pr}_2}\ar[dl]|{\mathsf{pr}_1} & \\ G(0) \ar[rr]_{g} &\ar@{}[u]|<<<<<<{\Leftarrow} \ar@{}[u]|<<<<<<<<{\alpha} & G(1).}
    \end{gathered}
    \end{equation}
    From this we read off the desired left-lax functor $H \colon [2] \to \stable2Cat$ as
    \begin{equation}\label{eq:Hconstruction}
    \begin{gathered}
    \xymatrix@C=1em{
    & G(0) \ar[dr]^{g} \\ F(0) \ar[ur]^{\mathrm{pr}_1 \circ f} \ar[rr]_{\mathrm{pr}_2 \circ f} & \ar@{}[u]|<<<<<<{\big\Uparrow \eta} & G(1) \rlap{,}
    }
\end{gathered}
\end{equation}
where $\eta$ is the natural transformation obtained from whiskering $\alpha$ with $F$. 
\end{example}

\subsection{Left-lax diagrams from dispersions}\label{ssec:laxdiagram_reconstruction}

Equipped with \cref{lemma:iteraterll} and \cref{prop:base_change} we are now in a position to prove the following result, which forms half the proof of \cref{thm:catl2g}. 

\begin{proposition}\label{prop:iterated}
    Let $(\sfT,\chi)$ be dispersible tt-category of height $n$ and with prism $\pmb{P} = \Prism(\sfT^{\omega})$. Then there is a left-lax diagram $\cL_{(\sfT,\chi)} \colon [n] \to \stable2Cat$ with
    \[
    i  \xmapsto[\qquad]{} \sfT_i, \qquad \qquad (i \to i+1) \xmapsto[\qquad]{\qquad} \left(\Gamma_{\pmb{P}_{i+1}}\circ t_i \colon \sfT_i \to \sfT_{i+1}\right).
    \]
    whose right-lax limit is equivalent to $\sfT$.
\end{proposition}

\begin{proof}
    We argue by induction on the height $n$, the base case $n=1$ holding trivially. For the inductive step with $n\geqslant 1$, we begin by describing the construction of the claimed left-lax diagram. This is done iteratively using the construction outlined in \cref{lemma:iteraterll}. 

    By \cref{lem:iterative_step} and the inductive hypothesis, there exists a left-lax diagram
        \[
            \cL_{(\sfT_{\geqslant 1},\chi_{\geqslant 1})}\colon \pmb{P}_{\geqslant 1} \to \{1,2,\ldots,n\}
        \]
    with $\cL_{(\sfT_{\geqslant 1},\chi_{\geqslant 1})}(i) \simeq \sfT_i$ for all $i\geqslant 1$ and whose right-lax limit is equivalent to $\sfT_{\geqslant 1}$. Moreover, the transition functors from $i \to i+1$ are given by $\Gamma_{\pmb{P}_i}\circ t_i$. Here, we use that, although the strata $\sfT_i$ are formed within $\sfT
    _{\geqslant 1}$, by \cref{prop:base_change} they are equivalent to the strata formed in $\sfT$ itself. Likewise, \cref{lem:homotopy-cartesian-height1-complete} provides a diagram 
        \[
        [1]  \xrightarrow[\qquad]{\cL} \stable2Cat \qquad \qquad (0 \to 1)  \xmapsto[\qquad]{} (t_{0} \colon \sfT_0 \to \sfT_{\geqslant 1}),
        \]
    whose right-lax limit is equivalent to $\sfT$.  
    
    Applying \cref{lemma:iteraterll} with $F = \cL$ and $G = \cL_{(\sfT_{\geqslant 1},\chi_{\geqslant 1})}$ and unwinding the construction, we obtain a left-lax diagram $\cL_{\sfT,\chi} = H\colon [n] \to \stable2Cat$ with the desired properties. 
 \end{proof}

\begin{remark}\label{rem:lax_implicit}
    The construction of \cref{prop:iterated} encodes the approach via iterated recollements from \cref{ssec:iteratedrecoll} in a single left-lax diagram $\cL_{\sfT,\chi}$. However, the description of the transition functors and higher coherences in the diagram is still iterative and thus difficult to write down explicitly. We will return to this point in \cref{sec:catstrat_strict}, where we will unfold this data as a punctured cubical diagram.  
\end{remark}

\subsection{Local factors}\label{sec:localfactors}

In \cref{prop:iterated} we saw that we can decompose a dispersible tensor-triangulated category as a right-lax limit into the strata $\sfT_i$. However, for the categorical local-to-global principle we require a decomposition involving the local factors $\Gamma_x \sfT$.  The goal of this section is to show that the category $\sfT_i$ is equivalent to $\prod_{x \in \pmb{P}_i} \Gamma_x \sfT$ using the properties of a dispersion. Combining this result with \cref{prop:iterated} we will obtain the proof of \cref{thm:catl2g}.

 \begin{lemma}\label{lem:disjointvisible}
 Let $x$ be a minimal isolated point of $\Prism(\sfT^\omega)$. Then there is a compact object $\kappa(x) \in \sfT^\omega$ such that $\supp(\kappa(x)) = \{x\}$. Moreover, if $x$ and $y$ are two such minimal isolated points such that $x \neq y$ then $\kappa(x) \otimes \kappa(y) = 0$.
 \end{lemma}
 
 \begin{proof}
The existence of $\kappa(x)$ follows from \cite[Lemma 7.8]{balmerfavi_idempotents} bearing in mind \cref{tab:dictionary} Now assume that we have points $x \neq y$ . Then
 \[
 \supp(\kappa(x) \otimes \kappa(y)) = \supp (\kappa(x)) \cap \supp (\kappa(y)) = \{x\} \cap \{y\} = \varnothing.
 \]
 As support detects triviality of compact objects, the result follows.
 \end{proof}
 
\begin{proposition}\label{prop:pointwisedecomp}
Let $\sfT$ be a rigidly-compactly generated tensor-triangulated category with prism $\Prism(\sfT^\omega)$. Let $\pmb{U} \subseteq \Prism(\sfT^\omega)$ be a subset consisting of isolated minimal points. Then there is a canonical equivalence
\[
\sfT(\pmb{U}) \simeq   \prod_{x \in \pmb{U}} \Gamma_x \sfT .
\]
\end{proposition}
\begin{proof}
For every $x \in \pmb{U}$ there is a natural factorization of the counit map $\Gamma_x \to \id$ through the counit map of $\Gamma_{\pmb{U}}$. As such, there is a commutative diagram
\[
\xymatrix@C=50pt{\sfT \ar@{->}[r]^-{(\Gamma_x)_{x \in \pmb{U}}} \ar@{<-}[dr]_{\Gamma_{\pmb{U}}} & \displaystyle{\prod_{x \in \pmb{U}}} \Gamma_x \sfT \\ & \sfT(\pmb{U}), \ar[u]_-{\Delta_{\pmb{U}}}}
\]
where the functor $\Delta_{\pmb{U}}$ is the diagonal given by $\Delta_{\pmb{U}}(M) = (\Gamma_x M)_{x \in \pmb{U}}$. 

We wish to show that $\Delta_{\pmb{U}}$ is an equivalence. We shall do so by constructing an explicit inverse. Let $\bigoplus_{\pmb{U}} \colon \prod_{x \in \pmb{U}} \Gamma_x \sfT \to \sfT(\pmb{U})$  be the functor that sends $(M_x)_{x \in \pmb{U}}$ to $\bigoplus_{x \in \pmb{U}} M_x$. Note that this is indeed an object of $\sfT(\pmb{U})$ as torsion objects are closed under direct sums. By construction we have $\Delta_{\pmb{U}}$ is the left adjoint to $\bigoplus_{\pmb{U}}$. We now show that the unit $\eta$ and counit $\mu$ are equivalences.

For $M \in \sfT(\pmb{U})$ the unit map $\eta_M$ is the canonical map
\[
\bigoplus_{\pmb{U}} \Delta_{\pmb{U}}(M) \simeq \bigoplus_{x \in \pmb{U}} \Gamma_x M \to M. 
\]
This map is an equivalence if and only if it is an equivalence after tensoring with $\kappa(y)$ for all $y \in \pmb{U}$. Using \cref{lem:disjointvisible}, and the fact that $\kappa(y)$ is compact  we have
\[
\kappa(y) \otimes \bigoplus_{x \in {\pmb{U}}} \Gamma_x M \simeq \bigoplus_{x \in \pmb{U}} \kappa(y) \otimes \Gamma_x M \simeq \kappa(y) \otimes \Gamma_y M.
\]
As such, $\eta$ is an equivalence. 

Now consider $(M_x)_{x \in \pmb{U}} \in \prod_{x \in \pmb{U}} \Gamma_x \sfT$. For a given $y \in \pmb{U}$, write $\pi_y \colon \prod_{x \in \pmb{U}} \Gamma_x \sfT \to \Gamma_y \sfT$ for the canonical projection functor. It suffices to show that $\pi_y \mu_{(M_x)_{x \in \pmb{U}}}$ is an equivalence for all $y \in \pmb{U}$. By construction, the domain of this map is
\[
\pi_y \Delta_{\pmb{U}} \prod_{x \in \pmb{U}} (M_x)_{x \in \pmb{U}} \simeq \pi_y \prod_{x \in \pmb{U}} M_x \simeq M_y.
\]
The claim then follows.
\end{proof}

We are now in a position to prove \cref{thm:catl2g}.

\begin{proof}[Proof of \cref{thm:catl2g}]
    Let $\sfT$  be a dispersible rigidly-compactly generated tensor-triangulated category with prism $\pmb{P}$. By \cref{prop:iterated} we can write $\sfT$ as a right-lax limit over the left-lax diagram $\cL_{(\sfT,\chi)}$ with $\cL_{(\sfT,\chi)}(i) = \sfT_i$, so it remains to identify the contributions of the strata. To this end, as each $\pmb{P}_i$ consists of isolated minimal points in $\pmb{P}_{\geqslant i}$ (\cref{prop:dispersionsbasicproperties}), we can apply \cref{prop:pointwisedecomp} to decompose $\sfT_i$ as product $\prod_{x \in \pmb{P}_i} \Gamma_x \sfT_{\geqslant i}$ in the category $\sfT_{\geqslant i}$. Finally, since the singletons $\{x\} \subseteq \sfT_{\geqslant i}$ are open and down-closed for all $x \in \sfT_i$, we may apply \cref{prop:base_change} to see that 
        \[
            \sfT_i \simeq \prod_{x \in \pmb{P}_i} \Gamma_x \sfT_{\geqslant i} \simeq \prod_{x \in \pmb{P}_i} \Gamma_x \sfT,
        \]
    as required.
\end{proof}

\section{The categorical local-to-global principle II: strict limits}\label{sec:catstrat_strict}

Let $(\sfT,\chi)$ be a dispersible tt-category of height $n$; we continue to assume \cref{con:rcgentt}. We now shift our focus to the rewriting of the right-lax diagram $\cL_{(\sfT,\chi)}$ appearing in \cref{thm:catl2g} as a (strict) homotopy limit of a punctured cube $\cD_{(\sfT,\chi)}$. This makes the reconstruction process more concrete; from a more abstract perspective, it has the distinct advantage of avoiding $(\infty,2)$-categorical constructions. The contents of this section and the relevant combinatorics can be found in~\cite[Appendix B]{AMGR} for diagrams over general posets. The combinatorics in the particular case of diagrams over $[n]$ (i.e., finite total orders) have previously appeared in work of the second author with Antol\'{i}n-Camarena \cite{fracture_cube}. The idea is that the lax diagram can be rewritten in a strict manner by recording the lax compositions via diagrams in a way that we will make explicit.  Before turning to the abstract result, we demonstrate how the combinatorics look in the height 1 and the height 2 case.

\subsection{The one-dimensional case}\label{ssec:catstrat_1d}

This case is a reformulation of the discussions appearing in \cref{ssec:localization} and \cref{ssec:iteratedrecoll}. A dispersion $\chi\colon \pmb{P}\coloneqq \Prism(\sfT^{\omega}) \to [1]$ of $\sfT$ of height $1$ gives rise a to a left-lax diagram
    \[
        \cL_{(\sfT,\chi)}\colon [1] \to \stable2Cat
    \]
with right-lax limit $\rlaxlim \cL_{(\sfT,\chi)} \simeq \sfT$, see \cref{lem:homotopy-cartesian-height1-complete}. By construction, $\cL_{(\sfT,\chi)}(0) = \sfT_0$ and $\cL_{(\sfT,\chi)}(1) = \sfT_1$, so the diagram $\cL_{(\sfT,\chi)}$ is given by the exact functor
    \[
        t_0\coloneqq \cL_{(\sfT,\chi)}(0 \to 1) \colon \sfT_0 \to \sfT_1. 
    \]
Viewing this as a Tate construction, we may then reformulate, reconstructing $\sfT$ from $t_0$ as a (strict) homotopy limit over a larger diagram. Indeed, by virtue of \cref{cor:homotopy-cartesian-height1}, $\sfT$ is equivalent to the homotopy limit of the following 2-dimensional punctured cube $\cD_{(\sfT,\chi)}$:
    \[
        \begin{gathered}
            \xymatrix{
            & \sfT_0 \ar[d]^{t_0} \\ \Fun([1] , \sfT_1) \ar[r]_-{\pi_1} & \sfT_1 \,.
            }
        \end{gathered}
    \]
Here, $\pi_1$ denotes evaluation at $1$.\footnote{The intuitive picture the reader might wish to keep in mind is that of the computation of homotopy pullbacks in topological spaces via the path-space fibration.} Consequently, the right-lax limit over the left-lax diagram $\cL_{(\sfT,\chi)}$ is equivalent to the  limit of the punctured cubical diagram $\cD_{(\sfT,\chi)}$. We can then go further and decompose the strata $\sfT_i$ according to \cref{prop:pointwisedecomp} into a product of local factors.

\subsection{The two-dimensional case}\label{ssec:catstrat_2d}

We now consider the two-dimensional case in more detail. 

\begin{example}\label{ex:2dimex}
Let $\sfT$ be a rigidly-compactly generated tensor-triangulated category with prism $\pmb{P} = \Prism(\sfT^\omega)$ which admits a dispersion of height two. 
As in \cref{ssec:catstrat_1d}, we see that the first stage of the dispersion gives us a homotopy pullback square
\[
\xymatrix{
\sfT \ar[r] \ar[d] & \sfT_0 \ar[d]^{{t_0}} \\ \Fun([1] , \sfT_{\geqslant 1}) \ar[r]_-{\pi_1} & \sfT_{\geqslant 1} \,.
}
\]
We then observe that we can decompose $\sfT_{\geqslant 1}$ in a similar fashion where we use the fact that $\pmb{P}_{\geqslant 2} = \pmb{P}_2$ by virtue of having a height two dispersion:
\[
\xymatrix{
\sfT_{\geqslant 1} \ar[r] \ar[d] & \sfT_{1} \ar[d]^{t_{1}} \\ \Fun([1] , \sfT_{2}) \ar[r]_-{\pi_1} & \sfT_{2} \,.
}
\]
Now, using \cref{prop:base_change} we are able to replace the instances of  $\sfT_{\geqslant 1}$ in the first diagram by the second diagram. 
The resulting diagram is a three-dimensional punctured cubical diagram $\cD_{\sfT,\chi}$, where the two blue cospans (in the front and the back face) are exactly the pasted in diagrams for $\sfT_{\geqslant 1}$ and $\eta$ is the transformation as constructed in \eqref{eq:Hconstruction}:
\[
\xymatrix{
 && \textcolor{darkBlue}{\Fun([1] {\times} [1] , \sfT_2 )} \ar[dr]^{\pi_1} \ar@[darkBlue][dd]|\hole^<<<<<<<{\textcolor{darkBlue}{\pi_1}} & \\
& \sfT_0 \ar[rr]^<<<<<<<<<<<{\eta} \ar[dd]_<<<<<<<{t_{0}} && \textcolor{darkBlue}{\Fun([1], \sfT_2)} \ar@[darkBlue][dd]^{\textcolor{darkBlue}{\pi_1}} \\
\textcolor{darkBlue}{\textcolor{darkBlue}{\Fun([1], \sfT_1)}}\ar@[darkBlue][rr]|<<<<<<<<<<<<\hole_<<<<<<<<<<<<<<<<<{\textcolor{darkBlue}{t_{1}}} \ar[dr]_{\pi_1} &&  \textcolor{darkBlue}{\Fun([1]  , \sfT_2)} \ar[dr]^{\pi_1} & \\
& \textcolor{darkBlue}{\sfT_1} \ar@[darkBlue][rr]_{\textcolor{darkBlue}{t_{1}}} && \textcolor{darkBlue}{\sfT_2}
}
\]
whose homotopy limit is equivalent to $\sfT$. By \cref{prop:pointwisedecomp}, the strata $\sfT_i$ decompose again as a product of local factors.
\end{example}

\begin{remark}
It is natural to ask for an explicit description of the image of an object $X \in \sfT$ under the equivalence in \cref{ex:2dimex}, i.e., as a homotopy limit of a punctured cubical diagram of local factors. In general, such a description involves restricted products of alternating completions and localizations akin to the construction of the Beilinson--Parshin adeles, see \cite{Beilinson_adeles} or \cite{huber}. This is also analogous to the adelic decomposition of objects given in \cite{adelicm}.

However, in particularly simple cases we can give a concise formula of the image of an object. To this end, suppose that $\pmb{P} = \Prism(\sfT^\omega)$ is the Priestley space $([2], \leqslant)$. That is, we have three primes $0,1,$ and $2$, with $0$ being the minimal element, and we equip it with the obvious dispersion.

We write $\Lambda_{{0}}$ for the completion of $\sfT$ at the open down-set $\pmb{P}_{0}$, and $L_{0}$ for the corresponding localization away from it. Similarly we write $\Lambda_{{1}}$ for the completion of $\sfT_{\geqslant 1}$ at the open down-set $\pmb{P}_{1} \subseteq \pmb{P}_{\geqslant 1}$ and $L_{1}$ for localization away from it.  With this notation in hand, and unravelling definitions, we obtain the following decomposition of $X \in \sfT$:
\[
\xymatrix@C=-1.5em@R=2em{
    & & \mathbb{D}  \ar[dr] \ar[dd]|\hole \\
    & \Lambda_0 X \ar[rr] \ar[dd] && (L_1 \Lambda_0 X \to L_1 \Lambda_1 L_0 \Lambda_0 X) \ar[dd] \\
    (\Lambda_1 L_0 X \to \Lambda_1 L_0 \Lambda_0 X) \ar[rr]|-\hole \ar[dr] && (L_1 \Lambda_1 L_0 X \to L_1 \Lambda_1 L_0 \Lambda_0 X) \ar[dr]  \\
    & \Lambda_1 L_0 \Lambda_0 X  \ar[rr] && L_1 \Lambda_1 L_0 \Lambda_0 X,
}
\]
where $\mathbb{D}$ is the following diagram:
\[
\mathbb{D} \coloneqq \left(\begin{gathered}\xymatrix{
L_1 X \ar[r] \ar[d] & L_1 \Lambda_1 L_0 X \ar[d] \\
L_1 \Lambda_0 X \ar[r] & L_1 \Lambda_1 L_0 \Lambda_0 X \rlap{ { }{ }.}
}\end{gathered} \right) 
\]
\end{remark}

\subsection{The general case}\label{sec:adelicification}

The goal of the following discussion is to provide the general combinatorics for a dispersion of height $n$ that follows from the iterative approach alluded to above to rewrite right-lax limits as homotopy limits over punctured cubes. For a poset $X$, we write $\mathcal{P}(X)$ for the powerset of $X$, and $\mathcal{P}_{\neq \varnothing} (X)$ for the collection of non-empty subsets. When $X$ is the poset $[n]$ then we have that $\mathcal{P}_{\neq \varnothing} ([n])$ is a punctured cube. It is equivalent to the subdivision of $[n]$, denoted $\mathsf{sd}([n])$.

The main ingredient we will need is that of the \emph{isomax undercategory} of some $\varphi \in \mathcal{P}_{\neq \varnothing} ([n])$. These objects are introduced more generally in \cite[Definition A.6.2]{AMGR} for any poset $X$, where they are denoted $\mathsf{sd}(X)_{\varphi / \mathrm{isomax}}$, but we will focus our attention to $X = [n]$ for our required application.

\begin{definition}
Let $\varphi \in \mathcal{P}_{\neq \varnothing} ([n])$. The \emph{isomax undercategory}, denoted $\mathcal{P}_{\neq \varnothing} ([n])_{\varphi}$, is the poset whose elements are the morphisms $\varphi \to \psi$ in $\mathcal{P}_{\neq \varnothing} ([n])$ such that $\mathrm{max}(\varphi) = \mathrm{max}(\psi)$ with the poset structure induced from that of $\mathcal{P}_{\neq \varnothing} ([n])$.
\end{definition}

\begin{example}\label{ex:diagramcats}
Let $n = 2$. Then $\mathcal{P}_{\neq \varnothing} ([2])$ is the punctured 2-cube:
\[
\xymatrix@!0{
&
& & 2 \ar[dd] \ar[dr]
\\   
&&0 \ar[dd]\ar[rr] & & 02 \ar[dd]
\\
&1 \ar[rr]|\hole \ar[dr]
& & 12
\\
&&01 \ar[rr]
& & 012 \ar@{<-}[ul] \rlap{.}
}
\]

Let us record the values of $\mathcal{P}_{\neq \varnothing} ([2])_{\varphi}$ for all possible $\varphi$.
\begin{itemize}
\item $\mathcal{P}_{\neq \varnothing} ([2])_{0 } = \{0\} \cong [0]$.
\item $\mathcal{P}_{\neq \varnothing} ([2])_{1 } = \{1, 01\} \cong [1]$.
\item $\mathcal{P}_{\neq \varnothing} ([2])_{2 } = \{2, 02, 12, 012\} \cong [1] \times [1]$.
\item $\mathcal{P}_{\neq \varnothing} ([2])_{01 } = \{01\} \cong [0]$.
\item $\mathcal{P}_{\neq \varnothing} ([2])_{02 } = \{02, 012\} \cong [1]$.
\item $\mathcal{P}_{\neq \varnothing} ([2])_{12 } =  \{12, 012\} \cong [1]$.
\item $\mathcal{P}_{\neq \varnothing} ([2])_{012 } =  \{012\} \cong [0]$.
\end{itemize}
\end{example}

The following lemma provides a description of the explicit combinatorics of the isomax undercategory in the case that  the poset is a finite total order.

\begin{lemma}\label{lem:isomax_cat}
Let $\varphi = (i_1, \dots, i_k) \in \mathcal{P}_{\neq \varnothing} ([n])$. Then $\mathcal{P}_{\neq \varnothing} ([n])_{\varphi } \cong [1]^\ell$ where $\ell = i_k-k+1$.
\end{lemma}

\begin{proof}
This follows from noting that $[n]$ is totally ordered and counting the number of elements $\psi$ with $\psi \subseteq \varphi$ and $\max(\psi) = \max(\varphi)$.
\end{proof}

The following result gives the promised rewriting of a right-lax limit of a left-lax functor over $[n]$ as a limit over the diagram $\mathcal{P}_{\neq \varnothing}([n]) $. We will first discuss the maps that will appear. Let $\varphi = (i_1, \dots i_k) \in \mathcal{P}_{\neq \varnothing}([n])$ and consider $\psi = \varphi \cup \{j\}$ where $j \not\in \varphi$. Using \cref{lem:isomax_cat} we can see that there are only three possible options on what can happen to the isomax undercategory. Writing $\ell$ such that $\mathcal{P}_{\neq \varnothing} ([n])_{\varphi } \cong [1]^\ell$ we have
\[
\mathcal{P}_{\neq \varnothing} ([n])_{\psi } \cong \begin{cases}
			[1]^{\ell - 1} & \text{if $j < i_k$}\\
            [1]^\ell & \text{if $j = i_k+1$}\\
           [1]^{\ell+ \zeta} & \text{else}.
		 \end{cases}
\]
where $\zeta = (j-i_k - 1) > 1$.

We note that in the first case, that is, when $j < i_k$, there is a well defined projection
\begin{equation}\label{eq:isoprojection}
p_j \colon \mathcal{P}_{\neq \varnothing} ([n])_{\psi } \cong [1]^\ell \to [1]^{\ell -1}  \cong \mathcal{P}_{\neq \varnothing} ([n])_{\psi }
\end{equation}
which projects away from $j$. This can be seen explicitly in the case of $n=2$ via \cref{ex:diagramcats}. 

\begin{proposition}[{\cite[Lemma A.6.5]{AMGR}}]\label{prop:AMGR_adelic}
Let $F \colon [n] \to \stable2Cat$ be a left-lax functor. Then the right-lax limit of this functor is equivalent to the limit of the diagram
\begin{align*}
\mathcal{P}_{\neq \varnothing}([n]) &\to \stableCat \\
\varphi  &\mapsto \Fun(\mathcal{P}_{\neq \varnothing} ([n])_{\varphi }, F({\max(\varphi)}))\\
\end{align*}
where an edge $\varphi = (i_1, \dots i_k)$ in $\mathcal{P}_{\neq \varnothing}([n])$ gets mapped to
\[(\varphi \mapsto \varphi \cup \{j\}) \mapsto 
\begin{cases}
			\pi_j \colon \Fun([1]^{\ell}, F(i_k)) \to  \Fun([1]^{\ell-1}, F(i_k)) & \mathrm{if } \, j < i_k\\
            F(i \mapsto i+1)^\Delta \colon \Fun([1]^{\ell}, F(i_k)) \to  \Fun([1]^{\ell}, F(i_k+1)) & \mathrm{if } \, j = i_k+1\\
            \eta_{i_k, j} \colon \Fun([1]^{\ell}, F(i_k)) \to  \Fun([1]^{\ell+\zeta}, F(j)) & \mathrm{else}
\end{cases}
\]
for $\ell = i_k -k + 1$ and $\zeta = j -i_k - 1$. Here $\pi_j$ is the functor induced by the projection \eqref{eq:isoprojection}, $F(i \mapsto i+1)^\Delta$ is the functor that applies the functor $F(i \mapsto i+1)$ at each vertex, and $\eta_{i_k,j}$ is the functor that records the diagram controlling the laxness of the compositions $i_k \to j$ in $F$.
\end{proposition}

\begin{remark}\label{rem:addinginlax}
    It is worth commenting further on the functors $\eta_{i_k,j}$ appearing in \cref{prop:AMGR_adelic}, and we do so via a small example. Suppose that we have $F \colon [3] \to \stable2Cat$ a left-lax functor, i.e., $F$ has the form of \cref{fig:tetra}.
    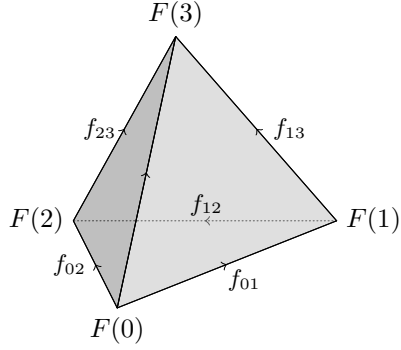
\begin{figure}[ht]
        \centering
        \begin{tikzpicture}[line join = round, line cap = round, scale=2]
\coordinate [label=above:$F(3)$] (3) at (0,{sqrt(2)},0);
\coordinate [label=left:$F(2)$] (2) at ({-.5*sqrt(3)},0,-.5);
\coordinate [label=below:$F(0)$] (1) at (0,0,1);
\coordinate [label=right:$F(1)$] (0) at ({.5*sqrt(3)},0,-.5);
\begin{scope}[decoration={markings,mark=at position 0.5 with {\arrow{to}}}]
\draw[densely dotted,postaction={decorate}] (0)--(2);
\draw[fill=lightgray,fill opacity=.5] (1)--(0)--(3)--cycle;
\draw[fill=gray,fill opacity=.5] (2)--(1)--(3)--cycle;
\draw[postaction={decorate}] (1)--(0);
\draw[postaction={decorate}] (1)--(2);
\draw[postaction={decorate}] (2)--(3);
\draw[postaction={decorate}] (1)--(3);
\draw[postaction={decorate}] (0)--(3);
\node at (0.45,-0.2) {\small$f_{01}$};
\node at (-0.7,-0.1) {\small$f_{02}$};
\node at (-0.5,0.8) {\small$f_{23}$};
\node at (0.75,0.8) {\small$f_{13}$};
\node at (0.2,0.3) {\small$f_{12}$};
\end{scope}
\end{tikzpicture}
        \caption{A left-lax functor $F \colon [3] \to \Cat_\infty$. For each face $0 \leqslant i < j < k \leqslant 3$ there is a natural transformation $\eta_{ijk} \colon f_{ik} \Rightarrow f_{jk}f_{ij}$.}
        \label{fig:tetra}
    \end{figure}
    
    Let $i_k = 0$ and $j = 3$. That is, we wish to understand the functor
    \[
    \eta_{0,2} \colon \Fun([0], F(0)) \to  \Fun([1]^{2}, F(3)).
    \]
    Starting at $F(0)$ there are four distinct ways to get to $F(3)$ in \cref{fig:tetra}, and due to the fact that the diagram is left-lax we can compare these as in \eqref{eq:storinglax} using the $\eta_{ijk}$, where we have written the isomax undercategory $\mathcal{P}([3])_{03}$ on the left-hand side as $[1]^2$ for comparison:
    \begin{equation}\label{eq:storinglax}
    \begin{gathered}
    \xymatrix@C=3em{
    013 \ar[r] & 0123 & & f_{13}f_{01} \ar@{=>}[r]^-{\eta_{123} f_{01}} & f_{23}f_{12}f_{01} \\
    03 \ar[u] \ar[r] & 023 \ar[u] & & f_{03} \ar@{=>}[r]_-{\eta_{023}} \ar@{=>}[u]^-{\eta_{013}}  & f_{23}f_{02} \ar@{=>}[u]_-{f_{23}\eta_{012}} \rlap{.}
    }
    \end{gathered}
    \end{equation}
    It is exactly the right-hand square of \eqref{eq:storinglax} that the functor $\eta_{0,2}$ records.

    Explicit descriptions of the functors $\eta_{i_k,j}$ in general can be found at \cite[Lemma A.6.5]{AMGR}.
\end{remark}

With \cref{prop:AMGR_adelic} in hand, we can state and prove the main result of this section. 

\begin{theorem}\label{theorem:categoricalL2G}
Let $\sfT$ be a finitely dispersible rigidly-compactly generated tensor-triangulated category with prism $\pmb{P}$.  Then $\sfT$ is equivalent to the homotopy limit of the punctured cube
\begin{align*}
\cD_{(\sfT,\chi)}\colon\mathcal{P}_{\neq \varnothing}([n]) &\to \stableCat \\
\varphi  &\mapsto \Fun \left(\mathcal{P}_{\neq \varnothing} ([n])_{\varphi }, \prod_{x \in \pmb{P}_{\max(\varphi)}} \Gamma_x \sfT\right)
\end{align*}
with maps as in \cref{prop:AMGR_adelic}.
\end{theorem}

\begin{proof}
    By \cref{thm:catl2g}, $\sfT$ is equivalent to the right-lax limit of the left-lax functor which sends $i$ to $\prod_{x \in \pmb{P}_{i}} \Gamma_x\sfT$. The result then follows by combining this with \cref{prop:AMGR_adelic}.
\end{proof}

\subsection{Examples}\label{ssec:catstrat_examples}

We finish this section by providing two examples which appear in the gluing behaviour that will arise in our main example of interest in \cref{part:gspectra}.

\begin{example}\label{ex:dzexamplecont}
Let us continue the example of $\sfD(\Z)$ from \cref{ex:derivedabelianobj}. Recall from \cref{ex:running1} that this admits a height 1 dispersion where the open down-set $\pmb{U}$ consists of the collection of non-zero prime ideals, and $\pmb{Z} = \pmb{U}^c = (0)$. Applying \cref{lem:homotopy-cartesian-height1-complete} and identifying the relevant pieces (including identifying the category of derived $p$-torsion with that of derived $p$-completion) we obtain the following homotopy cartesian square
\[\xymatrix{
\sfD(\Z) \ar[r] \ar[d] & \displaystyle{\prod_{p \in \mathbb{P}}} \sfD(\Z)_p^\wedge \ar[d]^-{\Q \otimes \prod_\mathbb{P}} \\ \Fun([1] , \sfD(\Q)) \ar[r]_-{\pi_1} & \sfD(\Q) \rlap{.}
}\]
An object of this homotopy cartesian square is the data of:
\begin{itemize}
    \item modules $M_p \in \sfD(\Z)_p^\wedge$ for all $p \in \mathbb{P}$;
    \item a rational vector space $V \in \sfD(\Z)$;
    \item a map of rational vector spaces $V \to \Q \otimes \prod_\mathbb{P} M_p$.
\end{itemize}
\end{example}

\begin{example}\label{ex:derivedabsflatring}
We return to \cref{ex:running2}. That is, we let $R$ be the ring of continuous functions $C(\N^*,k)$ for $k$ a field equipped with the discrete topology, and $\sfT = \mathsf{D}(R)$. This again is a dispersible tensor-triangulated category of height 1 by \cref{ex:nstartisdisp}, and as such we can apply \cref{lem:homotopy-cartesian-height1-complete} and identify the relevant pieces: 
\[\xymatrix{
\sfD(R) \ar[r] \ar[d] & \displaystyle{\prod_{n \in \mathbb{N}}} \sfD(k) \ar[d]^-{c_\ast} \\ \Fun([1] , \sfD(k)) \ar[r]_-{\pi_1} & \sfD(k)
}\]
where it remains to describe the Tate construction which we have suggestively denoted $c_\ast$. This can be rectified by considering the behaviour of the point at $\{\infty\}$ via the pushforward of a sheaf on $\N$ under the Alexandroff extension $c \colon \N \to \N^*$. Spelling this out, suppose that we have a sheaf $\mathcal{F}$ on $\N$. Let $U \subseteq \N^*$ be an open subset. Then, by definition, we have $c_\ast \mathcal{F}(U) \coloneqq \mathcal{F}(c^{-1}(U))$. The stalk at $\infty$ is then
\[
(c_\ast \mathcal{F})_{\infty} \coloneqq \colim_{U \ni \infty} (c_\ast \mathcal{F})(U) = \colim_{U \ni \infty} \mathcal{F}(c^{-1} (U)).
\]
Recalling \cref{def:onepointcpt}, the open subsets $U \ni \infty$ are all of the form $(\N \setminus C) \cup \{\infty\}$ where $C$ is a finite subset in $\N$. In particular it follows that for such a $U$ the preimage $c^{-1}(U)$ is the complement of a finite subset of $\N$. By taking a cofinal system we may as well consider the system of opens $\{\N, \N\backslash \{0\}, \N \backslash \{0,1\}, \dots \}$. Finally, as $\N$ is discrete, the value of $\mathcal{F}$ on any open subset of $\N$ decomposes as a product over the stalks. Assembling all of this we have that
\[
(c_\ast \mathcal{F})_\infty \simeq \colim_i \prod_{n \geqslant i} \mathcal{F}_n,
\]
where the maps in the system are given by restrictions. Returning to the example at hand of $C(\N^*,k)$, this equivalence provides an explicit formula for the Tate construction $c_*$ in the above square. 
\end{example}

\begin{remark}
    We note the different character of the Tate splicing data in Examples \ref{ex:dzexamplecont} and \ref{ex:derivedabsflatring}. Giving a systematic account for all compact Lie groups of the splicing data in terms of intrinsic structure will be an essential ingredient in understanding the algebraic model.
\end{remark}

\newpage

\part{Dispersions for rational \texorpdfstring{$G$}{G}-spectra}\label{part:gspectra}

We now apply the theory developed in \cref{part:prisms} to the category $\SpG$ of $G$-spectra for a compact Lie group $G$.  It is the homotopy category of a stable monoidal $\infty$-category  and hence a  tensor-triangulated category with respect to the usual smash product. Here, we mean the category of `genuine' $G$-spectra, for which one can suspend and desuspend by an arbitrary real representation. We will focus particularly on the localization $\SpGQ$, whose objects are $G$-spectra with rational homotopy groups, 
which represent cohomology theories taking values in rational vector spaces.

The Balmer spectrum of $\SpGQo$ was determined in~\cite{greenlees_bs} and a partial extension to the integral case was given in \cite{BarthelGreenleesHausmann2020}. However the description of the topology in \cite{greenlees_bs} is difficult to work with in practice. Here we will give a much clearer description of $\SpGQ$ via its prism, and demonstrate the power of this approach by proving numerous general structural results. Indeed, both the constructible topology and the associated spectral ordering familiar structures. We will in particular construct a `prismatic' decomposition of $\SpGQ$: we show that it is dispersible, and hence it satisfies the categorical local-to-global principle. 

\renewcommand{\SpGQ}{\Sp_{G}}
\renewcommand{\SpGQo}{\Sp_{G}^{\omega}}

\begin{convention*}\label{conv:gspectra_notation}
    Throughout the rest of Part 2, we are dealing exclusively with rational spectra so we drop the subscript $\Q$ from the notation and write $\SpG$ for the category of rational $G$-spectra by $\SpGQ$. Similarly, we will write $\Sp$ for the category of rational spectra.
\end{convention*}

\section{Spaces of subgroups}\label{sec:sub}

We begin with a recollection on various topologies on spaces of subgroups of a given compact Lie group $G$ and their relationship. Throughout, we consider only \emph{closed} subgroups of $G$, without always saying so explicitly. At the end of this section, we will give a few simple examples to show the richness of the structures defined.

\subsection{The $h$-topology}\label{ssec:htopology}

We write $\Sub(G)$ for the set of closed subgroups of $G$, and summarise results from  tom Dieck \cite{tomDieck79}. We choose a bi-invariant metric $d$ on $G$: this gives the Hausdorff metric on the space of
compact subspaces of $G$, and hence on the subspace $\Sub(G)$. The resulting \emph{$h$-topology} on $\Sub(G)$ is independent of the choice of the bi-invariant metric $d$.

The conjugation action is continuous for the Hausdorff metric and we denote by $\Sub(G)/G$ the corresponding quotient space with the induced topology, which we refer to as the \emph{$h$-topology}. We record some standard results. 

\begin{proposition}[{\cite[1.7.27]{palais_classification}}]\label{prop:subgmgcountable}
The underlying set of $\Sub(G)/G$ is countable.
\end{proposition}

\begin{proposition}[{\cite[Propositions 5.6.1 and 5.6.2]{tomDieck79}}]\label{prop:subgisstone}
The space $\Sub(G)/G$ equipped with the $h$-topology is compact, Hausdorff and totally disconnected:  it is a Stone space.
\end{proposition}

Next, the collection of conjugacy of subgroups of $G$ admits a poset structure under the cotoral order. This order will play a central role below.

\begin{definition}\label{def:cotoral}
For subgroups $H, K$ of $G$ we say that $K$ is \emph{cotoral} in $H$, denoted $K\cotoral H $, if $K$ is normal in $H$  and $H/K$ is a torus. For conjugacy classes, we say $(K)\cotoral (H)$ if there are representative subgroups with $K\cotoral H$. Two subgroups $H$ and $K$ are said to be \emph{cotorally unrelated} if neither $K \cotoral H$ nor $H \cotoral K$.

For any set $M \subseteq \Sub(G)/G$ of conjugacy classes, we write $\Lct(M)$ for its cotoral down-closure. 
\end{definition}

We soon return to examples of these structures after we have introduced two more topologies on the set $\Sub(G)/G$ and recorded some facts about them.

\subsection{Weyl groups and the $f$-topology}
Before giving any examples, we need to say a bit more about the
Burnside ring.  To any subgroup $H \subseteq G$ we associate its 
\emph{Weyl group} $W_G(H) = N_G(H)/H$ where $N_G(H)$ is the normalizer
of $H$ in $G$. We write  $\cF(G) \subseteq \Sub(G)$ for the set of  subgroups which are of finite index in their normalizer (i.e., those that have finite Weyl groups). This is a  
closed subspace of $\Sub(G)$ and we  equip it with the $h$-topology. 
The corresponding space of conjugacy classes is denoted $\Phi(G) \coloneqq
\cF(G)/G$. By \cite[Proposition 5.6.1]{tomDieck79}, $\Phi(G)$ is compact, Hausdorff and totally disconnected. The significance of this space is tom Dieck's theorem (\cite[Corollary 5.7.6 and Proposition 5.9.13]{tomDieck79}) that fixed
points give an isomorphism
\begin{equation}\label{eq:burnsidering}
[S^0,S^0]^G\tensor \Q \stackrel{\cong}\longrightarrow C(\Phi G,\Q),
\end{equation}
where the target denotes the ring of locally-constant functions on $\Phi(G)$ with values in the rationals. We will write $A_{\Q}(G) \coloneqq [S^0,S^0]^G\tensor \Q$ for the \emph{rational Burnside ring}. This feeds into our analysis through  a second topology on $\Sub(G)$
called the \emph{$f$-topology}. It again comes from the Hausdorff
metric, but now subgroups can only be close if their Weyl groups are
finite (hence the use of the letter `$f$').  We first give an explicit
definition, but for purposes of proof, we always use
the characterisation in \cref{lem:ftop} below. 

\begin{definition}\label{def:ftopology}
 For a closed subgroup $H \subseteq G$ and $\varepsilon > 0$ we define the ball
\[
O(H, \varepsilon) = \{K \in \cF(H) \mid d(H,K) < \varepsilon\}
\]
in $\Sub(G)$.  Given a neighbourhood $A$ of the identity in $G$ consider the following set
\[
O(H, \varepsilon, A) = \bigcup_{a \in A} O(H, \varepsilon)^a
\]
where $O(H, \varepsilon)^a$ is the set of $a$-conjugates of elements of $O(H, \varepsilon)$. The \emph{$f$-topology} on $\Sub(G)$  is the topology generated by the sets $O(H, \varepsilon, A) $ as $H,A,\varepsilon$ vary. We can then pass the $f$-topology to $\Sub(G)/G$. 
\end{definition}

\begin{remark}
For any closed subgroup $H$ in $G$, the maps $\Phi H\lra \sub(G)/G$ are continuous so every $h$-closed set is $f$-closed. The converse does not hold: for example the set of all finite subgroups of the circle group $S^1$ is $f$-closed but not $h$-closed as we will see in \cref{ex:s1topologies}. 
\end{remark}

The following is a useful characterization of the $f$-topology.

\begin{lemma}[{\cite[Lemma 6.2]{greenlees_bs}}]
\label{lem:ftop}
The $f$-topology on $\Sub(G)/G$ is the quotient topology for the maps $\Phi H \to \Sub(G)/G$ as $H$ runs through the closed subgroups of $G$. As such, a set is open in $\Sub(G)/G$ equipped with the $f$-topology if and only if its pullback to $\Phi H$ is open for all $H$.
\end{lemma}

\begin{lemma}\label{lem:fkhc}
If a subset is $f$-compact, then it is also $h$-closed.
\end{lemma}

\begin{proof}
The identity map $(\Sub(G)/G)_f \to (\Sub(G)/G)_h$ is continuous by \cite[Lemma 8.6]{greenlees_rationalmackey}, so it preserves compact subsets. But the target is Hausdorff, so compact implies closed. 
\end{proof}

\subsection{The $zf$-topology}\label{subsec:zftop}

Let us define yet another topology on $\Sub(G)/G$ which will be of great importance. We summarize relevant results from \cite{greenlees_bs}.

\begin{definition}
The $zf$-topology on $\sub(G)/G$ has a basis of closed sets given by the cotoral 
specializations $\Lct (V)$ where $V \subseteq \Sub(G)/G$ is $f$-compact and $f$-open.
\end{definition}

 \begin{remark}
\label{rem:zfmaxtopiszftop}
It is also proved as part of  \cite[Theorem 8.4]{greenlees_bs} that the $zf$-topology 
is generated by the down-closures of the $f$-compact, $f$-open and 
cotorally unrelated sets. In other words the two collections 
\begin{multline*}
\{ \Lct (M)\st M \mbox{ is $f$-compact and $f$-closed and } M=\mct M 
\}\\
\subseteq \{ \Lct (M)\st M \mbox{ is $f$-compact and $f$-closed}\}
\end{multline*}
of closed sets generate the same topology. 
\end{remark}

\begin{remark}
Subsection 9.B of the paper \cite{greenlees_bs} gives a different definition of the $zf$-topology, which we rename here as the $zfk$-topology for clarity. The closed sets in the $zfk$-topology are generated by the cotoral specializations $\Lct(V)$ of the $f$-compact sets $V$. However, the Zariski topology used in \cite[Theorem 9.3]{greenlees_bs} refers back to \cite[Theorem 8.4]{greenlees_bs}, where the topology is the $zf$-topology as defined here. Thus mention of the $zfk$ was a mistake, and there is no reason to discuss it further.
\end{remark}

\subsection{Examples} 

In summary, we have defined three different topologies on $\Sub(G)/G$:
\begin{enumerate}
    \item The $h$-topology: $(\Sub(G)/G)_h$;
    \item the $f$-topology: $(\Sub(G)/G)_f$;
    \item the $zf$-topology: $(\Sub(G)/G)_{zf}$.
\end{enumerate}
Similarly, we will occasionally indicate the topology in use on $\Sub(G)$ by a subscript. A few simple examples will provide an anchor in reality.

\begin{example}
If $G$ is finite, then the $h$-topology, the $f$-topology, and the $zf$-topology are all discrete on $\Sub(G)/G$.
\end{example}

\begin{example}\label{ex:s1topologies}
Let $G=S^1$, then the set of closed subgroups of $S^1$ consists of $S^1$ together with the set of finite cyclic groups $C_n$ for $n\geqslant 1$. The group is abelian, so conjugacy classes are singletons. We have: 
    \begin{itemize}
        \item $\Sub(S^1)_h$ is homeomorphic to the one-point compactification $(\Nvee)^*$ of the positive integers, where $S^1$ corresponds to the compactifying point $\infty$. Indeed, every neighbourhood of $S^1$ in the $h$-topology contains infinitely many finite cyclic subgroups. 
    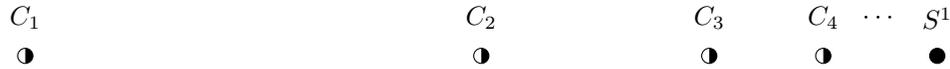
\begin{figure}[ht!]
        \centering
            \begin{tikzpicture}[xscale=-1]
                \node at (12,1.5) {$C_1$};
                \node at (6,1.5) {$C_2$};
                \node at (3,1.5) {$C_3$};
                \node at (1.5,1.5) {$C_4$};
                \node at (0.75,1.5) {$\cdots$};
                \node at (0,1.5) {$S^1$};
                \draw[fill=white] (12,1) circle (.1cm);
                \fill[black] (12,0.9) arc (270:90:.1cm);
                \draw[fill=white] (6,1) circle (.1cm);
                \fill[black] (6,0.9) arc (270:90:.1cm);
                \draw[fill=white] (3,1) circle (.1cm);
                \fill[black] (3,0.9) arc (270:90:.1cm);
                \draw[fill=white] (1.5,1) circle (.1cm);
                \fill[black] (1.5,0.9) arc (270:90:.1cm);
                \draw[fill=black] (0,1) circle (.1cm);
            \end{tikzpicture}\caption{The topological space $\Sub(S^1)_h$.}
    \end{figure}
        \item The definition of basic open sets for the $f$-topology includes a  finiteness condition on Weyl groups, so the singleton $\{S^1\}$ is $f$-open. It follows that $\Sub(S^1)_f$ is discrete, with underlying set $\Nvee \cup    \{\infty\}$, where $C_{n}$ corresponds to ${n} \in \Nvee$ and $S^1$ corresponds to the point $\infty$. In particular, the $f$-topology is in general not compact and hence not spectral. 
        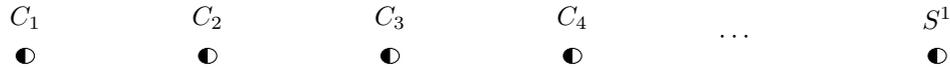
\begin{figure}[ht!]
            \centering
                \begin{tikzpicture}[xscale=1.2]
                \node at (0,1.5) {$C_1$};
                \draw[fill=white] (0,1) circle (.1cm);
                \fill[black] (0,0.9) arc (270:90:.1cm);
                
                \node at (2,1.5) {$C_2$};
                \draw[fill=white] (2,1) circle (.1cm);
                \fill[black] (2,0.9) arc (270:90:.1cm);
 
                \node at (4,1.5) {$C_3$};
                \draw[fill=white] (4,1) circle (.1cm);
                \fill[black] (4,0.9) arc (270:90:.1cm);
                
                \node at (6,1.5) {$C_4$};
                \draw[fill=white] (6,1) circle (.1cm);  
                \fill[black] (6,0.9) arc (270:90:.1cm);
 
                \node at (7.8,1.25) {$\cdots$};

                \node at (10,1.5) {$S^1$};
                \draw[fill=white] (10,1) circle (.1cm);
                \fill[black] (10,0.9) arc (270:90:.1cm);
            \end{tikzpicture}\caption{The topological space $\Sub(S^1)_f$.}
    \end{figure}     
        \item $\Sub(S^1)_{zf}$ is homeomorphic to $\Spec(\Z)$ with the Zariski topology, where the generic point corresponds to $S^1$ (cf. \cref{ex:specz}). 
        \begin{figure}[ht!]
            \begin{tikzpicture}[yscale=1]
            \draw [->,line join=round,decorate, decoration={zigzag, segment length=6, amplitude=.9,post=lineto, post length=2pt}]  (0,0) -- (-3.75,-1.2);
            \draw[white, fill=white] (0,0) circle (.1cm);
            \draw [->,line join=round,decorate, decoration={zigzag, segment length=6, amplitude=.9,post=lineto, post length=2pt}]  (0,0) -- (3.75,-1.2);
            \draw[white, fill=white] (0,0) circle (.3cm);
            \draw [->,line join=round,decorate, decoration={zigzag, segment length=6, amplitude=.9,post=lineto, post length=2pt}]  (0,0) -- (-1.8,-1.25);
            \draw [->,line join=round,decorate, decoration={zigzag, segment length=6, amplitude=.9,post=lineto, post length=2pt}]  (0,0) -- (1.8,-1.25);
            \draw [->,line join=round,decorate, decoration={zigzag, segment length=6, amplitude=.9,post=lineto, post length=2pt}]  (0,0) -- (0.0,-1.4);
            \draw[white, fill=white] (0,0) circle (.3cm);
            \draw[dotted, fill=black!10,thick] (0,0) circle (.15cm);
            \draw[fill = black] (-4,-1.25) circle (0.1cm);
            \draw[fill = black] (-2,-1.5) circle (0.1cm);
            \draw[fill = black] (0,-1.75) circle (0.1cm);
            \draw[fill = black] (2,-1.5) circle (0.1cm);
            \node at (0,.5) {$S^1$};
            \node at (-4,-1.75) {$C_1$};
            \node at (-2,-2.0) {$C_2$};
            \node at (0,-2.25) {$C_3$};
            \node at (2,-2.0) {$C_4$};
            \node at (4,-1.75) {$\cdots$};
            \end{tikzpicture}\caption{The topological space $\Sub(S^1)_{zf}$.}\label{fig:s1zf}
        \end{figure}
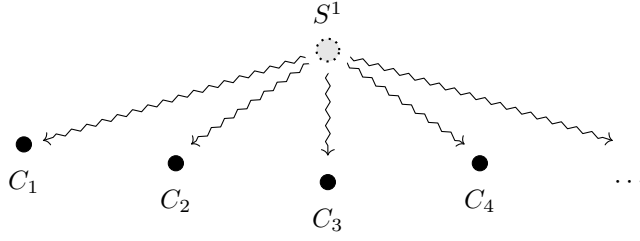
    \end{itemize}
\end{example}

\begin{example}\label{ex:dO2topologies}
Let $G=O(2)$ and consider the dihedral part $\cD(O(2))$ of $\Sub(O(2))/O(2)$, consisting of $O(2)$ itself and  one conjugacy class $(D_{2n})$ of finite dihedral groups for each $n\geqslant 1$. The Weyl group of $D_{2n}$ is a group of order 2.
    \begin{itemize}
        \item The $f$-topology and the $h$-topology agree on $\cD(O(2))$, and both are homeomorphic to the one-point compactification of $\Nvee$. 
        \item  There are no non-trivial cotoral inclusions among the elements of $\cD(O(2))$, so the $zf$-topology also coincides with the $h$-topology.
    \end{itemize}
In summary, we have homeomorphisms 
    \[
        \cD(O(2))_f = \cD(O(2))_h = \cD(O(2))_{zf} \cong (\Nvee)^*, 
    \]
where $(D_{2n}) \in \cD(O(2))$ corresponds to $n \in \Nvee$ and $O(2)$ corresponds to the point  $\infty$. We then additionally have the non-dihedral part contributing a disjoint copy of the corresponding space from \cref{ex:s1topologies}. 
\end{example}

\section{The prism of rational \texorpdfstring{$G$}{G}-spectra}\label{sec:prism}

The goal of this section is to identify the prism of the category of rational $G$-spectra and use this to give a description of the Zariski-closed subsets. Throughout, two topologies will be prominent: the $h$-topology on $\sub(G)/G$ and the $zf$-topology. 

We begin our description of the spectrum and its prism with the identification of the underlying set of $\Spc(\SpGQo)$. To this end, recall that for any closed subgroup $H \subseteq G$, we have geometric fixed point functors
\[
\xymatrix{\Phi^H\colon \SpGQ \ar[r] & \Sp.}
\]
Up to natural equivalence, $\Phi^H$ only depends on the conjugacy class of $H$ in $G$, and the collection $(\Phi^H)_{H \in \Sub(G)/G}$ forms a jointly conservative family of symmetric monoidal and exact functors. In particular, we can pull back the unique prime ideal of $\Spc(\Sp^{\omega})$ along $\Phi^H$, leading to the following.

\begin{definition}
For $H \in \Sub(G)/G$, set $\cP_G(H) \coloneqq \{X \in \SpGQo \mid \Phi^HX = 0\} \in \Spc(\SpGQo)$.
\end{definition}

By \cite[Theorem 7.4]{greenlees_bs}, the resulting map 
\begin{equation}\label{eq:spectrumbijection}
    \xymatrix{\Sub(G)/G \ar[r] & \Spc(\SpGQo), \quad H \mapsto \cP_G(H).}
\end{equation}
is a bijection, see also \cite[Theorem 3.14]{BarthelGreenleesHausmann2020} for an integral version of this result.

The spectral order on the spectrum corresponds to the cotoral order on $\Sub(G)/G$; in other words, it determines the poset $\Spc(\SpGQo)$. This was proved in \cite{greenlees_bs}, see also \cite[Theorem 6.2(ii),(iii)]{BarthelGreenleesHausmann2020}, and the main ingredient is the Localization Theorem. 

\begin{lemma}\label{lem:cotoralorder}
    For closed subgroups $H,K \subseteq G$, the following conditions are equivalent:
        \begin{enumerate}
            \item $\cP_G(K) \subseteq \cP_G(H)$;
            \item $(K) \cotoral (H)$. 
        \end{enumerate}
\end{lemma}

Our next aim is to describe a basis of closed subsets for the spectrum $\Spc(\SpGQo)$ in terms of group-theoretic data. In preparation for the proof, we introduce an auxiliary notion of type function.

\begin{definition}
    The (rational) type function $\type_X\colon \Sub(G)/G \to \{0,1\}$ of a finite rational $G$-spectrum $X$ is the characteristic function of the complement of $\supp (X)$:  
        \[
        \type_X(H) = 
            \begin{cases}
                0 & \text{if } \Phi^H(X) \not \simeq *; \\
                1 & \text{if } \Phi^H(X) \simeq *.
            \end{cases}
        \]
\end{definition}
We note that rational type functions for finite rational $G$-spectra are truncations of the integral type functions considered in \cite{BarthelGreenleesHausmann2020}, collapsing the chromatic direction $[1,\infty]$ down to the single point $1$. For the purposes of the next statement, we also recall that, by construction, a basis of quasi-compact closed subsets of the Balmer spectrum is given by the supports of compact objects (\cite{balmer_spectrum}); we also refer to these subsets as \emph{basic Zariski-closed} subsets.

\begin{proposition}\label{prop:basicclosed}
    For a subset $Z \subseteq \Spc(\SpGQo)$, the following are equivalent:
        \begin{enumerate}
            \item $Z$ is basic Zariski-closed, i.e., there exists some $X \in \SpGQo$ with $Z = \supp(X)$;
            \item $Z$ is basic $zf$-closed, i.e., there exists an $f$-compact, $f$-open, cotorally unrelated $M$ with $Z = \Lct (M)$;
            \item $Z$ is $h$-clopen and closed under passage to cotoral subgroups.
        \end{enumerate}
\end{proposition}
\begin{proof}
The equivalence of (1) and (2) is established in \cite[Theorem 8.4]{greenlees_bs}.

Our proof of the equivalence between (1) and (3) essentially follows from the results of \cite{BarthelGreenleesHausmann2020}. Under the identification of \eqref{eq:spectrumbijection} and definition of the prime ideals $\cP_G(H)$, we have
\begin{equation}\label{eq:typesupport}
\supp(X) = \{\cP_G(H)\mid X \notin \cP_G(H)\} = \type_X^{-1}(0). 
\end{equation}
Upon rationalization, Corollary 5.8 in \cite{BarthelGreenleesHausmann2020} characterizes type functions of finite rational $G$-spectra as those functions
\[
\xymatrix{f\colon (\Sub(G)/G)_h \to \{0,1\}}
\]
which are locally constant and admissible. Unwinding \cite[Definition 5.1]{BarthelGreenleesHausmann2020} in the rational context, $f$ being admissible means that
\begin{equation}\label{eq:admissible}
\cP_G(H) \not\subseteq \cP_G(K) \text{ whenever } f(K) = 1 \text{ and } f(H)=0
\end{equation}
for all closed subgroups $H,K \subseteq G$. By virtue of \eqref{eq:typesupport}, the basic Zariski-closed subsets of $\Spc(\SpGQo)$ are thus given by $Z=f^{-1}(0)$, where $f$ runs through the locally constant admissible functions. 

Note that a function $f\colon (\Sub(G)/G)_h \to \{0,1\}$ is locally constant if and only if $f^{-1}(0)$ is $h$-clopen. Via \cref{lem:cotoralorder}, the admissibility condition \eqref{eq:admissible} in turn translates into the statement that $K \not\cotoral H$ whenever $H \in f^{-1}(0)$ and $K \notin f^{-1}(0)$; equivalently, if $H \in f^{-1}(0)$ and $K \cotoral H$, then $K \in f^{-1}(0)$. In plain words, $f$ is admissible if and only if $f^{-1}(0)$ is cotorally down-closed. In summary, the basic Zariski-closed subsets correspond precisely to the $h$-clopen and cotorally down-closed subset of $\Sub(G)/G$, as claimed.
\end{proof}

\begin{theorem}\label{thm:prism}
For any compact Lie group $G$, the Balmer spectrum of finite rational $G$-spectra is described as
\[
\Prism(\SpGQo) = ((\Sub(G)/G)_h, \cotoral).
\]
\end{theorem}
\begin{proof}
We first verify that $((\Sub(G)/G)_h, \cotoral)$ is a Priestley space. The topological space $(\Sub(G)/G)_h$ is a Stone space by \cref{prop:subgisstone}, so it remains the show that $\cotoral$ satisfies the Priestley separation axiom (\cref{def:spectralorder}) Consider $H,K \in (\Sub(G)/G)_h$ with $K \not\cotoral H$. Then there exists a basic Zariski-closed subset $Z$ with $H \in Z$ and $K \notin Z$. By \cref{prop:basicclosed}, $Z$ is $h$-clopen and down-closed with respect to $\cotoral$, thus separating $H$ from $K$.

The spectral space associated to $((\Sub(G)/G)_h, \cotoral)$ is determined by the collection of basic closed subsets, which under the Priestley correspondence \cref{cor:topologydesc}\ref{item:basicclosed} is given by the collection of $h$-clopen and cotorally closed subsets of $(\Sub(G)/G)$. In light of \cref{prop:basicclosed} again, these coincide with the basic closed sets of the Zariski topology under the identification \eqref{eq:spectrumbijection}. Therefore, $\tau_u((\Sub(G)/G)_h, \cotoral) = \Spc(\SpGQo)$ and, by applying $\Pries$, we get
\[
((\Sub(G)/G)_h, \cotoral) = \Pries \tau_u((\Sub(G)/G)_h, \cotoral) = \Pries \Spc(\SpGQo),
\]
as desired. 
\end{proof}

\begin{corollary}\label{cor:zfclosed}
    For a subset $C \subseteq \Sub(G)/G$, the following conditions are equivalent: 
        \begin{enumerate}
            \item $C$ is Zariski-closed; 
            \item $C$ is $zf$-closed;
            \item $C$ is $h$-closed and closed under cotoral specialization.
        \end{enumerate}
\end{corollary}
\begin{proof}
The Zariski topology and the $zf$-topology have been identified in \cite{greenlees_bs}, so it remains to describe their closed sets in terms of the $h$-topology and the specialization order. This follows from Priestley theory: by \cite[1.5.11(ii)]{book_spectralspaces}, the closed subsets of the spectral space attached to a Priestley space are given by the constructible-closed and down-closed sets. By \cref{thm:prism}, in the case at hand these are given precisely by the $h$-closed cotorally down-closed sets. 
\end{proof}

\begin{remark}\label{rem:prism_alternativeproof}
It is also possible to directly identify the $zf$-closed subsets with the cotorally closed subsets, appealing to the techniques of \cite{greenlees_bs} more explicitly. This provides an alternative approach to \cref{thm:prism}.
\end{remark}

\section{Subgroups of compact Lie groups}\label{sec:subgroups}

In the previous section, we have explained how the prism of $\SpGQ$ is determined by the space of closed subgroups of $G$ together with its cotoral ordering. The arguments regarding these spaces involve a number of different topologies, finiteness of Weyl groups, cotoral inclusions, and the interactions between these structures. In order to extract geometric information about the category of a $G$-spectra from its spectrum, we first need to collect some auxiliary facts about the structure theory of subgroups and the finer structure of the space of closed subgroups of a compact Lie group. 

\subsection{The Montgomery--Zippin theorem}\label{ssec:mztheorem}

Of fundamental importance in the study of subgroups of compact Lie groups is the Montgomery--Zippin theorem \cite{MZ}. Here, we recall the statement, state its consequences for the space of subgroups equipped with the $h$-topology, and establish an elementary counterpart for quotient groups.

    \begin{enumerate}[label=(\arabic*)] 
        \item[(MZ)]\label{item:mztheorem} The Montgomery--Zippin Theorem states that a convergent sequence $H_i \lra H_*$ of subgroups eventually consists of conjugates of subgroups of $H_*$. This means we can understand a neighbourhood of a subgroup $H$ in $\sub(G)/G$ by looking inside $\sub(H)/H$. Likewise, there are representatives in $\Sub(G)$ of a convergent sequence of conjugacy classes of subgroups which then also converge, see for example \cite[Proposition 5.6.2]{tomDieck79}. 
        Explicitly, if $(K_i)_G \to (K)_G$ is a convergent sequence with $K$ a subgroup of $H \subseteq G$ then, up to discarding finitely many terms if necessary, we can find representatives $\tilde{K}_i$ for $(K)_G$ such that $\tilde{K}_i \subseteq H$ and $\tilde{K}_i \to K$ converges in $\Sub(H)$. 

        Phrased differently yet again, there is a commutative square of compact Hausdorff spaces 
            \[
                \xymatrix{\Sub(H) \ar[r] \ar[d] & \Sub(G) \ar[d] \\
                \Sub(H)/H \ar[r] & \Sub(G)/G}
            \]
        in which all maps are both open and closed as well as continuous. Here, the vertical maps are the quotient maps, while the horizontal maps are induced by the inclusion of groups $H \hookrightarrow G$.        
    \end{enumerate}

    We also need a way to understand how convergence behaves under quotients of compact Lie groups. 

\begin{lemma}
\label{lem:pstarco}
If  $K$ is a normal subgroup of a compact Lie group $G$ with quotient $Q$ and $p\colon G\lra Q $ is the projection then $p_*\colon \sub(G)\lra \sub (Q)$ is continuous and closed. The image of a neighbourhood of a subgroup $H$ containing $K$ contains a neighbourhood of $H/K$.
\end{lemma}
\begin{proof} 
The following notation will be useful: for any closed set $A\subseteq X$,  We write 
    \[
    A_r=\{ x \st d(x,a)<r \mbox{ for some } a \in A\} 
    \]
for the subset of $X$ which is $A$ with an $r$-fringe. 

First we prove $p_*$ is continuous at $A\in \sub(G)$.
By continuity of $p$ and compactness of $G$, given $\epsilon >0$ there is a single $\delta >0$ so that $d(a,a')<\delta $ implies $d(pa,pa')<\epsilon $. Now move to the level of subsets and suppose $d(A,A')<\delta$. This means $A_\delta \supseteq A'$ and $A'_\delta \supseteq A$. In the former case, for every $a'\in A'$ there is an $a\in A $ with $d(a,a')<\delta $ and hence $d(pa, pa')<\epsilon$ so that $(pA)_\epsilon \subseteq (pA')$. Similarly with $A$ and $A'$ exchanged, so that $d(pA',pA)<\epsilon$. 

Both $\sub(G)$ and $\sub (Q)$ are compact Hausdorff, so $p_*$ is also closed.

 For the last statement, we consider a closed subgroup  $H$ containing $K$, and claim that the image of a ball around $H$ contains a ball around $H/K$. 
 Every subgroup of $Q$ is of the form $H'/K$ for some subgroup of $H'$ of $G$ so it suffices to observe that for each $\epsilon >0$ there is an $\delta >0$ so that $d(H/K, H'/K)<\delta $ implies $d(H,H')<\epsilon$. 

 The quotient map $p$ is a submersion, and hence locally a projection, and we may suppose the metrics are locally products. Thus we may take $\delta=\epsilon $ and note $d(H,H')\leqslant d(H/K, H'/K)<\epsilon$. 
\end{proof}

\subsection{Three continuity lemmas}\label{ssec:subgroups_continuity}

This subsection collects certain standard constructions on subgroups of a compact Lie group $G$ and shows that these vary continuously with the subgroup; this will turn out to be useful later in this section.

For any  compact Lie group $G$, there is a short exact sequence 
    \[
    1\lra G_e\lra G \lra G_d\lra 1,
    \]
where $G_e$ is the identity component of $G$ and $G_d$ is the discrete (finite) group of components. We say that a subgroup $\Gamma$ of $G$ is {\em full} if it maps onto $G_d$, i.e., $\Gamma$ meets every component of $G$. We also write $\Gamma_\e=G_e\cap \Gamma$ for the part of $\Gamma$ inside the identity component of $G$. Of course the identity component of $\Gamma$ will lie inside $\Gamma_\e$. 

\begin{lemma}\label{lem:ects}
The map $(-)_\e\colon \Sub(G)_h \to \Sub(G_e)_h$ is continuous. 
\end{lemma}
\begin{proof}
We need to check that if $H_i\lra H_*$ is a convergent sequence of closed subgroups, then this is also true of their parts inside $G_e$. Indeed, this holds since we may find a number $\varepsilon >0$ so that every point of $G$ outside $G_e$ has distance ${>}\varepsilon$ from $G_e$.
\end{proof}

By the structure theorem for compact connected Lie groups, there is a presentation
    \[
    G_e=\SSi \times_Z T
    \]
with $\SSi$ semisimple, $T$ a torus, and $Z$ a finite central subgroup. If $H\subseteq G_e=\SSi \times_ZT$ we may consider the inverse image $\widetilde{H}$ of $H$ under the quotient map
    \[
    p \colon \SSi \times T=\Gt \lra G=\SSi \times_ZT,
    \]
and one can check the following result.

\begin{lemma}\label{lem:tildects}
The map $\widetilde{(-)}$ on subgroups of $G_e$  is continuous. 
\end{lemma}

There is a well-defined map
    \[
    \omega\colon \sub(G)/G\lra \Phi G. 
    \]
which takes a conjugacy class $(H)$  to the conjugacy class of a maximal subgroup in which $H$ is cotoral; see \cite[Section 4]{greenlees_rationalmackey}, where the map is called $q$. It is shown in \cite[Section 4]{greenlees_rationalmackey} that $\omega$ is continuous (also proved by Fausk--Oliver~\cite{fauskoliver}). Since both $\sub(G)/G$ and $\Phi(G)$ are compact Hausdorff, it follows that  $\omega$ is also a closed map. To summarize:

\begin{lemma}
  \label{lem:omegacts}
The map $\omega \colon \sub (G)/G\lra \Phi G$ on subgroups is continuous and closed. 
\end{lemma}

We note that for a set $M$ of subgroups with finite Weyl group, $\omega$ determines the cotoral specialization. That is, there is an identity
    \begin{equation}\label{eq:preomega=lct}
    \omega^{-1}(M)=\Lct (M),
    \end{equation}
for any $M \subseteq \Phi G$.

\subsection{Reduction to the case with identity component a torus}\label{ssec:subgroups_reduction}

The richness of the topology on spaces of subgroups arise from the case of finite groups acting on tori, with the first and simplest non-trivial example being $O(2)$. We explain here how this class of groups captures the topology of $\Sub(G)/G$ around $G$. 

Many results of this type are proved by tom Dieck \cite{tomDieck79}. For example, he states (Proposition 5.6.3) that $G$ is isolated in $\Sub(G)$ if and only if $G$ is semisimple. We will prove a more precise result; our proof is a little more elaborate, but tom Dieck's proof seems to be incomplete (since it does not discuss the fact that a torus may be the limit of subtori of lower dimension). This gap is filled by appealing to \cref{lem:torusnormalizer} below.

We continue with the notational conventions of the previous subsection. In particular, we are given a presentation of the identity component of $G$ as
    \[
    G_e=\SSi \times_Z T
    \]
with $\SSi$ semisimple, $T$ a torus, and $Z$ a finite central subgroup. The image of $T$ in $G_e$ is the identity component of the centre of $G_e$ and hence is characteristic and invariant under the action of $G_d$.
 More precisely, we suppose $Z$ is a finite abelian group and we have maps $i\colon Z\lra \SSi$ and $j\colon Z\lra T$. By passage to quotients if necessary, we may suppose both $i$ and $j$ are injective and we write $Z_{\SSi}$ and $Z_T$ for their respective images. We write $\Delta Z=\{(iz,jz)\st z\in Z\} \subseteq \Sigma \times T$. 

Similarly, we note that 
\[
\SSi \lra \SSi \times T \lra \SSi\times_Z T
\]
is injective since $\Delta Z\cap \Sigma =1$, and write $\SSi$ for the image of the composite. 

\begin{lemma}
\label{lem:semisimplesubgroups}
If $\SSi'$ is a connected semisimple subgroup of $G$ then $\SSi'\subseteq \SSi$.
Consequently, the specific subgroup $\Sigma$ is the unique subgroup of $G$ with the same local type as  $\SSi$, and in particular it is a characteristic subgroup.
\end{lemma}

\begin{proof}
We have a diagram 
\[
\xymatrix{
1 \ar[r]&\SSi\ar[r]\ar@{=}[d] &\SSi\times T\ar[r]\ar[d]&T\ar[r]\ar[d]&1\\ 
1 \ar[r]&\SSi\ar[r] &\SSi\times_Z T\ar[r]&\Tbar \ar[r]&1 
}
\]
where $\Tbar=T/Z_T$. 
Since $\SSi'$ is semisimple, the composite $\SSi'\lra G_e\lra \Tbar$ is trivial and therefore $\SSi'\subseteq \SSi$.
\end{proof}

This will allow us to factor out semisimple subgroups in certain arguments, since the normalizer of a group normalizes any characteristic subgroup:

\begin{corollary}
\label{lem:semisimplenormalizers}
If $H$ is a subgroup of $G$ with $H_e=\SSi'\times_{Z'}T'$ then $N_G(H)\subseteq N_G(\SSi')$.
\end{corollary}

It will also be useful to have criteria for factoring out certain tori. 
\begin{lemma}
\label{lem:torusnormalizer}
Suppose $T\subseteq G$ is a torus and $T_i\subseteq T$ with $T_i\lra T$. If $i$ is sufficiently large we have $N_G(T_i)\subseteq N_G(T)$. 
\end{lemma}
\begin{proof}
It suffices to show that if $i$ is large enough and $g$ normalizes $T_i$ then it normalizes $T$. Since $T$ is connected, it is generated by a neighbourhood of its identity, and we can work in an open ball of radius larger than $\delta$ so that the exponential map is a homeomorphism $B_{LG}\cong B_G$ between the corresponding balls. Thus it suffices to show that the adjoint action $\ad(g)\colon LG \to LG$ takes the tangent space $LT$ to $LT$. If $T$ is of rank $s$ and we choose linearly independent elements $v_1, \ldots , v_s$ in $ LT$, since $\ad(g)$ is a linear isomorphism, their images will again be linearly independent in $LG$. It therefore suffices to show the $v_j$ can be chosen so that $\ad (g) (v_1), \ldots ,\ad(g) (v_s)$ lie inside $LT$.  

For this we will use the exponential map, and the fact that it is easy to understand within a torus. Since $\ad(g)$ is the derivative of the conjugation by $g$ map $c_g$ at the identity, and since every automorphism of a torus is linear, we have the following commutative diagram of abelian groups:
\[
\xymatrix{
&T^g&&LT^g\ar[ll]^{\exp}\\
T\ar[ur]^{c_g}&&LT\ar[ur]_{\ad(g)}\ar[ll]^<<<<<<<{\exp}&\\
&T_i^g\ar[uu]|-{\hole} &&LT_i^g\ar[ll]|-{\hole}^>>>>>>>{\exp}\ar[uu]\\
T_i\ar[ur]^{c_g}\ar[uu]&&LT_i.\ar[ur]_{\ad(g)}\ar[ll]^{\exp}\ar[uu]&\\
}
\]
We note that the exponential maps are local homeomorphisms within the balls $B_{V}=B_{LG}\cap V$ for vector subspaces $V$ of $LG$.

Choose coordinates and a metric on $LT$ so that it has orthonormal basis $e_1, \ldots , e_s$. By averaging we may assume $\ad(g)\colon LT \to LT^g$ is an isometry. We may find $\epsilon>0$ so that if $v_j\in LT$ is within $\epsilon $ of $\delta \cdot e_j\in B_{LT}$, then the $v_j$'s are linearly independent (any $\epsilon <\delta/2$ will do). Choose $N$ so that $d(T_i,T)<\epsilon$ for $i\geqslant N$. Thus there is a point of $T_i$ within $\epsilon$ of every point in $B_T$.

We may therefore choose points $x_j\in B_T\cap T_i$ so that they (or more properly their counterparts $v_j$ in $B_{LT}$) form a basis of $LT$. Since the exponential map is locally surjective, we may find $w_j\in LT_i$ so that $x_j=\exp(w_j)$. Note that $w_j$ will usually not be inside $B_{LT}$, but since $w_j$ and $v_j$ map to $x_j$ under the exponential map, they differ by the lattice $\ker (\exp \colon LT\lra T)$. 

Now if $g$ normalizes $T_i$ we have $T_i^g=T_i$ and so $\ad(g)(w_j)\in LT_i$ for all $j$, and hence 
\[
\exp (\ad(g)(v_j))=\exp (\ad(g)(w_j))\in \exp(LT_i)\subseteq \exp(LT)
\]
as required. A priori we only know $\ad(g)(v_j)\in B_{LG}$, but  $\exp$ is a homeomorphism on $B_{LG}$ and the images lie in $LT$, so it follows that $\ad(g)(v_j)\in B_{LG}\cap LT = B_{LT}$ as required. 
\end{proof}

\begin{proposition}\label{prop:topphi}\leavevmode
Consider the Hausdorff metric and the $h$-topology.
    \begin{enumerate} 
        \item If $G_e$ is semisimple then $G$ is isolated in $\sub(G)/G$.
        \item If $G_e=\SSi \times_Z T$ is connected with  $\SSi$ semisimple and $T$ a torus,  then any sequence of subgroups tending to $G$ in $\sub(G)/G$ consists of subgroups eventually containing $\SSi$ and their images converge to $G/\SSi$. 
    \end{enumerate}
\end{proposition}
\begin{proof} 
We argue by induction on the dimension of $G$. The statement is obvious if $G$ is of dimension 0. We suppose inductively that the result is proved for groups of lower dimension. 

First note that Statement (2) follows from Statement (1) for groups of lower dimension: the subgroup given by the image of $T$ in $G$ is normal so that if $H_i\lra G$ then $H_i/(H_i\cap T)\lra G/T$, and if $T\neq 1$ the group $G/T$ is of lower dimension with $(G/T)_e=\SSi/Z_{\SSi}$ semisimple so that by induction $H_i/(H_i\cap T)=G/T$ for $i\gg 0$ as required. 

Accordingly, it suffices to prove Statement (1): we assume that $G_e$ is semisimple and prove that it is isolated.

For the rest of this proof, we denote the identity component of a compact Lie group $K$ by $K^{e}$. Suppose $H_i\lra G$. By sequential compactness, we may pass to a subsequence and suppose $H_i^e\lra H$ for some connected subgroup $H$. Discarding a finite number of terms and conjugating we may suppose $H_i^e\subseteq H$ by \hyperref[item:mztheorem]{(MZ)}. Now suppose $H=\SSi'\times_{Z'}T'$ and write $\SSi'_i$ for the semisimple part of $H_i^e$. By \cref{lem:semisimplesubgroups} we have $\SSi'_i\subseteq \SSi'$. Since $\SSi'$ has only finitely many connected semisimple subgroups, each of which has finite centre we may discard terms and suppose there is a common value $\SSi''\subseteq \SSi'$ and a common central subgroup $Z''$ for each of the $H_i$'s. Considering the image modulo $T'$ we see $\SSi''=\SSi'$ and hence $H_i^e=\SSi' \times_{Z'}T'_i$. 

By \cref{lem:semisimplenormalizers}, $N_G(H_i^e)\subseteq N_G(\SSi'_i)=N_G(\SSi')$. The case that $N_G(\SSi')$ is smaller than $G$ is dealt with by induction, so we may assume $\SSi'$ is normal. Factoring out $\SSi'$ we may assume there is no semisimple component in the sequence so that  $H_i^e=T'_i$ and $H=T'$. 

Next, by \cref{lem:torusnormalizer} we may discard finitely many terms and assume $N_G(T'_i)\subseteq N_G(T)$. Again we may reduce the size of $G$ and assume $T$ is normal. Factoring out $T$ we may assume that $H_i^e$ and $H$ are the trivial group.  In other words we need only consider the special case of sequences of finite subgroups. Turing \cite{Turing38} proved that a connected group is a limit of finite subgroups if and only if it is a torus, so $G$ is not a limit of finite groups and it is isolated as required. 
\end{proof}

\begin{corollary}\label{cor:isolated_criterion}
    A compact Lie group $G$ is isolated in $\Sub(G)/G$ if and only if $G_e$ is semisimple.
\end{corollary}

\begin{proof}
\cref{prop:topphi} shows that if $G_e$ is semisimple then $G$ is isolated. Conversely if $G_e=\Sigma\times_Z T$ with $T$ non-trivial we note that $G_d$ acts on the characteristic subgroup $T$. The extension of $T\cong\R^s/\Z^s$ by $G_d$ is described by an element of  $H^2(G_d;T)\cong H^3(G_d;\Z^s)$, which is a torsion class as $G_d$ is a finite group. If it is of order $n$, the extension is supported on the finite subgroup $T[n]$ of $n$-torsion points on $T$, and we may suppose $G_F$ is the corresponding finite subgroup of $G$. Taking $r$ so that $Z$ is contained in $T[nr]$, then $G$ is the limit of the subgroups $G_r=\langle \Sigma \times_{Z} T[nr], G_F \rangle$ and hence not isolated.
\end{proof}

\begin{remark}
    There is recent independent work by Csik\'os, K\'atay, Kocsis, and P\'alfy \cite{CKKP2022pp} which contains a version of  \cref{cor:isolated_criterion}.
\end{remark}

\subsection{Finite groups acting on tori and full subgroups}\label{ssec:subgroups_finiteontorus}

We now suppose the identity component of $G$ is a torus so there is a short exact sequence 
    \[
    1\lra T \lra G\stackrel{\pi}\lra W\lra 1. 
    \]
\Cref{prop:topphi} has explained why this special case is significant.  The finite group $W$ acts on $T$ via a group homomorphism $W\lra \aut(T)$. We may describe the automorphism group by $\aut (T)=\aut (T^*) \cong GL_r(\Z )$, where $T^*$ denotes the Pontryagin dual of $T$, but it is more helpful to do this covariantly. Indeed, we may write $T=LT/\Lambda $ for a lattice $\Lambda$. By differentiation we have a representation of $W$ on $LT$, which gives the original action via the exponential map. Hence the action of $W$ is a representation of $W$ on $LT$ which restricts to a representation on $\Lambda $, so that $LT=\R \tensor \Lambda$. In particular the representation $LT$ is extended from the rationals. 

Since this is a rational representation,  by Maschke's Theorem we may write 
    \begin{equation}\label{eq:isotypical}
    \Q\tensor \Lambda =\bigoplus_{\alpha}V^{\Q}_{\alpha},
    \end{equation}
where $\alpha $ runs through simple rational representations and $V_{\alpha}^{\Q}$ is the $\alpha$-isotypical part and we suppose $\alpha=1$ gives the $W$-fixed  summand. We then write $V_{\alpha}=\R \tensor V_{\alpha}^{\Q}$. 

\begin{definition}\label{def:multiplicities}
    The (rational) multiplicities of $T$ are defined as the non-negative integers $m_{\alpha}$ such that $V^{\Q}_{\alpha} \cong \alpha^{\oplus m_{\alpha}}$, for $\alpha$ ranging through the simple $W$-representations.
\end{definition}

\begin{remark}\label{rem:integralwarning}
We can choose a basis of $V_{\alpha}^{\Q}$ consisting of elements of $\Lambda$, but we warn that $\Lambda $ itself may not decompose as a direct sum. For example if $W=\langle w\rangle$ is of order 2 then the short exact sequence
    \[
    0\lra J \lra \Z W\lra \Z \lra 0
    \]
does not split, although $1-w$ gives a basis of $\Q J$ in $\Q W$ and $1+w$ gives a basis of the $\Q$-orthogonal complement. Indeed, $\langle 1+w, 1-w\rangle$ spans an index 2 sublattice of $\Z W$.
\end{remark}

We continue to suppose that $G$ is a group with identity component the torus $T$ and component group $W$. The goal of the remainder of this section is to describe the topology on neighbourhoods of $G$ in $\sub(G)/G$ in terms of $LT$ as a representation of $W$. While they form an integral part of our understanding of the finer structure of $\Sub(G)/G$, this material is not strictly necessary for the applications to the category of $G$-spectra given in this paper. 

Let $H$ be a closed subgroup of $G$. First note that $S=H\cap T$ is normal in $H$, and hence we have an extension of the same form
    \[
    1\lra S \lra H\stackrel{\pi}\lra W\lra 1. 
    \]
Even the case when $G=O(2)$ shows that $S$ does not determine the  subgroup $H$: indeed, the space of subgroups isomorphic to a dihedral group $D_{2n}$ is homeomorphic to a circle. However if we also choose a function $\sigma\colon W\lra H$ splitting $\pi$,  then $S$ and $\sigma$ together specify $H$. Given $G$ and $S$ the condition that a section $\sigma\colon W \to G$ of $\pi\colon G \to W$ determines a subgroup $H$ is $\sigma(w)\sigma(w^{-1})\in S$ and $\sigma (v)\sigma(w)\sigma (vw)^{-1}\in S$.  In summary,  full subgroups of $G$ are given by a subgroup $S$ of $T$ invariant under $W$ and a section $\sigma$ to $\pi$ and we write $H(S, \sigma )$ for the corresponding subgroup of $G$. The identity component of $H(S, \sigma )$ is the identity component of $S$. 

Given a section $\sigma$ to $\pi$, an element $g \in G$ and an element $w \in W$, we write $g^w = \sigma(w)^{-1}g\sigma(w)$.

\begin{lemma}\label{lem:HSnorm}
The normalizer of a full subgroup $H(S, \sigma)$ is given by 
    \[
    N_G(H(S, \sigma))= H(S^+, \sigma ) \mbox{ where } S^+=\{t \st t^wt^{-1}\in S \; \forall \; w \in W\}.  
    \]
The subset $S^+$ of $T$ is a subgroup. 
\end{lemma}
\begin{proof}
Let $H=H(S, \sigma)$ where $\sigma\colon W\lra H$ splits $\pi$. Any group element $g\in G$ can be written $g=\sigma (w)t$ for $t\in T$ where $w=\pi (g)$. Evidently $H\subseteq N_G(H)$, so $g\in N_G(H)$ if and only if $t\in N_G(H)$. Now any $h\in H$ can be written $h=\sigma (v)s$ for $s \in S$, and we wish to find a necessary and sufficient condition on $t$ such that $t[\sigma (v)s]t^{-1} \in H$. Since conjugation by an element of $T$ does not change the image in $W$,  we have $\pi(t h t^{-1}) = \pi(h)=v$. In other words, there exists some $t' \in T$ with
    \[
    \sigma (v)t'=t[\sigma (v)s]t^{-1} =t[\sigma (v)]t^{-1}\cdot tst^{-1}.
    \]
Now $t \in N_G(H)$ if and only if $\sigma (v)t' \in H$ if and only if $t' \in S$. Since $tst^{-1}=s$,  the condition that $t$ lies in the normalizer of $H$ is thus:
    \[
    t^vt^{-1} =\sigma (v)^{-1}t\sigma (v)t^{-1} \in S. 
    \]
The statement that $S^+$ is a group follows from the fact  that $T$ is abelian. 
\end{proof}

\begin{remark}
For a subgroup $H=H(S,\sigma)$, $S=H\cap T$ is determined, but  $\sigma$ is only determined up to multiplication by elements of $S$: if $\tau\colon W\lra S$ then $\sigma $ and $\sigma \tau$ determine the same subgroup. 

The lemma actually shows that conjugation does not change $S$, and 
    \[
    tH(S,\sigma)t^{-1}=H(S, \sigma \tau_t)
    \]
where $\tau_t\colon W\lra T$ is defined by  $\tau_t (v)=t^vt^{-1}$. 
\end{remark}

We are particularly interested in subgroups with finite index in their normalizer. We note that if we have a containment $H\subseteq H'$ of full subgroups, then $S=H\cap T\subseteq H'\cap T=S'$ and we may choose the same section $\sigma$ so the containment is of the form $H=H(S, \sigma)\subseteq H(S', \sigma)=H'$. Such a containment occurs if and only if $S\subseteq S'$ and that in that case 
    \begin{equation}\label{eq:HSnorm}
    \begin{gathered}
        H(S, \sigma )\normal H(S', \sigma ) \mbox{ if $S'\subseteq S^+$ and then } H(S', \sigma)/H(S,   \sigma)\cong S'/S. 
    \end{gathered}
    \end{equation}
Indeed, $H(S, \sigma )$ is normal in $H(S', \sigma )$ if and only if $H(S', \sigma ) \subseteq H(S^+, \sigma )$ by \cref{lem:HSnorm}, which in turn is equivalent to $S' \subseteq S^+$ as just observed.

\begin{proposition}\label{prop:finweyl}
The subgroup $H(S, \sigma )$ is of finite index in its normalizer if and only if $\dim(S^W)=\dim( T^W)$. 
\end{proposition}
\begin{proof}
By \cref{lem:HSnorm}, we have $N_G(H(S, \sigma))= H(S^+, \sigma )$, so taking $S' = S^+$ in \eqref{eq:HSnorm}, we see that $H(S, \sigma )$ is of finite index in its normalizer if and only if $S^+/S$ is finite, or equivalently if $LS^+=LS$. From the description of $S^+$ in \cref{lem:HSnorm}, we get
    \[
    LS^+=\{ x \in LT \st x-wx\in LS \mbox{ for all } w\in W\}. 
    \]
Since $S$ is a $W$-invariant subgroup of $T$ we see
    \[
    LS=\bigoplus_{\alpha}U_{\alpha} \mbox{ with } U_{\alpha}\leqslant  V_{\alpha}.  
    \]
Picking a submodule $U_{\alpha}' \leqslant V_{\alpha}$  so that $V_{\alpha}=U_{\alpha}\oplus U'_{\alpha}$, this implies that if $x\in LS^+$ has components 
    \[
    (x_{\alpha}, x_{\alpha}')\in U_{\alpha}\oplus U_{\alpha}'=V_{\alpha}
    \]
then the requirement that $LS^+=LS$ gives $x_{\alpha}'-w x_{\alpha}'=0$ for all $w\in W$ and all $\alpha$. Thus $x_{\alpha}'=0$ for $\alpha \neq 1$ and 
    \[
    LS^+/LS=V_1/U_1. 
    \]
This in turn is the case if and only if $\dim(S^W)=\dim( T^W)$. 
\end{proof}

\section{Finiteness properties of the prism}\label{sec:finitenessproperties}

The identification of the prism in \cref{thm:prism} affords us a greater understanding of the topology of the Balmer spectrum for rational $G$-spectra. In particular, in \cref{prop:noetherian} we show that these tools fully resolve when the Balmer spectrum of $\SpGQo$ is a Noetherian topological space. This is valuable since models for tensor-triangulated categories are significantly simpler when the Balmer spectrum is Noetherian (see for example~\cite{adelicm,adelic1}). Moreover, we prove that the Balmer spectrum always satisfies the weaker finiteness condition of being generically Noetherian. In combination with results from \cite{greenlees_bs} and \cite{BHS2023}, this allows us to conclude that the category of rational $G$-spectra satisfies the telescope conjecture for any compact Lie group $G$.

\subsection{Noetherian spectra}\label{ssec:noetherianprism}

In this subsection, we give three characterizations of when the spectrum of rational $G$-spectra is a Noetherian space: in terms of the group structure of $G$, the spaces of subgroups of $G$, and the rational Burnside ring of $G$. An important precursor is the next result due to tom Dieck \cite[Proposition 5.8.10]{tomDieck79} which classifies the compact Lie groups that have only a finite number of conjugacy classes of subgroups with finite Weyl group. 

\begin{proposition}[tom Dieck]\label{prop:phifinite_tomdieck}
    For a compact Lie group $G$, the space $\Phi(G)$ is finite discrete if and only if the Weyl group action on the maximal torus of $G$ is trivial. 
\end{proposition}

As explained by tom Dieck, the equivalent conditions of this proposition guarantee that $G_e$ has no semisimple component. It follows that $G$ has finite $\Phi(G)$ if and only if it can be written as a central extension
    \begin{equation}\label{eq:centralextension}
        \xymatrix{1 \ar[r] & T \ar[r] & G \ar[r] & F \ar[r] & 1}
    \end{equation}
of a torus $T$ by a finite group $F$, where $F$ acts trivially on $T$: $G$ is a central extension of $T$ by $F$. We pause to consider the structure of such groups. We note that the extension is classified by an element $\epsilon$ of $H^2(F;T)\cong H^3(F; \Z^r)$. This is a finite abelian group so it is annihilated by some number $n$. Considering the short exact sequence
    \[
    \xymatrix{1 \ar[r] & T[n] \ar[r] & T \ar[r] & \Tbar \ar[r] & 1,} 
    \]
of coefficients, we see $\epsilon$ maps to zero in $H^2(F; \Tbar)$. Thus $G/T[n]=\Tbar \times F$ and $\epsilon$ comes from $H^2(F; T[n])$, and the non-trivial extension is carried by a finite subgroup. 

\begin{example}\label{ex:centralextension}
   If we take $F$ to be the Klein 4-group $V\cong C_2\times C_2$, and $T$ to be the circle group we see $H^2(V;T)=H^3(V;\Z)\cong \Z/2$. The non-zero element is of order 2, so the corresponding extension $\epsilon$
    \[
    \xymatrix{1 \ar[r] & T \ar[r] & G \ar[r] & V \ar[r] & 1,}
    \]
is prolonged from a central extension $\epsilon'$
    \[
    \xymatrix{1 \ar[r] & C_2 \ar[r] & H \ar[r] & V\ar[r] & 1,}
    \]
so that $G=H\times_{C_2}T$ for a central subgroup $C_2$.
Of course, $H^2(V; C_2)\cong \Z/2^3$, and its non-zero elements correspond to $C_4\times C_2$ (3 extensions), $D_8$ (2 extensions) and $Q_8$ (2 extensions). 
The three abelian ones become split on prolongation to $T$, and the remaining 4 become equivalent. 

Choosing the more symmetric description, $G = \{a + bi + cj + dk \st  a, b, c, d \in T \}$. The resulting group $G$ visibly has $Q_8$ as the subgroup with $a, b, c, d \in \{+1, -1\}$, which shows in particular that $G$ does not split as a direct product of $T$ and $V$.
\end{example}

\begin{lemma}\label{lem:phifinite_permanence}
    If $\Phi(G)$ is finite discrete, then so is $\Phi(H)$ for every closed subgroup $H$ of $G$.
\end{lemma}
\begin{proof}
    In light of tom Dieck's characterization \cref{prop:phifinite_tomdieck}, is suffices to observe that a closed subgroup of a central extension of a torus by a finite group is again of this form. 
\end{proof}

\begin{proposition}\label{prop:noetherian}
    The following are equivalent for a compact Lie group $G$:
        \begin{enumerate}
            \item $\Phi(G)$ is finite discrete;
            \item $(\Sub(G)/G)_f$ is discrete;
            \item $((\Sub(G)/G)_h, \cotoral)$ is a Noetherian Priestley space.
        \end{enumerate}
\end{proposition}
\begin{proof}\leavevmode

    $(1) \implies (2)$: The assumption combined with \cref{lem:phifinite_permanence} imply that $\Phi(H)$ is finite discrete for any closed subgroup $H$ in $G$. It then follows from \cref{lem:ftop} that $(\Sub(G)/G)_f$ is discrete as well. 

    $(2) \implies (3)$: By \cite[Theorem 8.1.11]{book_spectralspaces}, it is enough to check the descending chain condition on basic closed sets: a spectral space is Noetherian if and only if its collection of basic closed (i.e., complements of quasi-compact  opens) satisfies the descending chain condition with respect to inclusion. Consider such a sequence
    \[
    \cdots \subset Z_i \subset Z_{i-1} \subset \cdots \subset Z_1 \subset Z_0.
    \]
    For any basic closed subset $Z$, there then exists an $f$-open, $f$-compact subset $M$ with $Z = \Lct (M)$, see \cref{prop:basicclosed}. Since the $f$-topology is assumed to be discrete, this means that $M$ is in fact finite. Let $r$ be the rank of $G$. For the purposes of this proof, we then define the \emph{multirank} of $Z$ as the ordered $r$-tuple $\rho(Z) \coloneqq (m_r,m_{r-1},\ldots,m_0) \in \N^{\times (r+1)}$, where
    \[
    m_i \coloneqq \{K \in \max(Z)\colon \rank(K) = i\},
    \]
    the number of cotorally maximal elements of $Z$ of rank $i$. If we equip $\N^{\times (r+1)}$ with the lexicographical order, denoted $\prec$, then we have $\rho(Z_{i+1}) \prec \rho(Z_i)$ for all $i\geqslant 0$. Since the multirank does not take the value $\infty$ because $\max(Z_i)$ is always finite, this implies that the sequence $(Z_i)_{i\geqslant 0}$ has to be eventually constant, as required.

    $(3) \implies (1)$: We prove the contrapositive, so assume that $\Phi(G)$ is infinite. Since $\Phi(G)$ is Stone, it then cannot be Noetherian. In particular, there exists a strictly decreasing sequence of closed subsets $\cdots \subset W_2 \subset W_1 \subset W_0$ in $\Phi(G)$. Set $V_i \coloneqq \Lct(W_i) \subseteq (\Sub(G)/G)_{zf}$. Since $W_i$ is $h$-closed, we deduce from \eqref{eq:preomega=lct} that $V_i$ is an $h$-closed and cotorally closed subset of $\Sub(G)/G$. This provides an infinite strictly descending sequence
    \[
    \cdots \subset V_2 \subset V_1 \subset V_0
    \]
    of closed subsets of $(\Sub(G)/G)_{zf}$. Therefore, the space $(\Sub(G)/G)_{zf}$ is not Noetherian.
\end{proof}

We are now ready to put the pieces together, allowing us to effortlessly translate this proposition into a result about the category of rational $G$-spectra:

\begin{theorem}\label{thm:noetherian}
    For a compact Lie group, the following statements are equivalent:
        \begin{enumerate}
            \item the rational Burnside ring $A_{\Q}(G)$ is Noetherian (and hence Artinian); 
            \item the Balmer spectrum $\Spc(\SpGQo)$ is Noetherian;
            \item $G$ is a central extension of a torus by a finite group. 
        \end{enumerate}
\end{theorem}
\begin{proof}
    On the one hand, recall tom Dieck's identification of the rational Burnside ring as the ring of rational locally constant functions on $\Phi(G)$. Since $\Phi(G)$ is Stone, $A_{\Q}(G)$ is zero-dimensional and hence Noetherian if and only if it is Artinian, which in turn is the case if and only if $\Phi(G)$ is finite discrete. Therefore, Condition (1) is equivalent to \cref{prop:noetherian}$(1)$.

    On the other hand, \cite[Theorem 8.4]{greenlees_bs} provides a homeomorphism between $\Spc(\SpGQo)$ and $(\Sub(G)/G)_{zf}$. This shows the equivalence of $(1)$ and $(2)$ via \cref{prop:noetherian}. Finally, tom Dieck's characterization in \cref{prop:phifinite_tomdieck} combined with the discussion above establishes the equivalence $(2) \iff (3)$.
\end{proof}

\begin{corollary}\label{cor:noetherian}
The equivalent conditions of \cref{prop:noetherian} hold if and only if the set of maximal points of any closed subset $C \subseteq (\Sub(G)/G)_{zf}$ is finite. 
\end{corollary}
\begin{proof}
This is a consequence of the previous proposition via another characterization of Noetherian spectral spaces (\cite[Theorem 8.1.11]{book_spectralspaces}): A spectral space $X$ is Noetherian if and only if every non-empty closed subset of $X$ has only finitely many generic points and there are no infinite specialization chain. Since the length of any specialization chain in $(\Sub(G)/G)_{zf}$ is bounded by the rank of $G$, the claim follows.
\end{proof}

\subsection{Generically noetherian spectra and the telescope conjecture}\label{ssec:gennoetherianprism}

In \cref{prop:noetherian} we have given a characterization of when the spectrum of $\SpGQo$ is Noetherian: This happens if and only if $G$ is a central extension of the form \eqref{eq:centralextension}. In contrast to this, we now show that $\Spc(\SpGQo)$ always satisfies the weaker condition of being generically Noetherian that was introduced in \cite[Definition 9.5]{BHS2023} and is recalled below. As a consequence, we will deduce that $\SpGQo$ satisfies the generalized telescope conjecture for any compact Lie group $G$.

\begin{definition}\label{def:gennoetherian}
Let $X$ be a spectral space. Then $X$ is \emph{generically Noetherian} if for each $x \in X$, the generalization closure $\gen(x)$ of $x$ is a Noetherian subspace of $X$.
\end{definition}

\begin{proposition}\label{prop:spgqgen}
For any compact Lie group $G$ and closed subgroup $L$, the generalization closure of $(L)_G$ in $\Sub(G)/G$ is isomorphic to the set of conjugacy classes of subtori of $W_G(L)$ with the topology generated by the specialization closures of points. Accordingly, $\Spc(\SpGQo)$ is generically Noetherian and all points are weakly visible. 
\end{proposition}
\begin{proof} 
The generalization closure of $(L)_G$ in $\Sub(G)/G$, consists of conjugacy classes of subgroups $K \subseteq G$ with $L$ cotoral in $K$:
    \[
        \gen_G(L) = \{(K)_G \in \Sub(G)/G \mid L \cotoral K \}.
    \]
Pick such a subgroup $K$ and consider the resulting morphism of extensions:
    \[
        \xymatrix{1 \ar[r] & L \ar[r] \ar@{=}[d] & K \ar[r] \ar@{^{(}->}[d] & T \ar[r] \ar@{-->}[d] & 1 \\
        1 \ar[r] & L \ar[r] & N_G(L) \ar[r] &  W_G(L) \ar[r] & 1.}
    \]
This exhibits $T$ as a connected subgroup of a maximal torus $TW$ of $W_G(L)$. Let $\TTW$ be the inverse image of $TW$ in $N_G(L)$, so we obtain a map $\xi\colon \gen_{\TTW}(L) \to \gen_{TW}(1)$. This map is a bijection with inverse given by the pullback of a connected subgroup of $TW$ along $N_G(L) \to W_G(L)$. We claim that $\xi$ is in fact a homeomorphism. 

To this end, consider the following commutative diagram of spaces with the Hausdorff metric topology and its quotients,
\[
\xymatrix{
\sub(\TTW |\supseteq L)\ar[r] \ar[d] &\sub(\TTW |\supseteq L)/\TTW \ar[d] \ar[r]& \sub(G)/G\\
\sub(TW)\ar@{=}[r] & \sub(TW)/TW. &}
\]
Invoking \hyperref[item:mztheorem]{(MZ)} and \cref{lem:pstarco} applied to the quotient map $\TTW \twoheadrightarrow TW$, we see that the maps in the square are continuous and open. The commutative square part of the diagram then shows that $\xi$ is a homeomorphism, as claimed. 

Since a subspace of a Noetherian space is Noetherian, it follows from \cref{prop:noetherian} that $\gen_{TW}(1)$ and hence $\gen_{\TTW}(L)\subseteq \sub(\TTW)/\TTW$ are Noetherian. This implies that $\gen_G(L)$, the image of $\gen_{\TTW}(L)$ under the  map $\sub(\TTW)/\TTW \to \sub(G)/G$, has the same property. Via \cref{thm:prism}, we deduce that $\Spc(\SpGQo)$ is generically Noetherian. For the final part use that, by \cite[Lemma 9.9]{BHS2023}, a generically Noetherian spectral space is weakly Noetherian. 
\end{proof}

For convenience, we recall the formulation of the (generalized) telescope conjecture (aka smashing conjecture), motivated by Ravenel's telescope conjecture for $\Sp$ given in \cite{Ravenel1984}.

\begin{definition}\label{def:telescopeconjecture}
Let $\sfT$ be a rigidly-compactly generated tensor-triangulated category. We say that 
the \emph{telescope conjecture} holds for $\sfT$ if the kernel of every smashing localization of $\sfT$ is generated by compact objects. 
\end{definition}

\begin{corollary}\label{cor:telescopeconjecture}
For any compact Lie group $G$,  the telescope conjecture holds for the category of rational $G$-spectra.
\end{corollary}
\begin{proof}
    By \cite[Theorem 9.11]{BHS2023}, we need to verify that $\SpGQ$ is stratified and has generically Noetherian spectrum. The latter condition holds by \cref{prop:spgqgen}, while the former condition was shown in \cite[Theorem 1.6]{greenlees_bs}, as interpreted in the language of stratification in \cite[Theorem 12.22]{BHS2023}.
\end{proof}

\section{Dispersions for rational \texorpdfstring{$G$}{G}-spectra}\label{sec:dispersionsforGspec}

In order to apply the general machinery of \cref{part:prisms}, we need to show that the category of rational $G$-spectra is dispersible, and to make this useful we are required to identify a specific dispersion. We begin with the observation that the dimension of subgroups forms a finite dispersion, which follows essentially from the identification of the prism in \cref{thm:prism}.  With a view towards the categorical local-to-global principle, however, it becomes very useful to work with an intrinsic dispersion. As explained in \cref{sec:dispersions}, the best possible candidate dispersion is constructed by repeatedly removing Thomason points as in \cref{defn:cbrank}. From the dimension dispersion we readily deduce that  the Thomason height provides a finite dispersion for any compact Lie group $G$. The main theorem is then a formula for the Thomason height of each subgroup $H \subseteq G$ in terms of representation-theoretic data. In fact, our arguments shows more: the Thomason filtration coincides with the Cantor--Bendixson filtration on $(\Sub(G)/G)_h$ and hence is independent of the cotoral order. 

\subsection{Overview}

The abstract theory of prisms and dispersions developed in \cref{part:prisms} supplies two different height functions on $\Prism(\SpGQo)$:

\begin{definition}\label{def:heights_subg}
    For $G$ a compact Lie group, we have two functions on the prism of $\SpGQo$, as follows:
        \begin{enumerate} 
            \item $\thht^G = \height_{\Prism(\SpGQo)}^G$ denotes the \emph{Thomason height} on $\Prism(\SpGQo)$, introduced for any Priestley space in \cref{defn:cbrank};
            \item $\cbht^G$ is the Thomason height on $((\Sub(G)/G)_h,=)$, i.e., the \emph{Cantor--Bendixson rank} on the Stone space $(\Sub(G)/G)_h$, see \cref{ex:cantorbendixsonisthomason}.
        \end{enumerate}
\end{definition}

Via our identification of $\Prism(\SpGQo)$ in \cref{thm:prism}, we can make the filtrations corresponding to these functions explicit: while the Thomason filtration iteratively removes subgroups from $\Sub(G)/G$ which are cotorally minimal and $h$-isolated, the Cantor--Bendixson filtration instead iteratively removes those subgroups which are $h$-isolated only. 

In fact we show the filtrations are the same and give a dispersion, and one can give an explicit formula for the filtration of a subgroup in terms of rational representation theory. To this end, let $H$ be a closed subgroup of $G$ and write $T(H) = Z(H_e)_e$ for the identity component of the center of $H_e$. The finite component group $H_d$ of $H$ acts on $H_1(T(H);\Q)$, and we may decompose the resulting $H_d$-representation 
    \begin{equation}\label{eq:isotypicaldecomposition}
        H_1(T(H);\Q) \cong \bigoplus_{\alpha} V_\alpha(H)
    \end{equation}
into isotypical pieces $V_\alpha(H) \cong \alpha^{\oplus m_{\alpha}}$ as in \cref{ssec:subgroups_finiteontorus}.
 
\begin{definition}\label{def:height_rep}
    For a compact Lie group $G$, we define the value of $\repht^G$ on some $H \in \Sub(G)$ as the number $\sum_{\alpha}m_{\alpha}$ of simple summands in the rational $H_d$-representation $H_1(Z(H_e)_e;\Q)$. 
\end{definition}

\begin{example}\label{ex:height_rep}
    Every torus in $G$ is conjugate to a subgroup of the maximal torus, so the range of $\repht^G$ is bounded above by the rank of $G$. This bound is achieved by the maximal torus itself. On the other hand, not all relevant simple modules are one dimensional: for example the value of $\repht^G$ on the normalizer of the maximal torus in $SU(3)$ is 1, even though the maximal torus is of dimension 2.
\end{example}

The promised height formula is given by the next result, whose proof we defer to \cref{ssec:proof_heightformula}: 

\begin{theorem}(Height formula)\label{thm:heightformula}
    For any compact Lie group $G$, there are equalities of dispersions on $\Prism(\SpGQo)$: 
        \[
            \thht(H) = \cbht(H) = \repht(H)
        \]
    for any $H \in \Sub(G)$. In particular, $\SpGQ$ is amenable in the sense of \cref{def:ttdispersible}.
\end{theorem}

\subsection{Dispersibility}\label{ssec:dispersible}

We begin with the dispersibility of the prism of finite rational $G$-spectra, relying on the identification of $\Prism(\SpGQo)$ given in \cref{thm:prism}:

\begin{proposition}\label{prop:dim_dispersion}
    The dimension of subgroups defines a dispersion on $\Prism(\SpGQo)$. In particular, $\SpGQ$ is finitely dispersible in the sense of \cref{def:ttdispersible} and thus satisfies the categorical local-to-global principle. 
\end{proposition}
\begin{proof}
    The first condition of a dispersion holds since if $K$ is cotoral in $H$ then the dimension of $K$ are strictly less than that of $H$. Suppose then that $P$ is an $h$-closed subset of $\Sub(G)/G$ and $H$ is not $h$-isolated. Then there is a sequence of subgroups $K_i\lra H$, and by \hyperref[item:mztheorem]{(MZ)} we can suppose that $K_i\subseteq H$. Replacing $K_i$ by $(K_i)_\e $ we may suppose $H$ is connected. A proper subgroup of a connected compact Lie group has lower dimension so dimension certainly satisfies the second condition for a dispersion. 
\end{proof}

\begin{remark}\label{rem:rank_dispersion}
    The rank of subgroups of a compact Lie group $G$ defines a finite dispersion of $\Prism(\SpGQo)$ as well. This can be seen as in the proof of \cref{prop:dim_dispersion}. Indeed, the rank of a proper cotoral subgroup $K$ in $H \in \Sub(G)$ is strictly less than that of $H$. 
    
    To verify the second condition of a dispersion for the rank, write $H=\SSi \times_Z T$ and let $K_i \lra H$ is a sequence of subgroups converging to $H$. As above, we can reduce to the case that $H$ is connected and $K_i\subseteq H$. Note that by \cref{prop:topphi}, after discarding finitely many terms we then have $\SSi\subseteq K_i$. Considering the image of $K_i$ in $T$ we see $\rank (K_i)<\rank (H)$ as required. 
\end{remark}

\subsection{Proof of the height formula}\label{ssec:proof_heightformula}

It is a consequence of \cref{prop:dim_dispersion} and \cref{lem:thomasinfiltisdisp} that the Thomason height defines a finite dispersion on $\Prism(\SpGQo)$. In order to prepare for the proof of the height formula, we first observe that both the Thomason height and Cantor--Bendixson height functions do not depend on the ambient group:

\begin{lemma}\label{lem:height_restriction}
    Let $G$ be a compact Lie group and $H$ a subgroup of $G$. The inclusion $\Sub(H)/H \hookrightarrow \Sub(G)/G$ induces an equality of heights
        \[
            \thht^H(K) = \thht^G(K) \quad \text{and} \quad \cbht^H(K) = \cbht^G(K)
        \]
    for all subgroups $K \in \Sub(H)$. 
\end{lemma}
\begin{proof}
    We argue by induction on the height, supposing by induction that for subgroups of $H$ with height $<s$ (in either sense) their $H$-height agrees with their $G$-height. The induction starts since this is vacuously true for $s=0$, as there are no subgroups of height $-1$. We will abbreviate the respective prisms by $\pmb{P}(G) = \Prism(\SpGQo)$ and likewise for $H$. By a mild abuse of notation, we write $\pmb{P}(G)_{\geqslant s}$ for the $s$-th filtration step with respect to either height function. 

    Suppose then that we have proved the condition as stated, and we want to show it for height $s$. In other words we need to show that if $K\subseteq H$ such that $(K)_H \in \pmb{P}(H)_{\geqslant s}$ and $(K)_G \in \pmb{P}(G)_{\geqslant s}$, then $(K)_H$ is $h$-isolated (and cotorally minimal for $\thht$) in $\pmb{P}(H)_{\geqslant s}$ if and only if $(K)_G$ is $h$-isolated (and cotorally minimal for $\thht$) in $\pmb{P}(G)_{\geqslant s}$. 

    By construction of the Thomason filtration, the pieces fit into a diagram
        \begin{equation}\label{eq:compatiblefiltrationsquare}
            \vcenter{
            \xymatrix{\pmb{P}(H) \ar[r]^-{\iota} & \pmb{P}(G) \\
            \pmb{P}(H)_{\geqslant s} \ar@{-->}[r]_-{\iota_{\geqslant s}} \ar@{^{(}->}[u] & \pmb{P}(G)_{\geqslant s}. \ar@{^{(}->}[u]}
            }
        \end{equation}
    The vertical maps are inclusions of Priestley spaces, see \cref{cor:strataissub}. The induction hypothesis implies that the continuous map $\iota\colon \pmb{P}(H) \to \pmb{P}(G)$ restricts to $\pmb{P}(H)_{\geqslant s}$ as in the bottom horizontal map, rendering the square commutative. We now will discuss the two relevant conditions separately. 

    First, we consider $h$-isolation. The continuity of $\iota_{\geqslant s}$ shows immediately that if $K$ is isolated in $\pmb{P}(G)_{\geqslant s}$ it is isolated in $\pmb{P}(H)_{\geqslant s}$.  Conversely, suppose $(K_i)_G \to (K)_G$ is a sequence in $\pmb{P}(G)_{\geqslant s}$ convergent to $K \subseteq H$. By the Montgomery--Zippin Theorem \hyperref[item:mztheorem]{(MZ)}, we may discard finitely many initial terms and conjugate so that we can arrange that $K_i\subseteq K$. This shows the sequence of subgroups lies in $H$ so $K_i\lra K$ in $\Sub(H)$ and hence by passing to conjugacy classes also in $\Sub(H)/H$. By induction hypothesis, $K_i\in \pmb{P}(H)_{\geqslant s}$, thereby showing that $(K)_H$ is not isolated in $\pmb{P}(H)_{\geqslant s}$, as required.

    Second, we consider cotoral minimality. Certainly if $(K)_H$ is not cotorally minimal in $\pmb{P}(H)_{\geqslant s}$, we may choose representatives and find a proper cotoral subgroup $L$ of $K$ of height $\geqslant s$ in $\Sub(H)/H$. By induction $L$ is still of height $\geqslant s$ in $\pmb{P}(G)$, so $(K)_G$ is not cotorally minimal. Similarly, if  $(K)_G$ is not cotorally minimal we may again choose $L$ cotoral in $K$ of height $\geqslant s$ in $\Sub(G)/G$. Since $L\subseteq K$, it follows that $L$ is a cotoral subgroup of $H$ which by induction $L$ is of height $\geqslant s$ in $\pmb{P}(H)$.    
\end{proof}

This lemma allows us to remove the subgroups from the notation for the Thomason height and Cantor--Bendixson height functions, a practice we will adopt for the remainder of this section to simplify notation. 

\begin{proof}[Proof of \cref{thm:heightformula}]
    We will prove this theorem by induction on the height $s$ of subgroups $H$ in $G$, beginning with the case $s=0$. Because the Thomason height provides a finite dispersion of $\pmb{P}(G) = \Prism(\SpGQo)$ and hence reaches all subgroups of $G$, this will imply the theorem. Consider the following three conditions on a subgroup $H \in \Sub(G)$:
        \begin{enumerate}
            \item[(0.a)] $\thht(H) = 0$, i.e., $H$ is cotorally minimal and $h$-isolated in $\Sub(G)/G$;
            \item[(0.b)] $\cbht(H) = 0$, i.e., $H$ is $h$-isolated in $\Sub(G)/G$;
            \item[(0.c)] $\repht(H) = 0$, i.e., $Z(H_e)$ is finite or, equivalently, $H_e$ is semisimple. 
        \end{enumerate}
    The implication $(0.a) \implies (0.b)$ is obvious, while the converse follows from the observation that an $h$-isolated subgroup $H$ is cotorally minimal. Indeed, if $H/K$ is a (non-trivial) torus we may let $H\{n\}$ be the inverse image of the $n$-torsion points of $H/K$, so we see $H\{n\}\to H$, hence $H$ is not $h$-isolated. The equivalence between $(0.b)$ and $(0.c)$ has essentially been established in \cref{ssec:subgroups_reduction}: $H$ is $h$-isolated in $\Sub(G)/G$ if and only if it is isolated in $\Sub(H)/H$ (for example by \cref{lem:height_restriction}), so \cref{cor:isolated_criterion} applies. 
    
    The base of the induction also serves as a blueprint for the general case. Let $s \geqslant 1$ and suppose that the statement of the theorem has been verified for all subgroups of $G$ of height less than $s$. By induction hypothesis, we can write $\pmb{P}(G)_{\geqslant s}$ unambiguously for the $s$-th filtration step of $\pmb{P}$ with respect to any of three height functions under considerations. Given $H \in \pmb{P}(G)_{\geqslant s}$, we need to show that the following conditions are equivalent:
        \begin{enumerate}
            \item[(s.a)] $\thht(H) = s$, i.e., $H$ is cotorally minimal and $h$-isolated in $\pmb{P}(G)_{\geqslant s}$;
            \item[(s.b)] $\cbht(H) = s$, i.e., $H$ is $h$-isolated in $\pmb{P}(G)_{\geqslant s}$;
            \item[(s.c)] $\repht(H) = s$, i.e., the number of simple summands in the representation  \eqref{eq:isotypicaldecomposition} equals $s$.
        \end{enumerate}
    Let us start with a few preliminary reductions. Take a subgroup $H$ in $G$ for which we wish to show that the three notions of height coincide. By \cref{lem:height_restriction}, we may work inside $\Sub(H)/H$, in effect reducing to the case $H=G$. Suppose then we have a group $G$ where $Z(G_e)$ has identity component $T= T(G)$. The finite group $W=G_d$ acts on $H_1(T; \Q)$, with decomposition into isotypical components $H_1(T;\Q) \cong \bigoplus_{\alpha} V_\alpha$ as in \eqref{eq:isotypicaldecomposition}. Next, note that if $G_e=\SSi \times_Z T$  then by \cref{prop:topphi}, the only subgroups that can either approach $G$ or be cotoral in $G$ are ones containing $\SSi$, and we implicitly restrict attention to the subspace $\fX_\SSi$ of these conjugacy classes. 

    It is clear that $(s.a)$ implies $(s.b)$. Next, we will establish the implication $(s.b) \implies (s.c)$ by verifying its contrapositive. So suppose that $\repht(G) \geqslant s+1$ and let us show that $G$ is then not $h$-isolated in $\pmb{P}(G)_{\geqslant s}$ by constructing a sequence of subgroups converging to it. We note that $W$ preserves $T$ and omit $\SSi$, considering the toral subgroup $G_t=H(T,\sigma)$ with $\sigma$ a section of the canonical map $G \to G_d$, using the notation as in \cref{ssec:subgroups_finiteontorus}. By construction, the component group of $G_t$ coincides with that of $G$ and there is an isomorphism $H_1(T(G_t);\Q) \cong H_1(T(G);\Q)$ of representations. We claim that $G_t$ is not $h$-isolated. It then follows that $G$ is not $h$-isolated in $\pmb{P}(G)_{\geqslant s}$ either, since it is the image of $\SSi \times G_t$ as in \cref{prop:topphi}. 

    We can choose some simple $G_d$-representation $\beta$ with $V_\beta \cong \beta^{\oplus m_{\beta}} \neq 0$, which exists because $s \geqslant 0$. There is then a decomposition $V=V_\beta'\oplus V_\beta''\oplus U$ into subspaces, where $V_\beta''\cong \beta$ is a simple summand of $V_{\beta}$ with complement $V_\beta'$ and $U \cong \bigoplus_{\alpha \neq \beta} V_\alpha(H)$. Now choose $W$-invariant subgroups $T_\beta', T_\beta'', S\leqslant T$ with $H_1(T_\beta'; \Q)=V_\beta'$,  $H_1(T_\beta''; \Q)=V_\beta''$ and $H_1(S;\Q)=U$. Thus $T$ is generated by the three tori $T', T''$ and $S$: there is a surjective map $T'\times T''\times S\lra T$ which has a finite kernel. Now take $K_i=H(T_i, \sigma)$ where $T_i$ is the image of $T'\times (T''[i]) \times S$, and note $\repht(K_i) = \repht(H_t)-1 \geqslant s$, i.e., $K_i \in \pmb{P}(G)_{\geqslant s}$. Since $T'\times (T''[i])\times S\lra T'\times T''\times S$, we see $K_i\lra G_t$, so that $G_t$ is not $h$-isolated in $\pmb{P}(G)_{\geqslant s}$.

    Finally, assume $(s.c)$ holds for $G$, i.e., that $\repht(G) = s$. It follows that $G$ is cotorally minimal in $\pmb{P}(G)_{\geqslant s}$ since if $K$ is any subgroup which is properly cotoral in $G$, then $V_\alpha(K)\leqslant V_\alpha (G)$ with strict inequality for some $\alpha$.
    Thus $\repht(K)<\repht(G)$, meaning $K \neq \pmb{P}(G)_{\geqslant s}$ by induction hypothesis. Similarly, if $G$ was not $h$-isolated in $\pmb{P}(G)_{\geqslant s}$, then there is a sequence of subgroups $K_i\lra G$ with $K_i\in \pmb{P}(G)_{\geqslant s}$. This means $V_\alpha(K_i)\leqslant V_\alpha (G)$ for all $\alpha$. If there is strict inequality for some $\alpha$ then $\repht(K_i)<s$, which is not permitted, so there is equality at every stage and hence $Z((K_i)_e)$ has identity component $T$. Since $\SSi\subseteq K_i$, we get $(K_i)_e=G_e$ and therefore $K_i=G$ eventually. Therefore, $\thht(G) = s$ proving $(s.a)$ for $G$, as desired.
\end{proof}

\section{Categorical local-to-global for rational \texorpdfstring{$G$}{G}-spectra}\label{sec:catstratforGspectra}

In \cref{sec:dispersionsforGspec}  we have determined that
$\SpGQ$ is dispersible, so we can  apply the results of
\cref{sec:catstrat_strict} and build categorical models for all compact Lie
groups $G$. The first step is understanding the role of the local
factors $\Gamma_H \SpGQ \simeq \Lambda_H \SpGQ$ where $H$
is a closed subgroup of $G$.
To describe these local factors and their algebraic models we need  to 
recall some equivariant topology.

\subsection{Free and cofree $G$-spectra}
From \cite{greenlees_bs} we know that the support of an object $X \in
\SpGQ$ is its geometric isotropy. That is, $\supp(X)
= \{H \in \Sub(G)/G \mid \Phi^H(X) \not\simeq_1 0 \}$. In particular,
it follows that $\Gamma_H \SpGQ$ consists of those spectra with
geometric isotropy exactly at $H$ (which was denoted $\SpGQ
\langle H \rangle$ in \cite{greenlees_bs}).

A spectrum $X$ is said to be \emph{free} if the natural map $EG_+ \wedge
X \xrightarrow{\simeq} X$ is an equivalence, equivalently if it is in
the localizing subcategory generated by $G_+$. Similarly, $X$ is said
to be \emph{cofree} if the natural map $X \xrightarrow{\simeq}
F(EG_+,X)$ is an equivalence, equivalently if it is in the
colocalizing subcategory generated by $G_+$.\footnote{Among other names, cofree spectra are also known as \emph{Borel-complete spectra}.}
The well-known symmetric monoidal equivalence between free and cofree spectra
\[
\Sp_{G}^{\free} \simeq \Sp_{G}^{\cofree}
\]
may be viewed as an equivalence between right and left Bousfield localizations with respect to $G_+$, but is most simply seen using the adjunction between $F(EG_+, -) $ and $ EG_+\sm (-)$. 

To describe the algebraic models, recall that $G_e$ denotes the
identity component of $G$ and $G_d$ the finite group of components. 
Combining \cite{GSfree} with \cite{dwyer_complete_2002}, we see there are equivalences
\begin{equation}\label{eq:cofreemodel}
  \SpGQ^{\cofree}\simeq \SpGQ^{\free} \simeq \Gamma_I \DdgMod{H^\ast (BG_e)[G_d]} \simeq \Lambda_I\DdgMod{H^\ast (BG_e)[G_d]}.
\end{equation}
Here $\Gamma_I$ and $\Lambda_I$ refer to  the derived algebraic torsion and completion with respect to
the augmentation ideal $I$ of the  polynomial ring $H^\ast
(BG_e)$. 

The cellularization $\Gamma_I\DdgMod{H^\ast (BG_e)[G_d]}$ is equivalent to the derived category of the abelian category of $I$-power torsion modules, which does not admit a monoidal model structure. The category of derived-complete modules is not abelian, but it is proved in \cite{PWcofree} that the algebraic
category on the right of \eqref{eq:cofreemodel} is equivalent to the derived
category of an abelian category. Writing $L_0^I$ for the $0$-th derived functor of
completion, we say a module is \emph{$L^I_0$-complete} if the natural map $M\lra L_0^I M$ is an equivalence. The
category of $L^I_0$-complete modules is a monoidal abelian category
and  there is a symmetric monoidal equivalence:
\begin{equation}\label{eq:compandlcomp}
\Lambda_I \DdgMod{H^\ast (BG_e)[G_d]}\simeq 
\DdgMod[L_0^I \text{-} \comp]{H^\ast (BG_e)[G_d]}
\end{equation}
such that the composite equivalence from $\Sp^{\cofree}_{G}$ to the right hand side of \eqref{eq:compandlcomp} is monoidal \cite{PWcofree}.

\subsection{The local factors}

We obtain a description of the local factors by using standard change
of groups isomorphisms and then applying the description of free and
cofree categories to appropriate Weyl groups.

\begin{proposition}
    Let $G$ be a compact Lie group with closed subgroup $H$. There are symmetric monoidal equivalences of categories
    \[
    \Sp_{W_G(H)}^{\free} \simeq \Gamma_H \SpGQ  \simeq  \Lambda_H \SpGQ \simeq  \Sp_{W_G(H)}^{\cofree}
    \]
    where $W_G(H)=N_G(H)/H$. 
\end{proposition}

\begin{proof}
Support corresponds to geometric isotropy, and in view of the
construction $\Lambda_HX=F(\Gamma_H1, X)$, the counterpart for
cosupport works similarly. 

  The equivalences
  \[
  \Gamma_H\SpGQ\simeq \Gamma_H\Sp_{N_G(H)}\simeq
  \Gamma_1\Sp_{W_G(1)}
  \]
  are given in \cite[\S 4.3]{greenlees_bs}. The functors from left to right are first the
  forgetful functor and then  geometric fixed points.  
  
  The equivalence between complete and torsion modules
  $\Gamma_1 \SpGQ  \simeq  \Lambda_1 \SpGQ$ was described above.
\end{proof}

We will focus on the complete form of the local factors because of their 
better monoidal properties.

Assembling the above results, we have algebraic models of the complete
local factors. 
\begin{corollary}\label{cor:alg_models_for_lambda}
      Let $G$ be a compact Lie group with closed subgroup $H$. Then there are symmetric monoidal equivalences
        \begin{align*}
        \Lambda_H \SpGQ &\simeq \Sp_{W_G(H)}^{\cofree} \\
        &\simeq \Lambda_I \DdgMod{H^\ast  (W_G(H)_e)[W_G(H)_d]} \\
        &\simeq 
\DdgMod[L_0^I \text{-} \comp]{H^\ast (W_G(H)_e)[W_G(H)_d]} .
    \end{align*}
\end{corollary}

\subsubsection{Finite groups}
For a finite group $G$ we have
$\Prism(\SpGQo) =((\Sub(G)/G)_h, =)$. This is  finite and discrete, so
the dispersion on $\Prism(\SpGQo)$ is of height zero, and the
categorical local-to-global contains no splicing data. Therefore we obtain a symmetric monoidal model as the product of the local factors:
    \[
    \SpGQ \simeq \prod_{(H)} \Lambda_H \SpGQ.
    \]
    Explicitly, an object of the model consists of:
    \begin{enumerate}
        \item A collection of objects $\{V_H \} \in \prod_{(H)} \Lambda_H \SpGQ$
    \end{enumerate}
    with no splicing data required.

    \begin{remark}
      By \cref{cor:alg_models_for_lambda}, we see
      $\Lambda_H\Sp_{G,H}=\sfD (\mathrm{dg} \text{-} \Mod_{\Q[ W_G(H)]})$ and we have
    \[
    \SpGQ \simeq \prod_{(H)}\sfD (\mathrm{dg} \text{-} \Mod_{\Q[ W_G(H)]})
    \]
    where the $G$-spectrum $X$ corresponds to $V_H=\pi_*(\Phi^HX)$.

    This algebraic model is due to Greenlees--May~\cite[Appendix
    A]{greenlees_may} at the derived level, with
    Barnes~\cite{barnesfinite} and K\c{e}dziorek \cite{kedziorekexceptional} upgrading it to a monoidal Quillen equivalence. 
\end{remark}

\subsubsection{The circle}\label{subsec:circle} 
The first non-trivial example is the circle group $G = \T$.  By
\cref{thm:prism} the prism is given by
$\Prism(\Sp^{\omega}_{\T})\cong (\Sub (\bT)_h, \cotoral )$. Closed
subgroups of the circle are classified by their order, and a such we
see that $\Sub (\bT)_h \cong (\N^+)^\ast$ as in \cref{ex:s1topologies}. The Thomason dispersion provides us with
\begin{itemize}
    \item $\Prism(\Sp^{\omega}_{\T})_0  = \{C_i\}_{i \geqslant 1}$;
    \item $\Prism(\Sp^{\omega}_{\T})_1  = \{SO(2)\}$.
\end{itemize}  

The strata are therefore given by
\[(\Sp_{\T})_0  = \prod_{i \geqslant 1} \Lambda_{C_i} (\Sp_{\T}) \quad \text{and} \quad (\Sp_{\T})_1 \simeq \Lambda_{\bT} (\Sp_{\T}).
\]
Therefore, applying \cref{lem:homotopy-cartesian-height1-complete}, we see that $\Sp_{\T}$ is  equivalent to the homotopy limit of the diagram
\begin{equation}\label{eq:t-model}
\begin{gathered}
\xymatrix{
& \prod_{i\geqslant 1} \Lambda_{C_i} (\Sp_{\T})  \ar[d]\\
\Fun([1],\Lambda_{\bT} (\Sp_{\T})) \ar[r]_-{\pi_1} & \Lambda_{\bT} (\Sp_{\T}) \rlap{.}
}
\end{gathered}
\end{equation}
The vertical map takes the product of objects and then the functor
$\tilde{E} \mathcal{F} \wedge -$ where $\mathcal{F}$ is the collection
of finite subgroups, so the unit is the $\cF$-Tate construction of the
sphere spectrum. In conclusion, an object of the model is given by the following data:
\begin{enumerate}
    \item An object $V \in \Lambda_{\bT} (\Sp_{\T})$;
    \item A collection of objects $\{N_i\} \in \prod_{i\geqslant 1} \Lambda_{C_i} (\Sp_{\T}) $;
    \item A morphism  $V \to \tilde{E}\mathcal{F} \wedge \prod_i N_i$ in $\Lambda_{\bT} (\Sp_{\T})$.
\end{enumerate}

\begin{remark}\label{rem:model_is_different}
As in the finite group case, we can pass to the algebraic models for
the local factors. The information required to do so is that
$W_\T(C_i) = \T/C_i\cong \T$ and $W_\T(\T)$ is the trivial group. Doing so gives the following algebraic model
\begin{equation}\begin{gathered}
\xymatrix{
& \prod_{i\geqslant 1} \DdgMod[L_0^I \text{-} \comp]{\Q[c]} \ar[d]\\
\Fun([1], \sfD(\Q)) \ar[r]_-{\pi_1} & \sfD (\Q),
}
\end{gathered}
\end{equation}
where $c$ lives in degree -2. The vertical map takes the product of $L^I_0$-complete modules and then
inverts the multiplicatively closed set
\[
\cE=\{c^v \;|\; v \colon \Nvee\lra \Z_{\geqslant 0} \mbox{ zero almost everywhere}
\}\subseteq \prod_{i\geqslant 1}\Q[c]
\]
of \emph{Euler classes}. Thus $V$ is a rational vector space, $N_i$ is
an $L_0^I$-complete module over $\Q[c]$, and the splicing map is a map
$V\lra \cE^{-1}\prod_iN_i$ of rational vector spaces. 

This packages the data differently to  the various versions of the
algebraic model appearing in the literature
(see~\cite{greenleesrationals1,shipleys1,BGKSs1, adelic1}). The
models, each with their own merits, are
all broadly similar, but there are
two ways in which they differ: first in the shape of the
diagram, and second in whether the local factors are described as
localizations of other categories or as module categories. 

In our case, the image of the unit in \eqref{eq:t-model} is the diagram
\[
\xymatrix{
& \prod_i \Q[c_i]  \ar[d] \\
(\Q \to \cE^{-1} \prod_i \Q[c]) \ar[r] & \cE^{-1} \prod_i \Q[c]
}
\]
which may be viewed as a little more complicated than the usual
cospan. We can then recover the Tate cospan
\[
\xymatrix{
& \prod_i \Q[c_i]  \ar[d] \\
\Q \ar[r] & \cE^{-1} \prod_i \Q[c] \rlap{.}
}
\]
Most of the algebraic models appearing in the literature are usually based on this
cospan, so that for example $\Sp_{\T}$ is symmetrically monoidally equivalent to the limit of the diagram:
\[
\xymatrix{
& \mathsf{D}(\prod_i \Q[c_i])  \ar[d] \\
\mathsf{D}(\Q) \ar[r] & \mathsf{D}(\cE^{-1} \prod_i \Q[c]),
}
\]
where all functors involved are given by extensions of scalars. A direct comparison between this model and the model coming from the categorical local-to-global is provided in \cite{adelic1}.
\end{remark}

\subsubsection{$O(2)$}\label{ex:o2} The next case is the group $G=O(2)$. Algebraic models for this category have been considered previously in the work of the third author and Barnes~\cite{barneso2,greenleeso2}.  Once again these models differ in style as in \cref{rem:model_is_different}.

Following the discussion of \cref{ex:dO2topologies}, we see that we have a partition of the prism into two disjoint clopen pieces: 
\[
\Prism(\Sp_{O(2)}^\omega)\cong (\Sub (SO(2)), \cotoral )\sqcup
(\mcD , =),
\]
where $\mcD$ consists of the conjugacy classes of dihedral groups
$D_{2j}$ of order $2j$ for $j\geqslant 1$, a single conjugacy class of each, together with the group $O(2)$ itself. 

Using the Thomason dispersion we find
\begin{itemize}
    \item $\Prism(\Sp_{O(2)}^\omega)_0  = \{C_i\}_{i \geqslant 1}\sqcup
      \{D_{2j}\}_{j\geqslant 1}$,
    \item $\Prism(\Sp_{O(2)}^\omega)_1  = \{SO(2), O(2)\}$,
\end{itemize}  
giving rise to the strata
\begin{align*}
(\Sp_{O(2)})_0  &= \prod_{i \geqslant 1} \Lambda_{C_i} (\Sp_{O(2)}) \times \prod_{j \geqslant 1} \Lambda_{D_{2j}} (\Sp_{O(2)}) \\ 
(\Sp_{O(2)})_1 &= \Lambda_{SO(2)} (\Sp_{O(2)}) 
\times\Lambda_{O(2)} (\Sp_{O(2)}).
\end{align*}
The splicing data for the clopen set consisting of subgroups of $SO(2)$ is already understood. As such, the cyclic part of the model is given by the homotopy limit of the following cospan
\begin{equation}\label{eq:t-model_o2}\begin{gathered}
\xymatrix{
& \prod_{i \geqslant 1} \Lambda_{C_i} (\Sp_{O(2)}) \ar[d]\\
\Fun([1], \Lambda_{SO(2)} (\Sp_{O(2)})) \ar[r]_-{\pi_1} & \Lambda_{SO(2)} (\Sp_{O(2)}),
}
\end{gathered}
\end{equation}
where the vertical map once again is given by $\tilde{E} \mathcal{F}
\wedge -$ after forming the product.

The splicing data for the dihedral part is given as in
\cref{ex:derivedabsflatring}. That is, the category is  the homotopy limit of the following cospan
\begin{equation}\label{eq:D-model_o2}\begin{gathered}
\xymatrix{
& \prod_{j \geqslant 1} \Lambda_{D_{2j}} (\Sp_{O(2)})  \ar[d]\\
\Fun([1], \Lambda_{O(2)} (\Sp_{O(2)})) \ar[r]_-{\pi_1} & \Lambda_{O(2)} (\Sp_{O(2)}),
}
\end{gathered}
\end{equation}
where the vertical map takes a collection $\{W_j\}_j \in \prod_{j \geqslant 1} \Lambda_{D_{2j}} (\Sp_{O(2)})$ of objects and forms the end space $\colim_n \prod_{j \geqslant n} W_j$.

All in all, we see that an object of the model consists of:
\begin{enumerate}
    \item an object $V \in \Lambda_{SO(2)} \Sp_{O(2)}$;
    \item a collection of objects $\{N_i\} \in \prod_{i\geqslant 1} \Lambda_{C_i} \Sp_{O(2)}$;
    \item a morphism  $V \to \tilde{E} \mathcal{F} \wedge \prod_i N_i$ in $\Lambda_{SO(2)} \Sp_{O(2)}$;
    \item an object $W_{\infty} \in \Lambda_{O(2)} \Sp_{O(2)}$;
    \item a collection  objects $\{W_j\} \in \prod_{j \geqslant 1} \Lambda_{D_{2i}} \Sp_{O(2)}$;
    \item a morphism $W_{\infty} \to \colim_n \prod_{j \geqslant n} 
W_j$ in $\Lambda_{O(2)} \Sp_{O(2)}$.
\end{enumerate}

\begin{remark}
Applying \cref{cor:alg_models_for_lambda} we see that the local
factors differ in detail from those in \cref{subsec:circle}. For the component corresponding to $SO(2)$
the Weyl groups of each  subgroup of $SO(2)$ is a group $W$ of order
2. Thus the local factor at $SO(2)$ is $ \DdgMod[]{\Q[W]}$, and the local
factor at $C_i$ is $\DdgMod[L_0^I \text{-} \comp]{\Q[c][W]}$, with $W$ acting
to negate $c$. 

Similarly, the Weyl group of $O(2)$ is trivial, and the Weyl group of
each finite dihedral group is again group of order 2, so the local
factor at $O(2)$ is $\sfD (\Q)$ and the local factor at $D_{2j}$ is
$ \DdgMod[]{\Q[W]}$. The structure map is required to commute with the group
action, so maps into the fixed part of the end space. 
\end{remark}

\subsubsection{$SO(3)$}
Our semisimple example is  $G = SO(3)$. Models for this group have
previously been considered in work of the third author
\cite{greenleesso3} and K\c{e}dziorek
\cite{kedziorekso3}. The  three differences between $O(2)$ and the
present example are 
\begin{itemize}
    \item[(a)] there are new subgroups,
    \item[(b)] there is new fusion of conjugacy classes, and
    \item[(c)] there are some changes in Weyl
groups.
\end{itemize}
These are relatively minor since
\begin{itemize}
    \item[(a)] the new subgroups are isolated and  of finite index in their normalizers,
    \item[(b)] the only fusion is that the conjugacy classes of $D_2$ and $C_2$, distinct in $O(2) $ become fused in $SO(3)$, and
    \item[(c)]  the only significant change in Weyl
groups is at the identity.
\end{itemize}
In fact the situation is easy to understand
since the prism is the disjoint union of a number of clopen pieces:
    \[
    \Prism(\Sp^{\omega}_{SO(3)})\cong (\Sub (SO(2)), \cotoral )\sqcup (\mcD' , =)\sqcup (E, =),
    \]
where $\mcD'=\{D_{2j} | j\geqslant 3\}$ along with $O(2)$ and $E=\{
SO(3), (A_5), (\Sigma_4), (A_4), (D_4)\}$. The subgroup $D_2$ is now conjugate to $C_2$ and hence
cotoral in $SO(2)$. The subgroup $D_4$ has been moved to the exceptional
category since its Weyl group is the nonabelian group of order 6. We
will see that the model splits into 7 pieces: the cyclic piece, the
dihedral piece and 5 exceptional pieces. 

The Thomason dispersion then gives
\begin{itemize}
    \item $\Prism(\Sp^{\omega}_{SO(3)})_0  = \{C_i\}_{i \geqslant
        1}\sqcup \{(D_{2i})\}_{i\geqslant 3}\sqcup \{ SO(3),  (A_5), (A_4), (\Sigma_4), (D_4)\}$;
    \item $\Prism(\Sp^{\omega}_{SO(3)})_1  = \{SO(2), O(2)\} $.
\end{itemize}  
We highlight that the group itself, $SO(3)$, appears in the height 0
part of the dispersion instead of the height 1 part by \cref{thm:heightformula}, since its centre
is zero dimensional. This does not occur for the abelian rank 1 groups.

The zeroth stratum is given by
\[
\prod_{i \geqslant 1} \Lambda_{C_i} (\Sp_{SO(3)}) \times
\prod_{j \geqslant 3} \Lambda_{D_{2j}} (\Sp_{SO(3)}) \times \prod_{(H)\in E}\Lambda_H(\Sp_{SO(3)})
\]
while the first stratum is  simpler: 
\[
\Lambda_{SO(2)} (\Sp_{SO(3)}) 
\times\Lambda_{O(2)} (\Sp_{SO(3)}).
\]
The splicing data is very similar to that for $O(2)$, because  the
exceptional subgroups are not involved in any gluing.

\begin{remark}
Applying \cref{cor:alg_models_for_lambda} we see that the local
factors differ in detail from those in \cref{ex:o2}. The change in
Weyl groups of $D_4$ has been remarked on already, but the main point
is that the Weyl group of the trivial subgroup is $SO(3)$ itself, so
the local factor at 1 is now $\DdgMod[L_0^I \text{-} \comp]{H^*(BSO(3))}$
rather than $\DdgMod[L_0^I \text{-} \comp]{H^*(BSO(2))[W]}$ as it was for $O(2)$. 
\end{remark}

\subsubsection{Tori}\label{ex:torus}

For far we have only considered compact Lie groups whose dispersion is
of height 1. In this section we will consider $G= \T^2$ which by
\cref{thm:heightformula} has a 2-dimensional dispersion.  Again, we note that algebraic models have been given by the third author in joint work with Shipley \cite{greenleestorus,GStorus}; these models
group together subgroups with the same
identity component, and an equivalence with a model treating subgroups individually was given in \cite{greenlees_tnq3}.

The three strata are formed of the 0, 1 and 2 dimensional closed subgroups. Writing $H$ for an arbitrary subgroup of dimension 1, and $K$ for an arbitrary zero-dimensional subgroup we have:
\begin{align*}
 (\Sp_{\bT^2})_0  &= \prod_{K} \Lambda_{K} (\Sp_{\bT^2})\\
 (\Sp_{\bT^2})_1  &= \prod_{H} \Lambda_{H} (\Sp_{\bT^2})\\
 (\Sp_{\bT^2})_2 &= \Lambda_{\bT^2} (\Sp_{\bT^2}).
\end{align*}
Therefore, using \cref{ex:2dimex} as a guide, we see that the categorical local-to-global principle for $G =\T^2$ provides a  model as the homotopy limit of the following punctured cube:
    \[
    \xymatrix@C=-2em{
    && \Fun([1] {\times} [1] ,\Lambda_{\T^2} (\Sp_{\T^2})) \ar[dr] \ar[dd]|\hole & \\
    &  \prod_{K} \Lambda_K (\Sp_{\T^2}) \ar[rr] \ar[dd] && \Fun([1], \Lambda_{\T^2} (\Sp_{\T^2})) \ar[dd] \\
    \prod_{H} \Fun([1],\Lambda_H (\Sp_{\T^2}))\ar[rr]|<<<<<\hole \ar[dr]&&  \Fun([1]  , \Lambda_{\T^2} (\Sp_{\T^2})) \ar[dr] & \\
    &\prod_{H} \Lambda_H (\Sp_{\T^2}) \ar[rr]&& \Lambda_{\T^2} (\Sp_{\T^2})
    }
    \]
where the splicing data is described in \cref{lemma:iteraterll}.

\addtocontents{toc}{\vspace{5mm}}

\biblio
\bibliography{bibliography}\bibliographystyle{alpha}

\end{document}